\documentclass[12pt]{amsart}

\usepackage{hyperref}
\usepackage{geometry}
\usepackage{amsmath}
\usepackage{amsthm}
\usepackage{amssymb}
\usepackage{comment}
\usepackage{color}
\usepackage{mathrsfs}
\usepackage{float}
\usepackage{verbatim}
\usepackage{adjustbox}
\usepackage{fancyvrb}
\usepackage{cite}

\usepackage[usenames,dvipsnames]{xcolor}
\usepackage{tikz}

\newcommand{\Z}{\mathbb{Z}}
\newcommand{\Q}{\mathbb{Q}}

\newcommand{\torsion}{\textrm{torsion}}
\newcommand{\cm}{cohomologically minimal}
\newcommand{\Grid}{Grid }
\newcommand{\String}{String }
\newcommand{\Hex}{Hex }
\newcommand{\RAAG}{RAAG}
\newcommand{\RAAGs}{RAAGs}

\makeatletter
\newtheorem*{rep@theorem}{\rep@title}
\newcommand{\newreptheorem}[2]{
\newenvironment{rep#1}[1]{
 \def\rep@title{#2 \ref{##1}}
 \begin{rep@theorem}}
 {\end{rep@theorem}}}
\makeatother

\newtheorem{theorem}{Theorem}[section]
\newtheorem{corollary}[theorem]{Corollary} 
\newtheorem{lemma}[theorem]{Lemma} 
\newtheorem{conj}[theorem]{Conjecture} 
\newreptheorem{conj}{Conjecture}
\newtheorem{prop}[theorem]{Proposition} 	

\newtheorem{defn}[theorem]{Definition}

\theoremstyle{definition}
\theoremstyle{remark}
\newtheorem*{remark}{Remark}
\newtheorem*{example}{Example}
\newtheorem{numbered example}[theorem]{Example}

\title{The minimum $b_2$ problem for right-angled Artin groups}
\author{Alyson Hildum}
\address{Alyson Hildum \\ Dept. of Mathematics \& Statistics \\ McMaster University \\ 1280 Main Street West \\ Hamilton \\ Ontario, Canada L8S 4K1}

\begin{document}

\usetikzlibrary{patterns}
\usetikzlibrary{backgrounds}
\pgfdeclarelayer{background}
\pgfdeclarelayer{foreground}
\pgfsetlayers{background,main}

%%%%%%%%%%%%%%%%%%%%%%%%%%%%%%%%%%%%%%%%%%%%%%%%%%%%%%%%%%%%%%%%%%%%%%%%%%%%%%%%%%%%%%%%%%%%%%%%%%%%%%%%%%%%%%%%%%%%
%%%%%%%%%%%%%%%%%%%%%%%%%%%%%%%%%%%%%%%%%%%% ABSTRACT %%%%%%%%%%%%%%%%%%%%%%%%%%%%%%%%%%%%%%%%%%%%%%%%%%%%%%%%%%%%%%
%%%%%%%%%%%%%%%%%%%%%%%%%%%%%%%%%%%%%%%%%%%%%%%%%%%%%%%%%%%%%%%%%%%%%%%%%%%%%%%%%%%%%%%%%%%%%%%%%%%%%%%%%%%%%%%%%%%%
\begin{abstract}
This paper focuses on tools for constructing 4-manifolds that have fundamental group $G$ isomorphic to a right-angled Artin group and that are also minimal, in the sense that they minimize $b_2(M)$, the dimension of $H_2(M;\Q)$. For a finitely presented group $G$, define $h(G) = \min\{ b_2(M) | M \in \mathcal M(G) \}$.  

In this paper, we explore the ways in which we can bound $h(G)$ from below using group cohomology and the tools necessary to build 4-manifolds that realize these lower bounds. We give solutions for right-angled Artin groups, or \RAAGs, when the graph associated to $G$ has no 4-cliques, and further we reduce this problem to the case when the graph is connected and contains only 4-cliques. We then give solutions for many infinite families of \RAAGs\ and provide a conjecture to the solution for all \RAAGs. 
\end{abstract}

\maketitle

%%%%%%%%%%%%%%%%%%%%%%%%%%%%%%%%%%%%%%%%%%%%%%%%%%%%%%%%%%%%%%%%%%%%%%%%%%%%%%%%%%%%%%%%%%%%%%%%%%%%%%%%%%%%%%%%%%%%
%%%%%%%%%%%%%%%%%%%%%%%%%%%%%%%%%%%%%%%%%%%% INTRODUCTION %%%%%%%%%%%%%%%%%%%%%%%%%%%%%%%%%%%%%%%%%%%%%%%%%%%%%%%%%%
%%%%%%%%%%%%%%%%%%%%%%%%%%%%%%%%%%%%%%%%%%%%%%%%%%%%%%%%%%%%%%%%%%%%%%%%%%%%%%%%%%%%%%%%%%%%%%%%%%%%%%%%%%%%%%%%%%%%
\section{Introduction}

It is well known that for any finitely presented group $G$ there is a closed, orientable 4-dimensional manifold with fundamental group isomorphic to $G$. This paper explores the problem of constructing a 4-manifold $M$ with particular fundamental group that minimizes $b_2(M)$, the dimension of $H_2(M;\Q)$. We will refer to this as the \emph{minimum $b_2$ problem}. Many have researched this topic, including Hausmann and Weinberger \cite{HausmannWeinberger85}, Baldridge and Kirk \cite{BaldridgeKirk07, BaldridgeKirk09}, Eckmann \cite{Eckmann97}, Johnson and Kotschick \cite{JohnsonKotschick93} and independently Kotschick \cite{Kotschick94, Kotschick06}, Lu\"ck \cite{Luck94}, and most recently Kirk and Livingston \cite{KirkLivingston05}. However, the minimum $b_2$ problem remains open for all but a few classes of groups. 

Let $\mathcal M (G)$ denote the class of closed, oriented topological 4-manifolds with fundamental group isomorphic to a fixed group $G$. For a finitely presented group $G$, define $h(G) = \min\{ b_2(M) | M \in \mathcal M(G) \}$. Calculations of $h$ are known for free groups and free abelian groups, but little more. The underlying goal of the research represented in this paper is to generalize these calculations to right-angled Artin groups, of which free and free abelian groups are special cases. In particular, a right-angled Artin group (abbreviated \RAAG) has a presentation with a finite generating set where the relations consist solely of commutators between generators. RAAGs are also known as graph groups due to the fact that their presentations can uniquely be represented by graphs, where each vertex represents a generator and each edge between vertices represents a commutator relation between those generators. Hence, $F_n$ is associated to a graph with $n$ vertices with no edges and $\Z^n$ is associated to a complete graph with $n$ vertices.

We begin by exploring the minimum $b_2$ problem for arbitrary finitely presented groups, and show how the group cohomology plays an important role in bounding $h$ from below. Specifically, we prove the following useful proposition that holds for finitely presented groups.

\begin{prop} \label{hnewlowerbound}
For a finitely presented group $G$, 
\begin{equation*} 2b_2(G) - m_2(G) \leq h(G), 
\end{equation*}
where $m_2(G)$ is the maximum rank of the symmetric bilinear form
\begin{equation}
H^2(G;\Z_2) \times H^2(G;\Z_2) \to \Z_2, (a,b) \mapsto (a\cup b) \cap \alpha
\label{formmod2}
\end{equation}
taken over all choices of $\alpha \in H_4(G;\Z_2)$. 
\end{prop}

This proposition yields our first theorem for \RAAGs:
\begin{theorem} \label{trivialH4}
If a \RAAG\ $G$ has trivial $H^4(G)$, then $h(G)=2b_2(G)$. 
\end{theorem}
This result holds for all \RAAGs\ with associated graphs of dimension 3 (graphs with no 4-cliques). For \RAAGs\ with associated graphs of higher dimension, the calculation of $h$ depends on the structure of the graph.  

In Section \ref{tools} we discuss techniques for constructing manifolds that minimize $b_2$. In Section \ref{infinitefamilies} we prove that the lower bound given in the proposition above is an equality for several infinite families of \RAAGs. These results provide evidence for the following conjecture:

\begin{conj} \label{Conjecture}
If $G$ is a \RAAG, $h(G) = 2b_2(G)-m_2(G)$.
\end{conj}

We also prove three inductive theorems which together reduce the minimum $b_2$ problem to one in which the associated graphs are connected and contain only 4-cliques:

\begin{theorem} 
\label{freeproducts}
Let $G_1$ and $G_2$ be \RAAGs\ such that $h(G_i)=2b_2(G_i)-m_2(G_i)$ for $i=1,2$. Then 
$h(G_1*G_2) = h(G_1)+h(G_2)$.
\end{theorem}

\begin{theorem}
\label{wedging}
Let $\Gamma_1$ and $\Gamma_2$ be two graphs representing \RAAGs\ $G_1$ and $G_2$ such that $h(G_i)=2b_2(G_i)-m_2(G_i)$ for $i=1,2$.  Let $\{s_1,\dotsc,s_m\}$ and $\{t_1,\dotsc,t_m\}$ be two sets of pairwise non-adjacent vertices in $\Gamma_1$ and $\Gamma_2$, respectively. Suppose a new graph, $\Gamma$ is created by identifying $s_i$ with $t_i$, $i=1,\dotsc,m$. Then for the \RAAG\ $G$ represented by $\Gamma$, $h(G)=h(G_1)+h(G_2)$.
\end{theorem}

\begin{theorem}
\label{breakdown} 
Let $\Gamma$ be a graph associated to a \RAAG\ $G$. Let $r$ be the number of edges in $\Gamma$ that are not part of a 4-clique. Suppose the $r$ edges are deleted from $\Gamma$ resulting in $k$ disjoint subgraphs $\Gamma_1, \dotsc, \Gamma_k$. By construction, all the edges in the $\Gamma_i$ are necessarily part of at least one 4-clique. Let $G_i$ be the group associated to $\Gamma_i$. If $h(G_i)=2b_2(G_i)-m_2(G_i)$ for each $i$, then $h(G) = \sum_i h(G_i) + 2r$. 
\end{theorem}

The geography of 4-cliques in a graph is key to understanding the minimum $b_2$ problem for the associated \RAAG. Poincar\'e duality imposes restrictions between group theory and topology, which is strengthened in dimension 4. We observe this restriction for general finitely presented groups in Proposition \ref{hnewlowerbound}, in which we see that a portion of the intersection form of a 4-manifold contains the structure of the 2-dimensional cohomology of the fundamental group $G$. For a \RAAG, that cup product structure is completely understood in terms of the configuration of the 4-cliques in the associated graph.

%%%%%%%%%%%%%%%%%%%%%%%%%%%%%%%%%%%%%%%%%%%%%%%%%%%%%%%%%%%%%%%%%%%%%%%%%%%%%%%%%%%%%%%%%%%%%%%%%%%%%%%%%%%%%%%%%%%%
%%%%%%%%%%%%%%%%%%%%%%%%%%%%%%%%%%%%%%%%%%% ACKNOWLEDGEMENT %%%%%%%%%%%%%%%%%%%%%%%%%%%%%%%%%%%%%%%%%%%%%%%%%%%%%%%%
%%%%%%%%%%%%%%%%%%%%%%%%%%%%%%%%%%%%%%%%%%%%%%%%%%%%%%%%%%%%%%%%%%%%%%%%%%%%%%%%%%%%%%%%%%%%%%%%%%%%%%%%%%%%%%%%%%%%

\subsection*{Acknowledgement}
I would like to thank Danny Ruberman, my Ph.D. advisor, for the excellent help and guidance he has given me for this problem. The results in this paper are based on my dissertation research at Brandeis University. 

%%%%%%%%%%%%%%%%%%%%%%%%%%%%%%%%%%%%%%%%%%%%%%%%%%%%%%%%%%%%%%%%%%%%%%%%%%%%%%%%%%%%%%%%%%%%%%%%%%%%%%%%%%%%%%%%%%%%
%%%%%%%%%%%%%%%%%%%%%%%%%%%%%%%%% THE HAUSMANN WEINBERGER INVARIANT %%%%%%%%%%%%%%%%%%%%%%%%%%%%%%%%%%%%%%%%%%%%%%%%
%%%%%%%%%%%%%%%%%%%%%%%%%%%%%%%%%%%%%%%%%%%%%%%%%%%%%%%%%%%%%%%%%%%%%%%%%%%%%%%%%%%%%%%%%%%%%%%%%%%%%%%%%%%%%%%%%%%%
\section{The Hausmann-Weinberger invariant} \label{HW invariant}

%%%%%%%%%%%%%%%%%%%%%%%%%%%%%%%%%%%%%%%%%%%%%%%%%%%%%%%%%%%%%%%%%%%%%%%%%%%%%%%%%%%%%%%%%%%%%%%%%%%%%%%%%%%%%%%%%%%%
\subsection{Basic definitions}
In 1985, Hausmann and Weinberger defined the invariant $q(G)$ as the minimum Euler characteristic over all topological $M$ with fundamental group $G$. Advances have been made in studying $q$ for classes of groups including knot groups \cite{HausmannWeinberger85}, fundamental groups of aspherical manifolds \cite{KirkLivingston05, Kotschick94}, free groups, fundamental groups of closed oriented genus $g$ surfaces and 3-manifold groups \cite{Kotschick94}, and most recently finitely generated abelian and free abelian groups \cite{KirkLivingston05}. For the cases of infinite amenable groups \cite{Eckmann97} and groups with finite abelianization \cite{Luck94}, $L^2$-methods have been used to bound $q$ below by zero. 

For a 4-manifold $M$, the Euler characteristic $\chi(M)$ is given by the alternating sum of the ranks of homology (with rational coefficients). These ranks are commonly refered to as Betti numbers; we will denote the $i$th Betti number by $b_i(M)= \dim H_i(M;\Q)$. By Poincar\'e duality, $\chi(M) = 2 - 2b_1(M) + b_2(M)$. 

For a group $G$ we can similarly define $b_i(G) = \dim H_i(K(G,1); \Q)$, where $K(G,1)$ is an Eilenberg-Maclane space. If $G$ is a finitely presented group with a presentation $\mathscr P$ having $g$ generators and $r$ relations, define the \emph{deficiency} $d(\mathscr P)=g-r$. Then the \emph{deficiency} $d_G$ of $G$ is the maximum $d(\mathscr P)$ over all finite presentations $\mathscr P$ \cite{HausmannWeinberger85}.

A priori, we see that $q(G)$ takes integer values. We have lower and upper bounds on $q(G)$ which allow us to consider $q$ as the minimum rather than the infimum over all $\chi(M)$. 
\begin{theorem}[Hausmann-Weinberger, {\cite[Theorem 1]{HausmannWeinberger85}}]
\label{qbounds}
For a finitely presented group $G$, we have 
\begin{equation*}
2-2b_1(G) + b_2(G) \leq q(G) \leq 2(1-d_G).
\end{equation*}
\end{theorem}
\begin{proof}
Let $G$ be a finitely presented group with $g$ generators and $r$ relations such that $d_G= g-r$. Let $M \in \mathcal M(G)$ and $f: M \to K(G,1)$ be a map inducing an isomorphism on fundamental groups. The induced map on homology $f_*:H_i(M) \to H_i(G)$ is an isomorphism for $i=1$ and a surjection for $i=2$. The surjection in dimension 2 can be seen by considering the Hopf exact sequence, $\pi_2(M) \to H_2(M) \to H_2(\pi_1(M)) \to 0$. Thus $b_1(M) = b_1(G)$ and $b_2(M) \geq b_2(G)$. To see the upper bound, consider the following construction of a 4-manifold in $\mathcal M$: Build a handlebody $X$ consisting of one 0-handle, $g$ 1-handles, and $r$ 2-handles (attached to reflect each of the relations), and double it. The result is a closed orientable 4-manifold $M$ with $\pi_1(M) \cong G$ and $\chi(M) = 2-2g+2r = 2(1-d_G)$. 
\end{proof}

Since $b_1M = b_1G$, determining $q(G)$ comes down to refining the bounds on possible values of $b_2M$. Kirk and Livingston investigated the $q$ invariant for finitely generated abelian and free abelian groups in \cite{KirkLivingston05} and introduced an invariant equivalent to $q$:

\begin{defn}[{Kirk-Livingston, \cite{KirkLivingston05}}] 
For a finitely presented group $G$, define 
\begin{equation*}
h(G) = \min\{ b_2(M) | M \in \mathcal M(G) \}. 
\end{equation*}
\end{defn}
As mentioned in the introduction, we will refer to the problem of determining $h(G)$ for a group $G$ as the \emph{minimum $b_2$ problem for $G$}. By definition $q(G) = 2-2b_1G+h(G)$, so solving the minimum $b_2$ problem for $G$ is equivalent to finding $q(G)$. The following corollary then follows from Theorem \ref{qbounds}:
\begin{corollary}
\label{hbounds}
For a finitely presented group $G$ with $r$ relations, 
\begin{equation*}
b_2(G) \leq h(G) \leq 2r. 
\end{equation*}
\end{corollary} 

The basic technique to solving the minimum $b_2$ problem is to increase the lower bound on $h(G)$, if possible, and then construct a suitable 4-manifold $M$ with $b_2(M)$ equal to the lower bound, thus yielding an equality. We call such a 4-manifold $M \in \mathcal M(G)$ with $b_2(M) = h(G)$ a \emph{realizing manifold for $h(G)$}.

\begin{example}
For a free group $F_n$, $h(F_n) = 0$: 
Let $M$ be an arbitrary 4-manifold in $\mathcal M(F_n)$. We know from Theorem \ref{qbounds} that $b_2(F_n)\leq b_2(M)$. A bouquet of $n$ circles is a $K(F_n,1)$ complex in which $b_2(F_n)=0$. Thus, $h(F_n)$ is bounded below by zero. One 4-manifold realizing this lower bound is the connected sum of $n$ copies of $S^1 \times S^3$. Since $\pi_1(\#n S^1 \times S^3) \cong F_n$ and $b_2(\#n S^1 \times S^3)=0$, $h(F_n)=0$. 
\end{example}

\begin{example} \label{freeabelian}
The solution for free abelian groups, a special case of \RAAGs, is given in the theorem below: \end{example}
\begin{theorem}[{Kirk-Livingston, \cite[Theorem 1]{KirkLivingston05}}]
For a free abelian group $\Z^n$, $h(\Z^n) = \binom{n}{2} + \epsilon_n$ for all $n$, with the exception of $h(\Z^3)=6$ and $h(\Z^5)=14$ .
Here $\epsilon_n$ is an auxiliary function defined to be 0 if $\binom{n}{2}$ is even and 1 otherwise. 
\end{theorem}
\noindent When $b_2(\Z^n)=\binom{n}{2}$ is odd, the lower bound on $h$ is increased by 1. This argument is explained later by Proposition \ref{evenform}. The full details of the proof, including the 4-manifold constructions, can be found in \cite{KirkLivingston05}. \\

In the free abelian case, the constructions for realizing manifolds are built from products of surfaces that are surgered to identify generators or kill commutators. We shall see that manifolds realizing general \RAAGs\ can be constructed in a similar way. 

%%%%%%%%%%%%%%%%%%%%%%%%%%%%%%%%%%%%%%%%%%%%%%%%%%%%%%%%%%%%%%%%%%%%%%%%%%%%%%%%%%%%%%%%%%%%%%%%%%%%%%%%%%%%%%%%%%%%%%%%%%%%%%%%%%%%%%%%%% THE COHOMOLOGICAL OBSTRUCTION TO SOLVING THE MINIMUM $b_2$ PROBLEM %%%%%%%%%%%%%%%%%%%%%%%%%%
%%%%%%%%%%%%%%%%%%%%%%%%%%%%%%%%%%%%%%%%%%%%%%%%%%%%%%%%%%%%%%%%%%%%%%%%%%%%%%%%%%%%%%%%%%%%%%%%%%%%%%%%%%%%%%%%%%%%
\section[Cohomological obstruction]{The cohomological obstruction to solving the minimum $b_2$ problem} \label{CohomologicalObstruction}

%%%%%%%%%%%%%%%%%%%%%%%%%%%%%%%%%%%%%%%%%%%%%%%%%%%%%%%%%%%%%%%%%%%%%%%%%%%%%%%%%%%%%%%%%%%%%%%%%%%%%%%%%%%%%%%%%%%%
\subsection{Finding a better lower bound for $h$}\label{better lower bound}
Theorem \ref{qbounds} asserts that for any finitely presented group $G$, $b_2(G)\leq h(G)$. We will refer to $b_2(G)$ as the \emph{trivial lower bound} on $h(G)$. In many cases we can use the cohomological structure of $G$ to yield a better lower bound for $h(G)$. 

Let $f: M\to K(G,1)$ be a map that induces an isomorphism on fundamental groups, and let $f^*: H^i(G) \to H^i(M)$ be the induced map on cohomology. In the proof of Theorem \ref{qbounds} it is shown that the induced homological map $f_*:H_i(M) \to H_i(G)$ is an isomorphism for $i=1$ and a surjection for $i=2$. By the Universal Coefficient Theorem, $f^*:H^i(G) \to H^i(M)$ is an isomorphism for $i=1$ and an injection for $i=2$.  Denote by $I(M,f)$ the image $f^*(H^2(G))$ in $H^2(M)$ modulo torsion.

Consider the symmetric, bilinear pairing
\begin{equation} \label{form}
H^2(G) \times H^2(G) \to \Z \textrm{ by  } (a,b)\mapsto  (a\cup b) \cap \alpha
\end{equation}
for a homology class $\alpha\in H_4(G)$. If $\alpha = f_*([M])$, this form completely determines the restriction of the intersection form of $M$,
\begin{equation*}
H^2(M)/\torsion \times H^2(M)/\torsion \to \Z \textrm{ by } (x,y) \mapsto (x\cup y) \cap [M],
\end{equation*}
to $I(M,f)$ since $(f^*(a) \cup f^*(b)) \cap [M] = (a \cup b) \cap \alpha$. 

Given any group $G$ and homology class $\alpha \in H_4(G)$, there exists $M\in \mathcal M(G)$ and a continuous map $f:M\to K(G,1)$ so that $f_*([M])=\alpha$ \cite{BaldridgeKirk09}. Additionally, the rank of $I(M,f)$ is $b_2(G)$. These two observations allow us to make certain assumptions about the possible values of $h(G)$ independent of the 4-manifold $M$ or the classifying map $f:M\to K(G,1)$.

We introduce the following definition which is useful for improving the trivial lower bound on $h(G)$ for any finitely presented group $G$.

\begin{defn}
For a finitely presented group $G$, define $m(G)$ to be the maximum rank of a matrix associated to (\ref{form}) over all possible choices of $\alpha \in H_4(G)$. 
\end{defn}

Note that a priori, $0 \leq m(G) \leq b_2(G)$. If $m(G)$ is strictly less than $b_2(G)$, then $I(M,f)$ is represented by a singular matrix, which indicates the lower bound on $h(G)$ must be greater than $b_2(G)$, the dimension of $I(M,f)$. Unfortunately, computing $m(G)$ is impractical; in all nontrivial cases, there are infinitely many choices of $\alpha \in H_4(G;\Z)$. However,  $H_4(G;\Z_p)$ can be finite. If $p$ is prime, the intersection form of a 4-manifold $M$ with $\Z_p$ coefficients is also nonsingular. Thus we can calculate $m_p(G)$ instead, a mod $p$ version of $m(G)$. 

\begin{defn} \label{m2}
Define $m_p(G)$ to be the maximum rank of the symmetric bilinear form 
\begin{equation}
H^2(G;\Z_p) \times H^2(G;\Z_p) \to \Z_p, (a,b) \mapsto (a\cup b) \cap \alpha
\label{formmodp}
\end{equation}
over all possibe choices of $\alpha \in H_4(G;\Z_p)$. 
\end{defn}

In practice, for \RAAGs\ we need only consider the bilinear form on $H^2(G;\Z_p)$ for $p=2$; we only use $m_2(G)$, the invariant mentioned in the introduction. We now prove Proposition \ref{hnewlowerbound} (which holds for all prime $p$ although it is stated in the introduction for $p=2$). 
 
\begin{proof}[Proof of Proposition \ref{hnewlowerbound}]
Let $G$ be a finitely presented group, and let $X$ be a $K(G,1)$ space. Then $H_1(X)$ and $H_2(X)$ are finitely presented and $b_2(G) = \dim H_2(X;\Q) = \dim H^2(X;\Q)$, as we identify $H_*(G)$ with $H_*(X)$ and $H^*(G)$ with $H^*(X)$. Let $\tilde\alpha$ be the homology class that maximizes the rank of the form (\ref{formmodp}) over all $\alpha\in H_4(G; \Z_p)$. Consequently, $\tilde\alpha$ minimizes the radical of (\ref{formmodp}). Recall that for a symmetric bilinear form, the radical contains linear independent vectors $x_i$ such that $\langle x_i \cup y,\alpha\rangle=0$ for all $y\in H^2(G;\Z_p)$ and a choice of $\alpha \in H_4(G;\Z_p)$. Since the dimension of the form is $b_2(G)$, the minimum dimension of the radical is $b_2(G)-m_p(G)$ by the Rank-Nullity Theorem. In order for the intersection form on a manifold $M\in \mathcal M(G)$ to be nondegenerate, its rank must be at least $b_2(G) + (b_2(G) - m_p(G))$. Thus $2b_2(G) - m_p(G) \leq h(G)$. 
\end{proof}

%%%%%%%%%%%%%%%%%%%%%%%%%%%%%%%%%%%%%%%%%%%%%%%%%%%%%%%%%%%%%%%%%%%%%%%%%%%%%%%%%%%%%%%%%%%%%%%%%%%%%%%%%%%%%%%%%%%%
%%%%%%%%%%%%%%%%%%%%%%%%%%%%%%%%%% RIGHT-ANGLED ARTIN GROUPS %%%%%%%%%%%%%%%%%%%%%%%%%%%%%%%%%%%%%%%%%%%%%%%%%%%%%%%
%%%%%%%%%%%%%%%%%%%%%%%%%%%%%%%%%%%%%%%%%%%%%%%%%%%%%%%%%%%%%%%%%%%%%%%%%%%%%%%%%%%%%%%%%%%%%%%%%%%%%%%%%%%%%%%%%%%%
\section{Right-angled Artin groups} \label{raags}

We now restrict our discussion of the minimum $b_2$ problem to \RAAGs. The common construction of a \emph{Salvetti complex} is a compact $K(G,1)$ space and is used in the computation of the group cohomology of \RAAGs\ in \cite{CharneyDavis95}.

\begin{theorem}[{Charney-Davis, \cite[Theorem 3.2.4]{CharneyDavis95}}]
Suppose that $G$ is a \RAAG\ with generators $s_1,\dotsc, s_n$. Let $\Lambda[y_1,\dotsc,y_n]$ be the exterior algebra over $\Z$ on the variables $y_1,\dotsc,y_n$. Let $I$ be the ideal generated by all products $y_iy_j$ such that $s_i$ and $s_j$ do not commute in $G$. Then $H^*(G) \cong \Lambda[y_1,\dotsc,y_n]/I$. 
\end{theorem}

Nontrivial cup products in the cohomology ring of a \RAAG\ come from the commuting generators, as in the case of a torus. This is because the chain complex of a Salvetti complex injects into that of a torus, where all chain (and cochain) maps are trivial. 

For a graph $\Gamma$ associated to a \RAAG\ $G$, we can recognize generators of $H^*(G)$ straight from the graph $\Gamma$: vertices represent generators of $H^1(G)$, edges represent generators of $H^2(G)$, and triangles represent generators of $H^3(G)$. In general, $k$-cliques, or complete subgraphs of order $k$, represent generators of $H^k(G)$.  

\begin{example} Let $\Gamma$ be the graph in Figure \ref{fivefourthree}, representing a \RAAG, $G$. We can think of the vertices $\{s_1,\dotsc, s_6\}$ as representing generators $\{z_1,\dotsc, z_6\}$ of $H^1(G)$. To simplify notation, let $z_i\cup z_j$ be denoted by $z_{ij}$. Thus $z_{12}$, $z_{13}$, $z_{14}$, $z_{15}$, $z_{23}$, $z_{24}$, $z_{25}$, $z_{34}$, $z_{35}$, $z_{36}$, $z_{45}$, $z_{46}$, and $z_{56}$ represent generators of $H^2(G)$. $H^3(G)$ is generated by $z_{123}$, $z_{124}$, $z_{125}$, $z_{134}$, $z_{135}$, $z_{145}$, $z_{234}$, $z_{235}$, $z_{245}$, $z_{345}$, $z_{346}$, $z_{356}$, and $z_{456}$, and $H^4(G)$ is generated by $z_{1234}$, $z_{1235}$, $z_{1245}$, $z_{1345}$, $z_{2345}$, and $z_{3456}$. Lastly, $z_{12345}$ generates the top dimensional cohomology class in $H^5(G)$. 
\end{example}

\vskip -.3 in
\begin{figure}[H]
\begin{center}
\begin{tikzpicture}[scale=2]
\tikzstyle{every node}=[draw, shape=circle, inner sep=.5mm, minimum size=.7cm, line width=1pt];
\path (2*72:1cm) node (s1) {$s_1$};
\path (3*72:1cm) node (s2) {$s_2$};
\path (4*72:1cm) node (s3) {$s_3$};
\path (5*72:1cm) node (s4) {$s_4$};
\path (6*72:1cm) node (s5) {$s_5$};
\path (1.9,0) node (s6) {$s_6$};
\foreach \i in {1,...,5} {\foreach \j in {\i,...,5} {\draw[line width = 2pt] (s\i) -- (s\j); } }
\foreach \i in {3,4,5} {\draw[line width = 2pt] (s\i) -- (s6); } 
\end{tikzpicture}
\caption{A graph of a 5-clique attached to a 4-clique along a face}
\label{fivefourthree}
\end{center}
\end{figure}
\vskip -.3in
With this calculation of the cohomology ring, we can prove Theorem \ref{trivialH4}.

\begin{proof}[Proof of Theorem \ref{trivialH4}]
Let $G$ be a \RAAG\ with $g$ generators and $r$ relations. Then $b_1(G)=g$ and $b_2(G)=r$. Let $M$ be any 4-manifold in $\mathcal M(G)$ and let $f:M \to K(G,1)$ be a map inducing an isomorphism on fundamental groups. If $H^4(G)=0$, then the image $I(M,f)$ of any basis of $H^2(G)$ can be represented by a zero matrix of dimension $b_2(G)$. Clearly, $m_2(G)=0$. Thus $2b_2(G) \leq h(G)$ by Proposition \ref{hnewlowerbound}. By Corollary \ref{hbounds}, $h(G) \leq 2r = 2b_2(G)$. 
\end{proof}

For other \RAAGs, since $H_4(G)$ is finitely generated, a computer program can calculate $m_2(G)$. Appendix \ref{sage code} contains the code for a Sage function \verb|max_rank()| that will compute $m_2(G)$ given an adjacency matrix of a graph associated to $G$.

\begin{numbered example} \label{radical example with matrix}
Let $G$ be the group given by the following graph $\Gamma$ 
\vskip .05 in
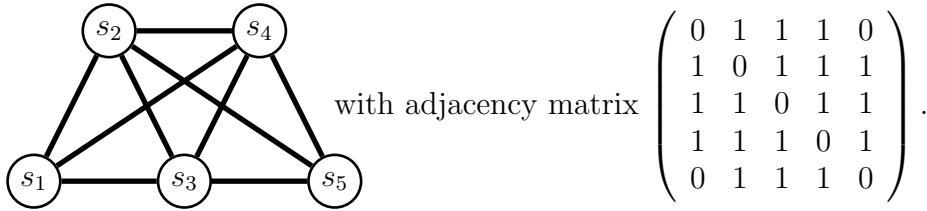
\begin{center}
\begin{minipage}{.2\linewidth}
\centering
\begin{tikzpicture}[scale = 2]
\tikzstyle{every node}=[draw, shape=circle, inner sep=.5mm, minimum size=.7cm, line width=1pt];
\path (0,0) node (s1) {$s_1$};
\path (0.5,1) node (s2) {$s_2$};
\path (1,0) node (s3) {$s_3$};
\path (1.5,1) node (s4) {$s_4$};
\path (2,0) node (s5) {$s_5$};
\draw[line width=2pt] (s1)--(s2)--(s3)--(s4)--(s5)--(s3)--(s1)--(s4)--(s2)--(s5);
\end{tikzpicture}
\end{minipage}
\begin{minipage}{.59\linewidth}
\centering
with adjacency matrix $\left(\begin{array}{cccccc}
0 & 1 & 1 & 1  & 0 \\
1 & 0 & 1 & 1 & 1 \\
1 & 1 & 0  & 1 & 1 \\
1 & 1 & 1 & 0  & 1 \\
0 & 1 & 1 & 1 & 0  \\
\end{array}\right).$
\end{minipage}
\end{center} 
\vskip .05 in
\noindent The set of vertices $\{s_1, \dotsc, s_5\}$ give an ordered basis for $H_1(G)$. Consider the following matrix represented by the form (\ref{formmod2}) under the ordered basis  $\{s_{12},s_{13},s_{14},s_{23},s_{24},s_{25},s_{34},s_{35},s_{45}\}$: 

\begin{equation*}
\left(\begin{array}{ccccccccc}
0 & 0 & 0 & 0 & 0 & 0 & a_1 & 0 & 0 \\
0 & 0 & 0 & 0 & a_1 & 0 & 0 & 0 & 0 \\
0 & 0 & 0 & a_1 & 0 & 0 & 0 & 0 & 0 \\
0 & 0 & a_1 & 0 & 0 & 0 & 0 & 0 & a_2 \\
0 & a_1 & 0 & 0 & 0 & 0 & 0 & a_2 & 0 \\
0 & 0 & 0 & 0 & 0 & 0 & a_2 & 0 & 0 \\
a_1 & 0 & 0 & 0 & 0 & a_2 & 0 & 0 & 0 \\
0 & 0 & 0 & 0 & a_2 & 0 & 0 & 0 & 0 \\
0 & 0 & 0 & a_2 & 0 & 0 & 0 & 0 & 0 \\
\end{array}\right)
\end{equation*}
The nonzero elements of this matrix are variables $a_1$ and $a_2$ representing the two generators $s_{1234}$ and $s_{2345}$ of $H^4(G;\Z_2)$. We compute $m_2(G)$ by finding all $2^{b_4(G)}$ ranks of the form and taking the maximum. Each rank is computed by replacing the $a_i$ in the above matrix with ones or zeros, each entry representing $\langle a_i,\alpha\rangle$. The Sage function \verb|max_rank()| from Appendix \ref{sage code} calculates $m_2(G)=6$ in this way. This implies that the minimum dimension of the radical is $b_2(G)-m_2(G) =  9-6 = 3$. However, since $b_4(G)$ is not too large, we can compute $m_2(G)$ easily by computing the minimum dimension of the radical by hand.

There are three nonzero choices in $H_4(G; \Z_2)$ for $\alpha$: $\alpha_1$, $\alpha_2$, and $\alpha_1+\alpha_2$, where $\langle a_i,\alpha_j\rangle = \delta_{ij}$. Note that if $\alpha=0$, the rank of the matrix is zero and the nullity is $b_2(G)=9$. If $\alpha=\alpha_1$, then $\langle a_1,\alpha\rangle =1$ and $\langle a_2,\alpha\rangle = 0$. In replacing $a_1$ with 1 and $a_2$ with 0, we see that this matrix has nullity 3. Similarly, if $\alpha=\alpha_2$, then we replace $a_1$ with 0 and $a_2$ with 1 and the matrix again has nullity 3. If $\alpha=\alpha_1+\alpha_2$, then we replace both $a_1$ and $a_2$ with 1. Three rows of the matrix (namely the fourth, fifth, and seventh) have two nonzero elements. Because we are computing the nullity of the matrix over $\Z_2$, the three linearly independent vectors 
\begin{equation*}
[0\ 0\ 1\ 0\ 0\ 0\ 0\ 0\ 1]^t, \; [0\ 1\ 0\ 0\ 0\ 0\ 0\ 1\ 0]^t, \; [1\ 0\ 0\ 0\ 0\ 1\ 0\ 0\ 0]^t
\end{equation*}
\noindent are in the kernel of the matrix. Since the matrix represents the form (\ref{formmod2}), the dimension of the radical for $\alpha=\alpha_1+\alpha_2$ is 3. Thus the minimum dimension of the radical is 3. Equivalently, the maximum rank is 6.
\end{numbered example}

The following proposition allows us to increase the trivial lower bound on $h(G)$ by 1 in the case when $b_2(G)$ is odd. 

\begin{prop}
\label{evenform}
If $G$ is a \RAAG, $m_2(G)$ is even. Thus if $b_2(G)$ is odd, $b_2(G)+1 \leq h(G)$. 
\end{prop}
\begin{proof}
Let $\{z_i\}$ be the set of generators of $H^1(G)$. Any nonzero generator of $H^2(G)$ is of the form $z_i \cup z_j$. Under the cup product map in (\ref{formmod2}), $\langle(z_i \cup z_j)^2,\alpha\rangle$ is zero for any choice of $\alpha$, since the $z_i$ are odd dimensional homology classes. A bilinear form $B:V \times V \to GF(q)$ is considered \emph{alternating} if $B(x,x)=0$ for all $x\in V$. Thus (\ref{formmod2}) is an alternating form. In \cite[Lemma 10]{DelsarteGoethals75} it is shown that alternating bilinear forms over $GF(q)$ have even rank. For $GF(q)=\Z_2$, we see that the rank must be even, and thus $m_2(G)$ must be even. If $b_2(G)$ is odd, then $m_2(G)$ is at most $b_2(G)-1$, and so $b_2(G) + 1 \leq h(G)$. 
\end{proof}

%%%%%%%%%%%%%%%%%%%%%%%%%%%%%%%%%%%%%%%%%%%%%%%%%%%%%%%%%%%%%%%%%%%%%%%%%%%%%%%%%%%%%%%%%%%%%%%%%%%%%%%%%%%%%%%%%%%%
\subsection{Finding $m_2(G)$ from a graph associated to $G$} \label{lower bound graph}
As discussed above, determining the maximum rank of (\ref{formmod2}) is equivalent to determining the minimum dimension of the radical of $H^2(G,\Z_2)$. In many cases it is not difficult to calculate this minimum dimension straight from the graph of $G$. 

In the following example, we let $\{s_i\}$ be a basis for the homology and $\{z_i\}$ be the dual basis for the cohomology. 

\begin{numbered example} \label{m2 by hand}
Let the graph of G be given below:

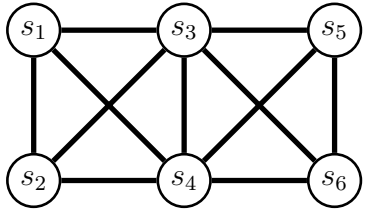
\begin{center}
\begin{tikzpicture}[scale=2]
\tikzstyle{every node}=[draw, shape=circle, inner sep=.5mm, minimum size=.7cm, line width=1pt];
\path (0,1) node (y1) {$s_1$};
\path (0,0) node (y2) {$s_2$};
\path (2,1) node (y3) {$s_5$};
\path (2,0) node (y4) {$s_6$};
\path (1,1) node (x1) {$s_3$};
\path (1,0) node (x2) {$s_4$};
\draw[line width = 2pt] (y2) -- (x1) -- (x2) -- (y1)  -- (x1) -- (y3) (y2) -- (x2) -- (y4) (y3) -- (x2) (y4) -- (x1) (y1) -- (y2) (y3) -- (y4);
\end{tikzpicture}
\end{center}

This graph is made up of exactly two 4-cliques, so $b_4(G)=2$. Since each 4-clique has 6 edges, we can compute $b_2(G)$ by multiplying 6 by the number of 4-cliques and subtracting the number of \emph{shared edges}, edges that belong to more than one 4-clique. In this example, $b_2(G) = 6(2)-1=11$. Define the two generators $s_{1234}$ and $s_{3456}$ of $H_4(G)$ to be $\alpha_1$ and $\alpha_2$, respectively. Let $\alpha$ be an arbitrary element of $H_4(G;\Z_2)$. Then $\alpha$ is of the form $c_1\alpha_1 + c_2\alpha_2$, where $c_1$ and $c_2$ are either 0 or 1. There are only three nontrivial choices for $\alpha$. If $c_1=0$, then $z_{13}$, $z_{14}$, $z_{23}$, and $z_{24}$ give a basis for the radical of the form (\ref{formmod2}). For any generator $z$ of $H^2(G)$, $\langle z_{13} \cup z,\alpha_2\rangle = 0$, $\langle z_{14} \cup z,\alpha_2\rangle =0$, $\langle z_{23} \cup z,\alpha_2\rangle = 0$, and $\langle z_{24} \cup z,\alpha_2\rangle =0$. This is shown in the graph since each of the corresponding four edges ($s_{13}$, $s_{14}$, $s_{23}$, and $s_{24}$) are only part of the 4-clique $\alpha_1$. Similarly, if $c_2=0$, then $z_{35}$, $z_{36}$, $z_{45}$, and $z_{46}$ give a basis for the radical.  

Lastly, consider the case when both $c_1=1$ and $c_2=1$. Consider the image of $z_{12} + z_{56}$ cupped with an arbitrary generator $z$ under the form (\ref{formmod2}):
\begin{equation*}
\langle (z_{12}+z_{56}) \cup z,\alpha_1 + \alpha_2\rangle = \langle z_{12} \cup z,\alpha_1\rangle + \langle z_{56} \cup z,\alpha_1\rangle+ \langle z_{12} \cup z,\alpha_2\rangle + \langle z_{56} \cup z, \alpha_2\rangle
\end{equation*}
On the right-hand side, the middle two summands are zero, and $\langle z_{12} \cup z,\alpha_1\rangle$ and $\langle z_{56} \cup z, \alpha_2\rangle$ are zero unless $z=z_{34}$. If $z=z_{34}$, then $\langle z_{12} \cup z_{34},\alpha_1\rangle + \langle z_{56} \cup z_{34}, \alpha_2\rangle = 1 + 1 \equiv 0 \mod 2$. One can check other linearly independent elements of $H^2(G;\Z_2)$ and see that this unique element provides a basis for the radical. Thus the maximum rank of the form is 10 instead of 11. This gives the lower bound $12 \leq h(G)$. 

Consider next a graph of three 4-cliques attached edge-to-edge, and an arbitrary element $\alpha = c_1\alpha_1 + c_2\alpha_2 + c_3\alpha_3 \in H_4(G,\Z_2)$, $c_i \in \{0,1\}$. If any $c_i = 0$, the nullity of the form is at least 4, for the same reason as in the above case.  Thus we may assume $\alpha = \alpha_1 + \alpha_2+ \alpha_3$. One can verify that there are no nonzero elements in the radical. 

In a graph of four 4-cliques attached edge-to-edge, we have the same assumption that the minimum nullity of the form occurs with the choice $\alpha = \alpha_1 + \dotsc + \alpha_4$. Again, the nullity is 1; the element in this radical is the sum of the generators represented by the \textbf{\textcolor{blue}{bold}} edges in the following graph:

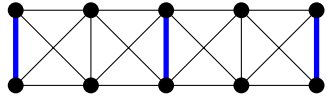
\begin{center}
\begin{tikzpicture}
\foreach \i in {0,...,4}{
    \path (\i*1,0) coordinate (v\i);
    \path (\i*1,1) coordinate (w\i);
    \draw (v\i) -- (w\i);
}
\draw (v0)--(v4) (w0)--(w4) (w0) -- (v1) -- (w2) -- (v3) -- (w4) (v0) -- (w1) -- (v2) -- (w3) -- (v4);
\foreach \i in {0,2,4}{\draw[line width=2pt,blue] (v\i)--(w\i);}
\foreach \i in {0,...,4}{\fill (v\i) circle (3pt); \fill (w\i) circle (3pt);}
\end{tikzpicture}
\end{center}

\noindent A pattern develops that indicates that in graphs with a string of $k$ 4-cliques attached edge-to-edge, the nullity is either 0 or 1, depending on the parity of $k$. 
\end{numbered example}

Alternatively, for a \RAAG\ $G$, one can bound $h(G)$ from below by finding a maximum isotropic subspace of the form (\ref{formmod2}). This yields the same calculation of the lower bound from Proposition \ref{hnewlowerbound}, since twice the dimension of a maximum isotropic subspace of $H^2(G;\Z_2)$ is equal to $2b_2(G)-m_2(G)$. In some cases we can find a subset of the generators of $H^2(G;\Z_2)$ that form a maximum isotropic subspace, which are represented in the associated graph as edges.
\vskip -.2 in
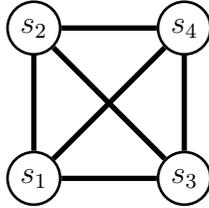
\begin{figure}[h]
\begin{center}
\begin{tikzpicture}[scale=2]
\tikzstyle{every node}=[draw, shape=circle, inner sep=.5mm, minimum size=.7cm, line width=1pt];
\path (0,0) node (x1) {$s_1$};
\path (0,1) node (x2) {$s_2$};
\path (1,0) node (x3) {$s_3$};
\path (1,1) node (x4) {$s_4$};
\draw[line width = 2pt] (x1) -- (x2) -- (x3) -- (x4) -- (x1) -- (x3) (x2) -- (x4);
\end{tikzpicture}
\caption{A 4-clique, the graph associated to $\Z^4$.}
\label{zfour}
\end{center}
\end{figure} 

Consider the graph of a 4-clique in Figure \ref{zfour}. The vertices $\{s_i\}$ determine an ordered basis $\{z_i\}$ for $H^1(\Z^4)$. Then  $\{z_{12}, z_{13}, z_{14}, z_{23}, z_{24}, z_{34}\}$ represent edges of the graph, and $z_{1234}$ represents the 4-clique. The following two sets give maximum isotropic subspaces for $H^2(\Z^4)$: $\{z_{12},z_{24},z_{14}\}$ and $\{z_{12},z_{24},z_{23}\}$. In each set, every pair of generators is either of the form $(z_{ij},z_{jk})$,  $(z_{ij},z_{ik})$, or $(z_{ij},z_{jk})$. In every pair, the product of the two generators is zero because $z_{ii}=0$ for all $i$. The edges represented by the two sets above form a triangle and a claw, respectively. The two isotropic subspaces are \textbf{\textcolor{blue}{highlighted}} in Figure \ref{isotropic} (a).  Of course, these sets are not the only choices for maximum isotropic subspaces for a 4-clique. However, any three dimensional isotropic subspace of a 4-clique will either form a triangle or a claw in the graph. 

Consequently, pairs of generators $(z_{ij},z_{kl})$ will cup nontrivially if $i,j,k,l$ are all distinct. Therefore in every 4-clique of a graph, the maximum isotropic subspace will never contain any pair of \textbf{\textcolor{red}}{bold} edges shown in Figure \ref{isotropic} (b).

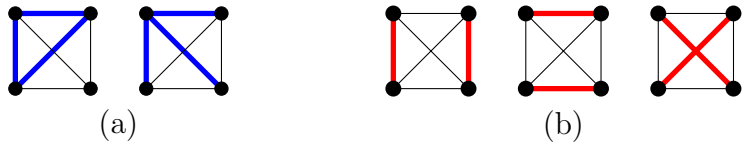
\begin{figure}[h]
\begin{center}
\begin{minipage}{.3\linewidth}
\begin{center}
\begin{tikzpicture}
\draw[blue, line width = 2pt] (0,0) -- (0,1) -- (1,1) --  (0,0);
\draw[black, fill=black]
	(0,0) circle (2.5pt)
	(0,1) circle (2.5pt)
	(1,1) circle (2.5pt)	
	(1,0) circle (2.5pt);
\draw (0,0) -- (1,0) -- (0,1) (1,0) -- (1,1);
\end{tikzpicture}
\quad
\begin{tikzpicture}
\draw[blue, line width = 2pt] (0,0) -- (0,1) -- (1,1)  (0,1) -- (1,0);
\draw[black, fill=black]
	(0,0) circle (2.5pt)
	(0,1) circle (2.5pt)
	(1,1) circle (2.5pt)	
	(1,0) circle (2.5pt);
\draw (0,0) -- (1,0) -- (1,1) -- cycle;
\end{tikzpicture} \\
(a)
\end{center}
\end{minipage}
\begin{minipage}{.4\linewidth}
\begin{center}
\begin{tikzpicture}
\draw[red, line width = 2pt] (0,0) -- (0,1)  (1,0) -- (1,1);
\draw[black, fill=black, line width=1pt]
	(0,0) circle (2.5pt)
	(0,1) circle (2.5pt)
	(1,1) circle (2.5pt)	
	(1,0) circle (2.5pt);
\draw (0,0) -- (1,0) -- (0,1) -- (1,1) -- cycle;
\end{tikzpicture}
\quad
\begin{tikzpicture}
\draw[red, line width = 2pt] (0,0) -- (1,0) (0,1) -- (1,1);
\draw[black, fill=black, line width=1pt]
	(0,0) circle (2.5pt)
	(0,1) circle (2.5pt)
	(1,1) circle (2.5pt)	
	(1,0) circle (2.5pt);
\draw (0,0) -- (0,1) -- (1,0) -- (1,1) -- cycle;
\end{tikzpicture}
\quad
\begin{tikzpicture}
\draw[red, line width = 2pt] (0,0) -- (1,1)  (0,1) -- (1,0);
\draw[black, fill=black, line width=1pt]
	(0,0) circle (2.5pt)
	(0,1) circle (2.5pt)
	(1,1) circle (2.5pt)	
	(1,0) circle (2.5pt);
\draw (0,0) -- (1,0) -- (1,1) -- (0,1) -- cycle;
\end{tikzpicture} \\
(b)
\end{center}
\end{minipage}
\end{center}
\caption{(a) Triangles and claws, formed by the \textbf{\textcolor{blue}{bold}} edges in (a), are subgraphs that make up an isotropic subspace in each 4-clique. Pairs of edges \textbf{\textcolor{red}{highlighted}} in (b) do not.}
\label{isotropic}
\end{figure}

%%%%%%%%%%%%%%%%%%%%%%%%%%%%%%%%%%%%%%%%%%%%%%%%%%%%%%%%%%%%%%%%%%%%%%%%%%%%%%%%%%%%%%%%%%%%%%%%%%%%%%%%%%%%%%%%%%%%
\subsection{Cohomologically minimal groups} \label{property}

The main question we will discuss in this paper is the following: 

\emph{For which \RAAGs\ does the structure of $H^*(G)$ yield the only obstruction to calculating $h(G)$?} 

Let us make the following definition.
\begin{defn}
We say that a finitely presented group $G$ is \emph{\cm}\ if  $h(G) = 2b_2(G) - m_2(G)$. 
\end{defn}

Restricting our discussion of the minimum $b_2$ problem to \cm\ groups, we now prove Theorems \ref{freeproducts}, \ref{wedging}, and \ref{breakdown}. 

\begin{proof}[Proof of Theorem \ref{freeproducts}]
By assumption, $h(G_1) = 2b_2(G_1) - m_2(G_1)$ and $h(G_2) = 2b_2(G_2) - m_2(G_2)$. For the free product $G_1*G_2$, $m_2(G_1*G_2)=m_2(G_1)+m_2(G_2)$ and $b_2(G_1*G_2) = b_2(G_1) + b_2(G_2)$. To see the former statement, note that the bilinear form under the free product splits into a direct sum of forms. For the latter statement, note that homology is additive under free products. This gives a lower bound on $h(G_1*G_2)$: 
\begin{eqnarray*}
2(b_2(G_1) + b_2(G_2)) - (m_2(G_1) + m_2(G_2)) &\leq& h(G_1*G_2) \\
h(G_1) + h(G_2) & \leq & h(G_1*G_2) 
\end{eqnarray*}

Let $M_i$ be a realizing manifold for $h(G_i)$; that is, $\pi_1(M_i) \cong G_i$ and $b_2(M_i) = h(G_i)$. Then $b_2(M_1 \# M_2) = b_2(M_1)+b_2(M_2) = h(G_1) + h(G_2)$.  Therefore $h(G_1*G_2) = h(G_1) + h(G_2)$.
Note that this implies one realizing manifold for $h(G_1*G_2)$ is the connected sum of the realizing manifolds for $h(G_1)$ and $h(G_2)$.
\end{proof} 

\begin{proof}[Proof of Theorem \ref{wedging}]
The proof of this theorem is very similar to that of Theorem \ref{freeproducts}. Let $G_i$ be the \RAAG\ associated to $\Gamma_i$. By assumption, $h(G_1) + h(G_2) = 2(b_2(G_1) + b_2(G_2)) - (m_2(G_1)+m_2(G_2))$. We will first see that by identifying generators of $G_1$ and $G_2$, we do not create any new 4-cliques, which will assert that $m_2(G) = m_2(G_1) + m_2(G_2)$. 

Say that by identifying $s_i$ with $t_i$ and $s_j$ with $t_j$ we create a 4-clique involving the two newly identified generators. This would require an edge between either $s_i$ and $s_j$ or $t_i$ and $t_j$. However, we have assumed both the $\{s_i\}$ and $\{t_i\}$ are pairwise non-adjacent. No edges in $\Gamma_1$ will form a 4-clique with edges in $\Gamma_2$ after the identifications of the vertices, so the bilinear form for $H^2(G)$ splits into a direct sum of forms for $H^2(G_1)$ and $H^2(G_2)$. Thus $h(G_1)+h(G_2) \leq h(G)$.

Let $M_i$ be a realizing manifold for $h(G_i)$. We can build a realizing manifold $M$ for $h(G)$ by taking $M_1\# M_2$ and performing $m$ surgeries, each identifying $s_i$ with $t_i$. These surgeries do not increase $b_2$, as we will see in Section \ref{KL tools}. Thus, $M$ has $b_2(M)=b_2(M_1)+b_2(M_2) = h(G_1)+h(G_2)$.
\end{proof}

\begin{proof}[Proof of Theorem \ref{breakdown}]
The $r$ edges deleted from $\Gamma$ represent basis elements of $H^2(G)$ (as are all edges of $\Gamma$) and necessarily cup to zero with any other basis element under (\ref{formmod2}), and so they are in the radical. This and the assumption that each $G_i$ is \cm\ imply that $m_2(G) = \sum_i m_2(G_i)$. Note also that $b_2(G) = \sum_i b_2(G_i) + r$. Therefore we have the following lower bound on $h(G)$:
\begin{equation*}
2(b_2(G_1) + \dotsm + b_2(G_k)+r) - (m_2(G_1) + \dotsm + m_2(G_k)) = h(G_1) + \dotsm h(G_k) + 2r \leq h(G).
\end{equation*}
Let $M_{i}$ be a realizing manifold for $h(G_i)$. Build a realizing manifold $M$ for $h(G)$ by starting with the connected sum $M_1 \# \dotsm \# M_k$ and performing $r$ surgeries to induce the relations we initially ignored from $G$. Each surgery increases $b_2$ by 2, as we will see in Section \ref{KL tools}. These surgeries yield a 4-manifold $M$ with $\pi_1(M) = G$ and $b_2(M) = \sum_i b_2(M_i) + 2r=\sum_i h(G_i) +2r$. 
\end{proof}

These theorems break down the minimum $b_2$ problem for \RAAGs\ into smaller subproblems. Specifically, one need only consider the case where $\Gamma$ is a connected graph containing only 4-cliques.

\begin{numbered example} \label{example of breakdown}
Let $G$ be a \RAAG\ with associated graph $\Gamma$ in Figure \ref{breakdown example} (a). Using the above theorems, we can break down the calculation of $h(G)$ into calculations for three different groups. 
\vskip .2 in
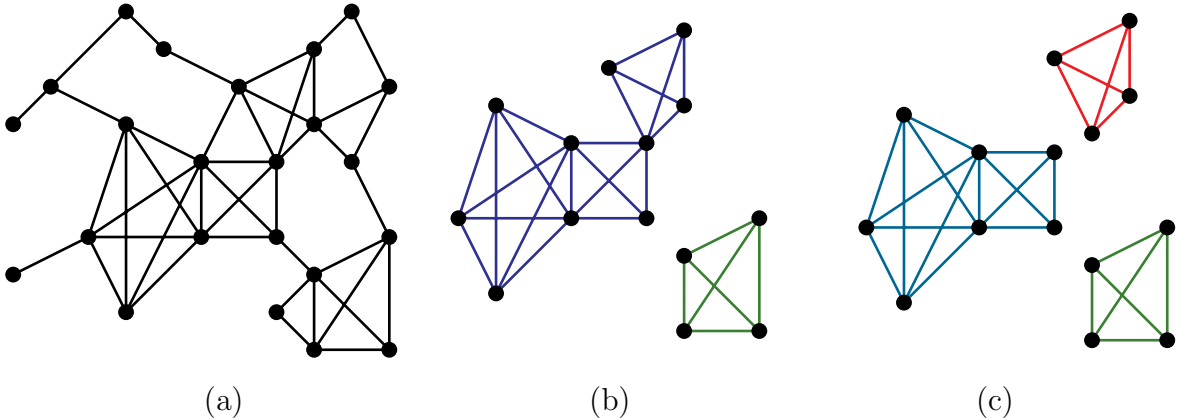
\begin{figure}[h]
\centering
\begin{minipage}{.32\linewidth}
\centering
\begin{tikzpicture}[scale=.5]
\path (0,5) coordinate (a); \fill (a) circle (6pt);
\path (1,6) coordinate (b); \fill (b) circle (6pt);
\path (3,8) coordinate (c); \fill (c) circle (6pt);
\path (3,5) coordinate (d); \fill (d) circle (6pt);
\path (4,7) coordinate (e); \fill (e) circle (6pt);
\path (5,4) coordinate (f); \fill (f) circle (6pt);
\path (6,6) coordinate (g); \fill (g) circle (6pt);
\path (0,1) coordinate (h); \fill (h) circle (6pt);
\path (2,2) coordinate (i); \fill (i) circle (6pt);
\path (3,0) coordinate (j); \fill (j) circle (6pt);
\path (5,2) coordinate (k); \fill (k) circle (6pt);
\path (7,2) coordinate (l); \fill (l) circle (6pt);
\path (7,0) coordinate (m); \fill (m) circle (6pt);
\path (8,1) coordinate (n); \fill (n) circle (6pt);
\path (8,-1) coordinate (o); \fill (o) circle (6pt);
\path (10,2) coordinate (p); \fill (p) circle (6pt);
\path (10,-1) coordinate (q); \fill (q) circle (6pt);
\path (7,4) coordinate (r); \fill (r) circle (6pt);
\path (8,5) coordinate (s); \fill (s) circle (6pt);
\path (8,7) coordinate (t); \fill (t) circle (6pt);
\path (9,8) coordinate (u); \fill (u) circle (6pt);
\path (10,6) coordinate (v); \fill (v) circle (6pt);
\path (9,4) coordinate (w); \fill (w) circle (6pt);
\draw[line width = 1pt] (a)--(b)--(c)--(e)--(g)--(f) (b)--(d) (t)--(u)--(v)--(s)--(w)--(v) (h)--(i) (l)--(n)--(m)--(o) (w)--(p);
\draw[line width = 1pt] (i)--(d)--(f)--(k)--(j)--(i)--(f)--(j)--(d)--(k)--(i) (f)--(r)--(l)--(k)--(r) (f)--(l);
\draw[line width = 1pt] (g)--(t)--(s)--(r)--(g)--(s) (r)--(t);
\draw[line width = 1pt] (o)--(n)--(p)--(q)--(o)--(p) (n)--(q);
\end{tikzpicture} 
\end{minipage}
\begin{minipage}{.32\linewidth}
\centering
\begin{tikzpicture}[scale=.5]
\path (3,5) coordinate (d); 
\path (5,4) coordinate (f); 
\path (6,6) coordinate (g);
\path (2,2) coordinate (i); 
\path (3,0) coordinate (j); 
\path (5,2) coordinate (k); 
\path (7,2) coordinate (l); 
\path (8,1) coordinate (n); 
\path (8,-1) coordinate (o); 
\path (10,2) coordinate (p); 
\path (10,-1) coordinate (q); 
\path (7,4) coordinate (r); 
\path (8,5) coordinate (s); 
\path (8,7) coordinate (t); 
\draw[line width = 1pt, Blue] (i)--(d)--(f)--(k)--(j)--(i)--(f)--(j)--(d)--(k)--(i) (f)--(r)--(l)--(k)--(r) (f)--(l);
\draw[line width = 1pt, Blue] (g)--(t)--(s)--(r)--(g)--(s) (r)--(t);
\draw[line width = 1pt, OliveGreen] (o)--(n)--(p)--(q)--(o)--(p) (n)--(q);
\fill (d) circle (6pt);
\fill (f) circle (6pt);
\fill (g) circle (6pt);
\fill (i) circle (6pt);
\fill (j) circle (6pt);
\fill (k) circle (6pt);
\fill (l) circle (6pt);
\fill (n) circle (6pt);
\fill (o) circle (6pt);
\fill (p) circle (6pt);
\fill (q) circle (6pt);
\fill (r) circle (6pt);
\fill (s) circle (6pt);
\fill (t) circle (6pt);
\end{tikzpicture} 
\end{minipage}
\begin{minipage}{.32\linewidth}
\centering
\begin{tikzpicture}[scale=.5]
\path (3,5) coordinate (d); 
\path (5,4) coordinate (f); 
\path (7,6.5) coordinate (g);
\path (2,2) coordinate (i); 
\path (3,0) coordinate (j); 
\path (5,2) coordinate (k); 
\path (7,2) coordinate (l); 
\path (8,1) coordinate (n); 
\path (8,-1) coordinate (o); 
\path (10,2) coordinate (p); 
\path (10,-1) coordinate (q); 
\path (8,4.5) coordinate (r'); 
\path (7,4) coordinate (r); 
\path (9,5.5) coordinate (s); 
\path (9,7.5) coordinate (t); 
\draw[line width = 1pt, MidnightBlue] (i)--(d)--(f)--(k)--(j)--(i)--(f)--(j)--(d)--(k)--(i) (f)--(r)--(l)--(k)--(r) (f)--(l);
\draw[line width = 1pt, Red] (g)--(t)--(s)--(r')--(g)--(s) (r')--(t);
\draw[line width = 1pt, OliveGreen] (o)--(n)--(p)--(q)--(o)--(p) (n)--(q);
\fill (d) circle (6pt);
\fill (f) circle (6pt);
\fill (g) circle (6pt);
\fill (i) circle (6pt);
\fill (j) circle (6pt);
\fill (k) circle (6pt);
\fill (l) circle (6pt);
\fill (n) circle (6pt);
\fill (o) circle (6pt);
\fill (p) circle (6pt);
\fill (q) circle (6pt);
\fill (r) circle (6pt);
\fill (s) circle (6pt);
\fill (t) circle (6pt);
\fill (r') circle (6pt);
\end{tikzpicture} 
\end{minipage}

\bigskip
(a) \hspace{1.7 in} (b) \hspace{1.7 in} (c)
\caption{An example of the breakdown of a graph into disjoint subgraphs, for the calculation of $h$ according to Theorems \ref{wedging} and \ref{breakdown}}
\label{breakdown example}
\end{figure}

By removing 16 edges in $\Gamma$ that are not part of a 4-clique, we get two disjoint graphs in Figure \ref{breakdown example} (b). Call these two graphs $\Gamma_1$ and $\Gamma_2$. Assuming the resulting \RAAGs\ $G_1$ and $G_2$ associated to $\Gamma_1$ and $\Gamma_2$ are \cm, Theorem \ref{breakdown} asserts that $h(G) = h(G_1)+h(G_2) + 2(16)$. Furthermore, the graph on the left in Figure \ref{breakdown example} (b) has two subgraphs joined at one vertex. By splitting the subgraphs apart, we have the three disjoint graphs in Figure \ref{breakdown example} (c). Call these disjoint graphs $\Gamma_a$, $\Gamma_b$, and $\Gamma_c$. Theorem \ref{wedging} asserts that $h(G_1)+h(G_2) = h(G_a) + h(G_b) + h(G_c)$, under the assumption that $G_a$, $G_b$, and $G_c$ are each \cm. Together, we have
\begin{equation*}
h(G) = h(G_a)+h(G_b)+h(G_c) + 32.
\end{equation*}
Indeed, the groups corresponding to the graphs in Figure \ref{breakdown example} (c) are \cm. In Section \ref{surface graphs} we will complete the calculation of $h(G)$ by calculating $h(G_a)$, $h(G_b)$, and $h(G_c)$. See Example \ref{5-clique and 4-clique along an edge} for details.
\end{numbered example}

%%%%%%%%%%%%%%%%%%%%%%%%%%%%%%%%%%%%%%%%%%%%%%%%%%%%%%%%%%%%%%%%%%%%%%%%%%%%%%%%%%%%%%%%%%%%%%%%%%%%%%%%%%%%%%%%%%%%
%%%%%%%%%%%%%%%%%%%%%%%%%%%%%%%%% TOOLS FOR 4-MANIFOLD CONSTRUCTIONS %%%%%%%%%%%%%%%%%%%%%%%%%%%%%%%%%%%%%%%%%%%%%%%
%%%%%%%%%%%%%%%%%%%%%%%%%%%%%%%%%%%%%%%%%%%%%%%%%%%%%%%%%%%%%%%%%%%%%%%%%%%%%%%%%%%%%%%%%%%%%%%%%%%%%%%%%%%%%%%%%%%%
\section[Construction tools]{Tools for 4-manifold constructions}\label{tools}

%%%%%%%%%%%%%%%%%%%%%%%%%%%%%%%%%%%%%%%%%%%%%%%%%%%%%%%%%%%%%%%%%%%%%%%%%%%%%%%%%%%%%%%%%%%%%%%%%%%%%%%%%%%%%%%%%%%%
\subsection{Tools from \cite{KirkLivingston05}} \label{KL tools}
We will make use of the following classical result.

\begin{lemma}[{Milnor, \cite[Lemma 2]{Milnor61}}] \label{surgeries}
If a 4-manifold $M'$ is constructed from a compact 4-manifold $M$ via surgery along a curve $\gamma$, then $b_2(M') = b_2(M)$ if $\gamma$ is of infinite order in $H_1(M)$ and $b_2(M') = b_2(M) + 2$ otherwise. 
\end{lemma}
\begin{proof}
Surgery on $M$ is performed by removing $S^1 \times B^3$ and replacing it with $D^2 \times S^2$, so $\chi(M') = \chi(M) + 2$. If $\gamma$ is of infinite order in $H_1(M)$, $b_1(M') = b_1(M)-1$ and $b_3(M') = b_3(M)-1$. Thus the difference in Euler characteristic comes from the change in $b_1$ and $b_3$, so $b_2(M') = b_2(M)$. If $\gamma$ is of finite order, $b_1$ and $b_3$ are unchanged, so the difference in Euler characteristic comes from an increase in $b_2$ by 2. 
\end{proof}

We will use this lemma to perform two types of surgeries on curves in a 4-manifold. The first type is surgery to identify generators of the fundamental group: surgery on the curve $\gamma = ab^{-1}$ identifies generators $a$ and $b$ and is a curve of infinite order in $H_1$. The second type is surgery to kill a commutator relation. Performing surgery on the curve $\gamma = aba^{-1}b^{-1}$ kills the commutator of $a$ and $b$, and is a nullhomologous curve. Lemma \ref{surgeries} implies that performing surgery to identify generators does not change $b_2$, whereas a surgery to kill a commutator increases $b_2$. 

The next definition and subsequent theorem were developed in \cite{KirkLivingston05} and are extremely useful in constructing realizing manifolds for \RAAGs. 

\begin{defn}[{Kirk-Livingston, \cite[Definition 5]{KirkLivingston05}}]
A \emph{4-reduction of a group $G$ by a 4-tuple of elements} $[w_1,w_2,w_3,w_4]$, $w_i \in G$, is the quotient of $G$ by the normal subgroup generated by the 6 commutators $[w_i,w_j]$, $i<j$. This quotient is denoted $G/[w_1,w_2,w_3,w_4]$. 
More generally, we say a group $G$ can be \emph{4-reduced to the group $H$ using the 4-tuples} $\{[w_{1k},w_{2k},w_{3k},w_{4k}]\}$, $k=1,\dotsc, \ell$ if $H$ is isomorphic to the quotient of $G$ by the normal subgroup generated by the $6\ell$ commutators $[w_{ik},w_{jk}]$, $i<j$, $k = 1,\dotsc, \ell$.
\end{defn}

\begin{theorem}[{Kirk-Livingston, \cite[Theorem 6]{KirkLivingston05}}]
If $M$ is a 4-manifold and $w_i \in \pi_1(M)$ for $i=1,\dotsc,4$, then there is a 4-manifold $M'$ with $\pi_1(M') = \pi_1(M)/[w_1,w_2,w_3,w_4]$ and $b_2(M') = b_2(M) + 6$.
\end{theorem}
\begin{proof}
Form the connected sum $M \# T^4$ which increases $b_2$ by 6. Let $\pi_1(T^4)$ be generated by $\{x_1,x_2,x_3,x_4\}$. Perform surgery on 4 curves $x_iw_i^{-1}$, $i=1,\dotsc, 4$ to identify the generators of $\pi_1(T^4)$ with the elements $w_i$. By Lemma \ref{surgeries}, these surgeries do not change $b_2$ since they are of infinite order in $H_1(M\# T^4)$. The effect of the surgeries is that each of the elements $w_i$ commute with each other, so $M'$ is a manifold with the fundamental group claimed.
\end{proof}

%%%%%%%%%%%%%%%%%%%%%%%%%%%%%%%%%%%%%%%%%%%%%%%%%%%%%%%%%%%%%%%%%%%%%%%%%%%%%%%%%%%%%%%%%%%%%%%%%%%%%%%%%%%%%%%%%%%%
\subsection{Graphical representations of fundamental groups} \label{surface graphs}

Many realizing 4-manifold constructions contain connected sums of 4-tori and other products of surfaces. It is very convenient to view 4-manifolds by the graphs of their fundamental groups, if possible. 

First let us consider the product of a torus $T^2$ with a genus 2 surface $\Sigma_2$, with $\pi_1$ generated by $\{x_1,x_2\}$ and $\{y_1,y_2,y_3,y_4\}$. This 4-manifold has a commutator relation between $x_1$ and $x_2$ as well as commutator relations between the $x_i$ and $y_j$. These we can represent in a graph of the fundamental group as edges between the corresponding vertices. In addition to the commutator relations we have the surface relation $[y_1,y_2][y_3,y_4]=1$, so this 4-manifold does not have a \RAAG\ as its fundamental group. However, for convenience, let us display the surface relation as two dashed edges, one between $y_1$ and $y_2$ and the other between $y_3$ and $y_4$, as in Figure \ref{toruscrosssurface}.
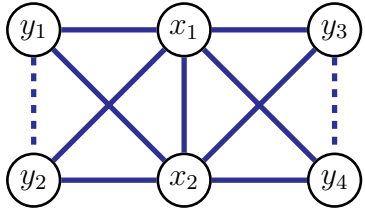
\begin{figure}[h]
\begin{center}
\begin{tikzpicture}[scale=2]
\tikzstyle{every node}=[draw, shape=circle, inner sep=.5mm, minimum size=.7cm, line width=1pt];
\path (0,1) node (y1) {$y_1$};
\path (0,0) node (y2) {$y_2$};
\path (2,1) node (y3) {$y_3$};
\path (2,0) node (y4) {$y_4$};
\path (1,1) node (x1) {$x_1$};
\path (1,0) node (x2) {$x_2$};
\draw[Blue, line width = 2pt] (y2) -- (x1) -- (x2) -- (y1)  -- (x1) -- (y3) (y2) -- (x2) -- (y4) (y3) -- (x2) (y4) -- (x1);
\draw[Blue, line width = 2 pt, dashed] (y1) -- (y2) (y3) -- (y4);
\end{tikzpicture}
\caption{A graph representing $\pi_1(T^2 \times \Sigma_2)$}
\label{toruscrosssurface}
\end{center}
\end{figure}

\noindent Note that if we perform surgery to either induce the commutator relation $[y_1,y_2]=1$ or $[y_3,y_4]=1$, or the relation is induced another way (for example, by a 4-reduction), then the resulting 4-manifold has a \RAAG\ as its fundamental group. 

\begin{numbered example}\label{5-clique and 4-clique along an edge}
Return to the graph $\Gamma$ from Example \ref{example of breakdown}. 

Two of the disjoint subgraphs in Figure \ref{breakdown example} (c) are 4-cliques. Without loss of generality, let these be $\Gamma_a$ and $\Gamma_b$. Both groups are copies of $\Z^4$, and $h(\Z^4)=6$. The third graph, $\Gamma_c$, consists of a 5-clique and a 4-clique sharing one edge. By calculating $m_2(G_c)=12$ and $b_2(G_c) =15$, we know $18 \leq h(G_c)$. A realizing manifold is built as follows: Start with $(T^2 \times \Sigma_2) \# T^4$, with $\pi_1$ generated by $\{x_1,x_2\}$, $\{y_1,y_2,y_3,y_4\}$, and $\{z_1,z_2,z_3,z_4\}$ and $b_2=16$. Perform surgery to identify $y_1$ with $z_2$, $x_1$ with $z_3$, and $z_4$ with $y_2$. These surgeries do not change $b_2$. Finally perform surgery to induce the commutator relation $[z_1,x_2]=1$. This surgery will increase $b_2$ by 2, and yields a manifold $M\in \mathcal M(G_c)$ with $b_2(M) = 18$. Figure \ref{figure 5-clique and 4-clique along an edge} shows the graph of $\pi_1(M)$. Thus $h(G_a)=6$, $h(G_b)=6$, and $h(G_c)=18$, and all groups are \cm.

Recall from Example \ref{example of breakdown} that 16 edges were deleted from $\Gamma$. By Theorems \ref{wedging} and \ref{breakdown}, we know $h(G) = h(G_a)+h(G_b)+h(G_c)+32$. Thus $h(G) = 6+ 6 + 18 + 32 = 62$.
 
\begin{figure}[h]
\begin{center}
\begin{tikzpicture}[scale=2]
\tikzstyle{every node}=[draw, shape=circle, inner sep=.5mm, minimum size=.8cm, line width=1pt];
\path (1,1.5) node (y1) {$y_1$};
\path (1,-0.5) node (y2) {$y_2$};
\path (3,1) node (y3) {$y_3$};
\path (3,0) node (y4) {$y_4$};
\path (2,1) node (x1) {$x_1$};
\path (2,0) node (x2) {$x_2$};
\path (2,1) node (z3) { };
\path (0.5,0.5) node (z1) {$z_1$};
\path (1,-0.5) node (z4) {};
\path (1,1.5) node (z2) {};
\draw[Blue, line width = 2pt] (y2) -- (x1) -- (x2) -- (y1)  -- (x1) -- (y3) (y2) -- (x2) -- (y4) (y3) -- (x2) (y4) -- (x1);
\foreach \i in {1,...,4}{\foreach \j in {\i,...,4}{\draw[red, line width = 2pt] (z\i) -- (z\j);}}	
\draw[red, line width = 2 pt] (y3) -- (y4);
\draw[Blue, line width = 2 pt, dashed] (y1) -- (y2) (y3) -- (y4);
\draw[OliveGreen, line width = 2pt] (z1)--(x2);
\end{tikzpicture} 
\caption{$\pi_1$ of the realizing manifold for $h(G_c)$}
\label{figure 5-clique and 4-clique along an edge}
\end{center}
\end{figure}
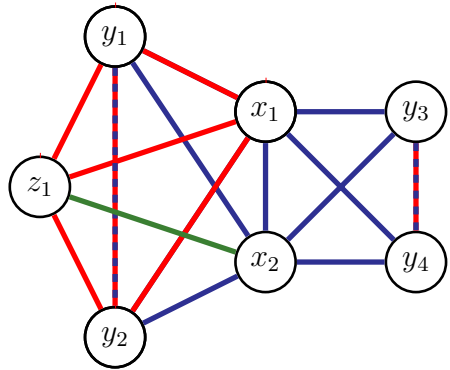
\end{numbered example}

%%%%%%%%%%%%%%%%%%%%%%%%%%%%%%%%%%%%%%%%%%%%%%%%%%%%%%%%%%%%%%%%%%%%%%%%%%%%%%%%%%%%%%%%%%%%%%%%%%%%%%%%%%%%%%%%%%%%
\subsection{4-reductions in action}\label{4-reductions in action}

Many realizing manifolds contain 4-reductions in their constructions, so it is helpful to see these reflected in the graphs of fundamental groups. Let us begin with a 4-manifold $M = \#5 (S^1 \times S^3)$ in which $b_2(M)=0$. Let the generators of $\pi_1(M)$ be $\{x_1,\dotsc, x_5\}$ as shown in Figure \ref{4 reduction at work} (a). Perform a 4-reduction on $[x_1x_5,x_2,x_3,x_4]$ to construct a 4-manifold $M'$. Recall that each 4-reduction consists of taking a 4-torus and identifying its generators with those in the 4-reduction. 

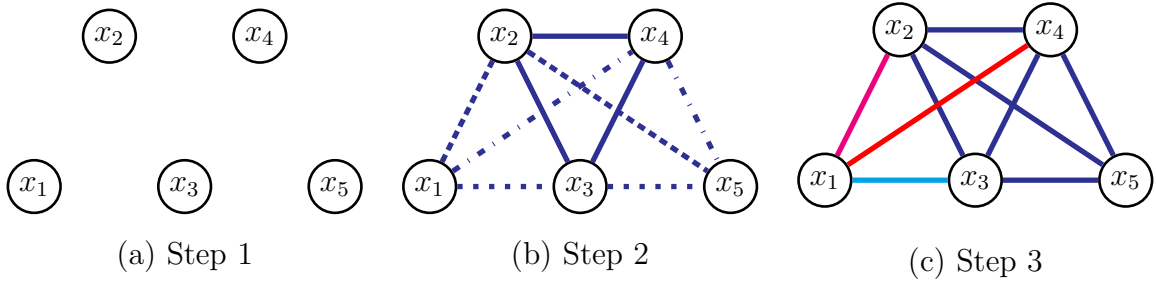
\begin{figure}[h]
\begin{center}
\begin{minipage}{.31\linewidth}
\centering
\begin{tikzpicture}[scale = 2]
\tikzstyle{every node}=[draw, shape=circle, inner sep=.5mm, minimum size=.7cm, line width=1pt];
\path (0,0) node (x1) {$x_1$};
\path (.5,1) node (x2) {$x_2$};
\path (1,0) node (x3) {$x_3$};
\path (1.5,1) node (x4) {$x_4$};
\path (2,0) node (x5) {$x_5$};
\end{tikzpicture}\\
\medskip
(a) Step 1
\end{minipage}
\begin{minipage}{.31\linewidth}
\centering
\begin{tikzpicture}[scale = 2]
\tikzstyle{every node}=[draw, shape=circle, inner sep=.5mm, minimum size=.7cm, line width=1pt];
\path (0,0) node (x1) {$x_1$};
\path (.5,1) node (x2) {$x_2$};
\path (1,0) node (x3) {$x_3$};
\path (1.5,1) node (x4) {$x_4$};
\path (2,0) node (x5) {$x_5$};
\draw[Blue, line width = 2pt] (x2) -- (x3) -- (x4) -- (x2);
\draw[Blue, loosely dashdotted, line width = 2pt] (x1) -- (x4) -- (x5);
\draw[Blue, densely dashed, line width = 2pt] (x1) -- (x2) -- (x5);
\draw[Blue, loosely dotted, line width = 2pt] (x1) -- (x3) -- (x5);		
\end{tikzpicture}\\
\medskip
(b) Step 2
\end{minipage}
\begin{minipage}{.31\linewidth}
\centering
\begin{tikzpicture}[scale = 2]
\tikzstyle{every node}=[draw, shape=circle, inner sep=.5mm, minimum size=.7cm, line width=1pt];
\path (0,0) node (x1) {$x_1$};
\path (.5,1) node (x2) {$x_2$};
\path (1,0) node (x3) {$x_3$};
\path (1.5,1) node (x4) {$x_4$};
\path (2,0) node (x5) {$x_5$};
\draw[Blue, line width = 2pt] (x2) -- (x3) -- (x4) -- (x2)  (x4) -- (x5) (x2) -- (x5) (x3) -- (x5);
\draw[red, line width = 2pt] (x1) -- (x4);
\draw[RubineRed,  line width = 2pt] (x1) -- (x2) ;
\draw[Cerulean, line width = 2pt] (x1) -- (x3) ;	
\end{tikzpicture}\\
\bigskip
(c) Step 3
\end{minipage}
\end{center}
\caption{A graph showing the path of edges created by the 4-reduction $[x_1x_5,x_2,x_3,x_4]$}
\label{4 reduction at work}
\end{figure}

Let us look a representation of the graph of $\pi_1(M')$ in Figure \ref{4 reduction at work} (b). The solid lines indicate the existence of the commutator relations between $x_2$, $x_3$, and $x_4$ given by the 4-reduction. The remaining three relations are $[x_1x_5,x_2]=1$, $[x_1x_5,x_3]=1$, and $[x_1x_5,x_4]=1$. We will refer to these types of commutator relations as \emph{surface-like relations} and we can view them as products of commutators. We consider $[x_1x_5,x_2]=1$ and $[x_1,x_2][x_5,x_2]=1$ equivalent relations since they represent the same commutator information. More formally, they normally generate the same subgroup. In the same way, we consider $[x_1x_5,x_3]=1$ equivalent to $[x_1,x_3][x_5,x_3]=1$ and $[x_1x_5,x_4]=1$ equivalent to $[x_1,x_4][x_5,x_4]=1$. 

Graphically, we will represent surface-like relations by dashed or dotted lines, as we did in Section \ref{surface graphs} with the surface relation of $\pi_1(\Sigma_2)$. Since we have three such relations, we can resemble them by three different styles of lines in the graph: dashed, dotted, and a combination of dashes and dots. 

Now perform surgery to induce the following relations: {$[x_1,x_2]=1$, $[x_1,x_3]=1$, and $[x_1,x_4]=1$. Because of the surface-like relations induced by the 4-reduction, we get three relations for free: $[x_2,x_5]=1$, $[x_3,x_5]=1$, and $[x_4,x_5]=1$. The resulting $\pi_1$ graph is in Figure \ref{4 reduction at work} (c). 

Consider a similar 4-reduction beginning with a manifold $M=\#6 (S^1 \times S^3)$, with generators $\{x_1,\dotsc, x_6\}$, as shown in Figure \ref{4 reduction at work 2} (a). Perform the following 4-reduction: $[x_1x_3x_6, x_2,x_4,x_5]$. As shown in Figure \ref{4 reduction at work 2} (b), the solid lines represent the three commutator relations between $x_2$, $x_4$, and $x_5$. The remaining three relations from the 4-reduction can be represented by the surface-like relations 
\begin{eqnarray*}
&[x_1,x_2][x_2,x_3][x_2,x_6]=1& \\
&[x_1,x_4][x_3,x_4][x_4,x_6]=1& \\
&[x_1,x_5][x_3,x_5][x_5,x_6]=1& 
\end{eqnarray*}

\noindent and are demonstrated by dashed, dotted, and dash-dotted lines in Figure \ref{4 reduction at work 2} (b). The following commutator surgeries result in a 4-manifold with the $\pi_1$ graph in Figure \ref{4 reduction at work 2} (c):
\begin{equation*}
[x_1,x_2], [x_2,x_3], [x_1,x_4], [x_4,x_6], [x_3,x_5], [x_5,x_6].
\end{equation*}

\begin{figure}[h]
\begin{center}
\begin{minipage}{.31\linewidth}
\centering
\begin{tikzpicture}[scale = 2]
\tikzstyle{every node}=[draw, shape=circle, inner sep=.5mm, minimum size=.7cm, line width=1pt];
\path (0,0) node (x1) {$x_1$};
\path (.5,1) node (x2) {$x_2$};
\path (1,2) node (x3) {$x_3$};
\path (1,0) node (x4) {$x_4$};
\path (1.5,1) node (x5) {$x_5$};
\path (2,0) node (x6) {$x_6$};
\end{tikzpicture}\\
\medskip
(a) Step 1
\end{minipage}
\begin{minipage}{.31\linewidth}
\centering
\begin{tikzpicture}[scale = 2]
\tikzstyle{every node}=[draw, shape=circle, inner sep=.5mm, minimum size=.7cm, line width=1pt];
\path (0,0) node (x1) {$x_1$};
\path (.5,1) node (x2) {$x_2$};
\path (1,2) node (x3) {$x_3$};
\path (1,0) node (x4) {$x_4$};
\path (1.5,1) node (x5) {$x_5$};
\path (2,0) node (x6) {$x_6$};
\draw[Blue, line width = 2pt] (x2) -- (x4) -- (x5) -- (x2);
\draw[Blue, loosely dashdotted, line width = 2pt] (x1) -- (x5) (x3) -- (x5) -- (x6);
\draw[Blue, densely dashed, line width = 2pt] (x1) -- (x2) -- (x3) (x2) -- (x6);
\draw[Blue, loosely dotted, line width = 2pt] (x1) -- (x4) -- (x6) (x4) -- (x3);	
\end{tikzpicture}\\
\medskip
(b) Step 2
\end{minipage}
\begin{minipage}{.31\linewidth}
\centering
\begin{tikzpicture}[scale = 2]
\tikzstyle{every node}=[draw, shape=circle, inner sep=.5mm, minimum size=.7cm, line width=1pt];
\path (0,0) node (x1) {$x_1$};
\path (.5,1) node (x2) {$x_2$};
\path (1,2) node (x3) {$x_3$};
\path (1,0) node (x4) {$x_4$};
\path (1.5,1) node (x5) {$x_5$};
\path (2,0) node (x6) {$x_6$};
\draw[Blue, line width = 2pt] (x2) -- (x4) -- (x5) -- (x2)  (x2) -- (x6) (x1) -- (x5) (x3) -- (x4);
\draw[red, line width = 2pt] (x3) -- (x5) -- (x6);
\draw[RubineRed,  line width = 2pt] (x1) -- (x2) (x2) -- (x3) ;
\draw[Cerulean, line width = 2pt] (x1) -- (x4) -- (x6);	
\end{tikzpicture}\\
\medskip
(c) Step 3
\end{minipage}
\end{center}
\caption{A graph showing the path of edges created by the 4-reduction $[x_1x_3x_6,x_2,x_4,x_5]$}
\label{4 reduction at work 2}
\end{figure}
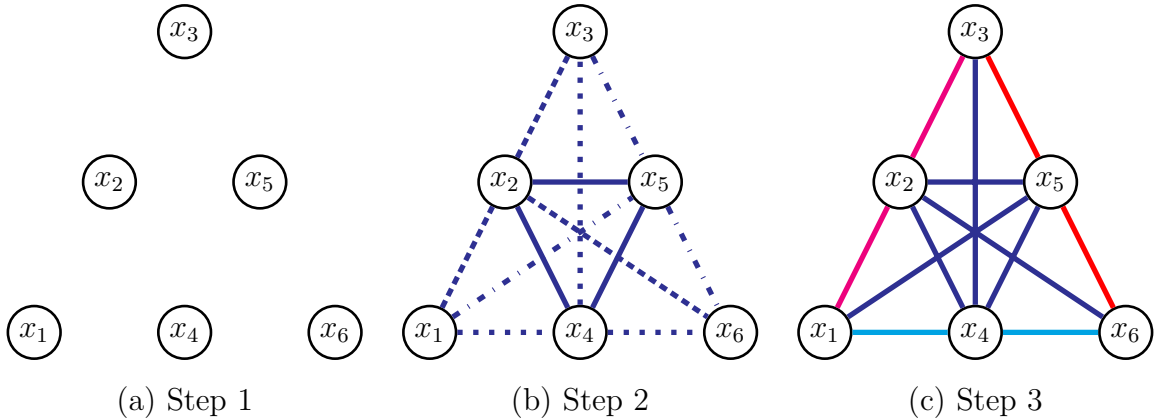

These are two examples of 4-reductions of the form $[a,b,c,de\dotsc]$, where $a$, $b$, and $c$ are generators of $\pi_1$ and the fourth element is a product of generators. If many of these types of 4-reductions are required in the construction of a realizing 4-manifold, it may be useful to highlight the three commutator relations between $a$, $b$, and $c$. In the two graphs below, we can \textbf{\textcolor{Magenta}{shade}} the area of the triangle bounded by the edges between vertices corresponding to $a$, $b$, and $c$. 
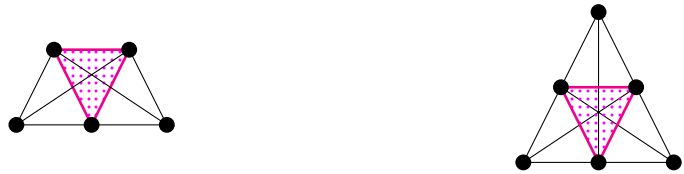
\begin{center}
\begin{minipage}{.4\linewidth}
\centering
\begin{tikzpicture}
\foreach \i in {1,...,3}{\path (\i-1,0) coordinate (a\i);}
\foreach \i in {1,2}{\path (\i-0.5,1) coordinate (b\i);}
\filldraw[Magenta, pattern=dots, pattern color=Magenta, line width=1pt] (a2)--(b1)--(b2)--cycle;
\draw (b1)--(a1)--(a3)--(b2) (a1)--(b2) (a3)--(b1);
\foreach \i in {1,2}{\fill (a\i) circle (3pt); \fill (b\i) circle (3pt);}
\fill (a3) circle (3pt);	
\end{tikzpicture}
\end{minipage}
\begin{minipage}{.4\linewidth}
\centering
\begin{tikzpicture}
\foreach \i in {1,...,3}{\path (\i-1,0) coordinate (a\i);}
\foreach \i in {1,2}{\path (\i-.5,1) coordinate (b\i);}
\path (1,2) coordinate (c1);
\draw (a1)--(c1)--(a3)--cycle (a1)--(b2) (a3)--(b1) (c1)--(a2);
\filldraw[Magenta, pattern=dots, pattern color= Magenta, line width = 1pt] (b1)--(b2)--(a2)--cycle;
\foreach \i in {a,b,c}{\fill (\i1) circle (3pt);}
\foreach \i in {a,b}{\fill (\i2) circle (3pt);}
\foreach \i in {a}{\fill (\i3) circle (3pt);}
\end{tikzpicture}
\end{minipage}
\vskip .2 in
\end{center}
This triangle represents the face that is shared by all 4-cliques whose fourth vertex is represented in the product of the last element of the 4-reduction. This shading technique will be useful in Section \ref{hexagonal grids} when we consider graphs of many 4-cliques attached along triangles.

Note that 4-reductions are not limited to the form $[a,b,c,de\dotsc]$ above. Each entry may involve many products of generators. The two examples given in this section are included to illustrate the use of 4-reductions for graphs of certain \RAAGs\ we discuss in Section \ref{hexagonal grids}.

%%%%%%%%%%%%%%%%%%%%%%%%%%%%%%%%%%%%%%%%%%%%%%%%%%%%%%%%%%%%%%%%%%%%%%%%%%%%%%%%%%%%%%%%%%%%%%%%%%%%%%%%%%%%%%%%%%%%
\subsection{Surgery on dual spheres}

Consider the following construction of the connected sum of three 4-tori, each with $\pi_1$ generated by $\{x_1,x_2,x_3,x_4\}$, $\{y_1,y_2,y_3,y_4\}$, and $\{z_1, z_2, z_3, z_4\}$, as shown in Figure \ref{threetori}. 

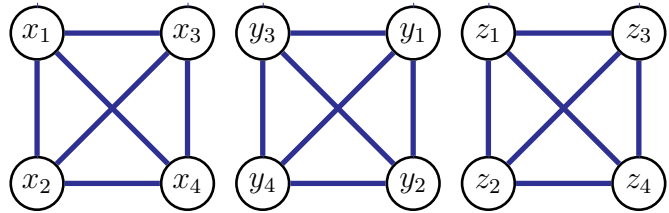
\begin{figure}[h]
\begin{center}
\begin{tikzpicture}[scale=2]
\tikzstyle{every node}=[draw, shape=circle, inner sep=.5mm, minimum size=.7cm, line width=1pt];
\path (0,1) node (x1) {$x_1$};
\path (0,0) node (x2) {$x_2$};
\path (1,1) node (x3) {$x_3$};
\path (1,0) node (x4) {$x_4$};
\path (1.5,1) node (y3) {$y_3$};
\path (1.5,0) node (y4) {$y_4$};
\path (2.5,1) node (y1) {$y_1$};
\path (2.5,0) node (y2) {$y_2$};
\path (3,1) node (z1) {$z_1$};
\path (3,0) node (z2) {$z_2$};
\path (4,1) node (z3) {$z_3$};
\path (4,0) node (z4) {$z_4$};
\foreach \i in {1,...,4}{\foreach \j in {\i,...,4}{ \draw[Blue, line width = 2pt] (x\i) -- (x\j);}}
\foreach \i in {1,...,4}{\foreach \j in {\i,...,4}{ \draw[Blue, line width = 2pt] (y\i) -- (y\j);}}	
\foreach \i in {1,...,4}{\foreach \j in {\i,...,4}{ \draw[Blue, line width = 2pt] (z\i) -- (z\j);}}
\end{tikzpicture} 
\caption{A graph representing $\pi_1(T^4\#T^4\#T^4)$}
\label{threetori}
\end{center}
\end{figure}

After surgery to identify the generators $x_3$ with $y_3$ as well as $x_4$ with $y_4$, we can find an embedded 2-sphere in the resulting 4-manifold.  View the first 4-torus as the product of two 2-tori, $x_1 \times x_2$ and $x_3 \times x_4$, and the second 4-torus as the product of $y_1\times y_2$ and $y_3\times y_4$. We can view the connected sum ambiently and after the identification surgeries, we see a 2-sphere embedded in the 4-manifold. 

Similarly, identifying $y_1$ with $z_1$ and $y_2$ with $z_2$ via surgery creates a second embedded 2-sphere. Because we can initially view the middle 4-torus as the product of two 2-tori, $y_1 \times y_2$ and $y_3\times y_4$, which intersect in exactly 1 point, so do the two embedded 2-spheres. %(See Figure \ref{EmbeddedSpheres}.) 
We will refer to such a pair of embedded 2-spheres intersecting in this way as a pair of \emph{dual 2-spheres}. 

We have already seen in Lemma \ref{surgeries} that performing surgery to identify generators of $\pi_1$ does not change $b_2$, and performing surgery to induce a commutator relation increases $b_2$ by two. By the next lemma, we can surger out a pair of dual 2-spheres without changing the fundamental group and also \emph{decrease} $b_2$ by two. 
\begin{lemma}
Suppose in 4-manifold $M$ there exist two 2-spheres intersecting exactly once with at least one embedded with trivial normal bundle. Then it is possible to remove both spheres via surgery without changing the fundamental group of $M$ and also decrease $b_2(M)$ by 2.
\label{sphere surgery}
\end{lemma}

\begin{proof}
Suppose $S$ is an embedded 2-sphere in a 4-manifold $M$, with self-intersection zero. Let $M' = M - S \times B^2$. Then $M$ is built from $M'$ by adding a 2-handle to a nullhomotopic curve and then adding a 4-handle. Neither handle addition changes $\pi_1$. Let $M_S$ be the resulting manifold after surgery on $S$. $M_S$ is built from $M'$ by adding a 3-handle and a 4-handle, thus $\pi_1$ remains unchanged. The homology classes of both $S$ and the second 2-sphere are killed by the surgery, thus the rank of $H_2(M;\Q)$ decreases by two. 
\end{proof}

\begin{remark}
Note that this lemma gives a slightly stronger result than what we need, since it allows for one sphere to be immersed. In practice, however, we will always use this lemma to surger out a pair of embedded dual 2-spheres.
\end{remark}

This is the only technique we will use to decrease $b_2$ in certain 4-manifolds. Moreover, for 4-manifolds with $\pi_1$ graphs of 4-cliques with more than one pair of dual spheres, in many cases we can surger many if not all pairs of embedded dual spheres to minimize $b_2$. 

\begin{example} 
Consider the setup of the following row of $k$ 4-cliques attached edge to edge, as in the graph below: 
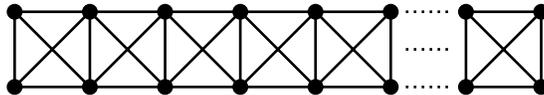
\begin{center}
\begin{tikzpicture}
\foreach \i in {0,...,7}
{
\path (\i*1,0) coordinate (v\i);
\fill (v\i) circle (3pt);
\path (\i*1,1) coordinate (w\i);
\fill (w\i) circle (3pt);
\draw[line width = 1pt] (v\i) -- (w\i);
}
\draw[line width = 1pt] (v0)--(v5) (v6)--(v7) (w0) -- (w5) (w6)--(w7) (w6)--(v7) (v6)--(w7) (w0) -- (v1) -- (w2) -- (v3) -- (w4) -- (v5) (v0) -- (w1) -- (v2) -- (w3) -- (v4) -- (w5);
\draw[line width = 1pt, dotted] (5.2,.5) -- (5.8,.5) (5.2,0) -- (5.8,0) (5.2,1) -- (5.8,1);
\end{tikzpicture}
\end{center}

\noindent Just as before, the way to construct a 4-manifold with minimum $b_2$ is to start with the connected sum of $k$ 4-tori, and perform surgery to identify the appropriate generators of $\pi_1$. Each pair of surgeries identifying the generators of one 4-tori with another creates an embedded 2-sphere, and each sphere intersects one before it and one after it (except the first and last sphere, respectively, where they intersect a 2-torus each). Thus for $k$ 4-cliques as shown above, we have a chain of $k-1$ 2-spheres, with a 2-torus on each end. We can make $\lfloor \frac{k-1}{2}\rfloor$ pairs of dual spheres disjoint by handle slides, and thus perform $\lfloor \frac{k-1}{2}\rfloor$ surgeries on these dual sphere pairs to decrease $b_2$. 
\end{example}

The following lemma will be useful in the following section when we calculate $h$ for certain examples of \RAAGs.
\begin{lemma}
Suppose we have the following graph representing a \RAAG\ $G$:
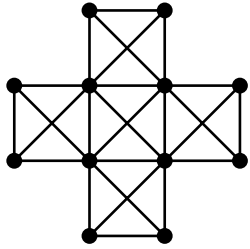
\begin{center}
\begin{tikzpicture}
\foreach \i in {0,...,3}{
    \path (\i*1,0) coordinate (v\i); \fill (v\i) circle (3pt);
    \path (\i*1,1) coordinate (w\i); \fill (w\i) circle (3pt);
    \draw[line width = 1pt] (v\i) -- (w\i);
}
\path (1,2) coordinate (x1); \fill (x1) circle (3pt);
\path (2,2) coordinate (x2); \fill (x2) circle (3pt);
\path (1,-1) coordinate (y1); \fill (y1) circle (3pt);
\path (2,-1) coordinate (y2); \fill (y2) circle (3pt);
\draw[line width = 1pt] (w1) -- (x1) -- (x2) -- (w2) -- (x1) (w1) -- (x2) (v1) -- (y1) -- (y2) -- (v2) -- (y1)  (y2) -- (v1) (v0) -- (v3) (w0) -- (w3) (w0) -- (v1) -- (w2) -- (v3) (v0) -- (w1) -- (v2) -- (w3);
\end{tikzpicture}
\end{center}
Consider the 4-manifold, $M$, constructed by the connected sum of five 4-tori with identification surgeries so that the fundamental group has the associated graph above. Then there are two pairs of dual 2-spheres in $M$.
Moreover, we can perform the identification surgeries in such away to make the dual sphere pairs disjoint from each other and thus perform surgery on each pair, thereby decreasing $b_2(M)$ by four.
\label{cross}
\end{lemma}

\begin{proof}
As stated in the lemma, begin with the connected sum of five 4-tori, each generated by $\{x_1,x_2,x_3,x_4\}$, $\{a_1,a_2,a_3,a_4\}$, $\{b_1,b_2,b_3,b_4\}$, $\{c_1,c_2,c_3,c_4\}$, and $\{d_1,d_2,d_3,d_4\}$. Denote by $X$ the ``middle" 4-torus generated by the $x_i$, and view $X$ as $[0,1]^4/ (0 \sim 1)$. Consider the following identifications of four submanifolds of $X$, where $x_i'$ is a push-off of $x_i$:
\begin{eqnarray*}
x_1 \times x_2 &=& [0,1]\times [0,1] \times \{0\} \times \{0\} / \sim \\
x_3 \times x_4 &=& \{0\} \times \{0\} \times [0,1] \times [0,1] / \sim\\
x_1' \times x_3' &=& [0,1] \times \left\{\frac 12\right\} \times [0,1] \times \left\{\frac12\right\} / \sim \\
x_2' \times x_4' &=& \left\{\frac 12\right\} \times [0,1] \times \left\{\frac 12\right\} \times [0,1] / \sim.
\end{eqnarray*}
We can see that $(x_1 \times x_2)\cap (x_3\times x_4) = (0,0,0,0)$ and $(x_1' \times x_3')\cap (x_2'\times x_4') = (\frac 12, \frac 12, \frac 12, \frac 12)$, but the intersections between the other pairs are empty. Thus we can perform the following identifications via surgery: $b_1=x_1$, $b_2=x_2$, $d_3 = x_3$, $d_4=x_4$, $a_1=x_1'$, $a_3 = x_3'$, $c_2 = x_2'$, and $c_4=x_4'$. After the identification surgeries, we get two distinct strings of three 4-cliques, representing the existence of two disjoint pairs of dual spheres. We can perform surgery on both of these dual sphere pairs, decreasing $b_2$ by four. 
\end{proof}

%%%%%%%%%%%%%%%%%%%%%%%%%%%%%%%%%%%%%%%%%%%%%%%%%%%%%%%%%%%%%%%%%%%%%%%%%%%%%%%%%%%%%%%%%%%%%%%%%%%%%%%%%%%%%%%%%%%%
%%%%%%%%%%%%%%%%%%%%%%%%%% EXAMPLES OF COHOMOLOGICALLY MINIMIMAL RAAGS %%%%%%%%%%%%%%%%%%%%%%%%%%%%%%%%%%%%%%%%%%%%%
%%%%%%%%%%%%%%%%%%%%%%%%%%%%%%%%%%%%%%%%%%%%%%%%%%%%%%%%%%%%%%%%%%%%%%%%%%%%%%%%%%%%%%%%%%%%%%%%%%%%%%%%%%%%%%%%%%%%
\section[Examples]{Examples of \cm\ \RAAGs} \label{infinitefamilies}

We have already seen that a \RAAG\ $G$ with trivial $H^4(G)$ is \cm. In the first three examples we focus on \RAAGs\ with trivial $H^5(G)$. We begin with graphs made up of multiple 4-cliques attached along edges, and continue with attachments along triangles, or faces. We also assume all graphs are connected and every edge belongs to at least one 4-clique, due to Theorems \ref{freeproducts} and \ref{breakdown}. Also, it should be noted that any reference to the \emph{radical} or \emph{(maximum) isotropic subspace} is in reference to those of the form (\ref{formmod2}).

%%%%%%%%%%%%%%%%%%%%%%%%%%%%%%%%%%%%%%%%%%%%%%%%%%%%%%%%%%%%%%%%%%%%%%%%%%%%%%%%%%%%%%%%%%%%%%%%%%%%%%%%%%%%%%%%%%%%
\subsection{Grids of 4-cliques sharing edges}\label{gridsection}

Consider the family of \RAAGs\ that have associated graphs composed of 4-cliques attached edge-to-edge in a grid pattern, aligned in rows and columns so that the vertices lie on a $\Z^2$ lattice. Figure \ref{planar} shows some examples. We refer to these graphs as members of the \emph{\Grid family}. 

\begin{figure}[h]
\begin{center}
\begin{tikzpicture}
\foreach \i in {0,...,4} {
	\path (\i,0) coordinate (v\i); \fill (v\i) circle (3pt);
	\path (\i,1) coordinate (w\i); \fill (w\i) circle (3pt);
	\draw[line width = 1pt] (v\i) -- (w\i);}
\foreach \i in {0,1,2} {
	\path (\i,-1) coordinate (t\i); \fill (t\i) circle (3pt);
	\path (\i,-2) coordinate (x\i); \fill (x\i) circle (3pt);
	\draw[line width = 1 pt] (v\i) -- (x\i);}
\draw[line width = 1pt] (v0)--(v4) (w0) -- (w4) (w0) -- (v1) -- (w2) -- (v3) -- (w4) (v0) -- (w1) -- (v2) -- (w3) -- (v4) (v0) -- (x2) -- (x0) -- (v2) (t0) -- (x1) -- (t2) -- (v1) -- (t0) -- (t2);
\end{tikzpicture} 
\qquad 
\begin{tikzpicture}
\foreach \i in {1,2,3} {
	\path (\i, 0) coordinate (x\i); \fill (x\i) circle (3pt);
	\path (\i, 1) coordinate (y\i); \fill (y\i) circle (3pt);
	\path (\i, 2) coordinate (z\i); \fill (z\i) circle (3pt);
	\path (\i, 3) coordinate (v\i); \fill (v\i) circle (3pt);
	\draw[line width = 1pt] (x\i) -- (v\i);}
\path (4,0) coordinate (x4); \fill (x4) circle (3pt);
\path (4,3) coordinate (v4); \fill (v4) circle (3pt);
\path (0,1) coordinate (y0); \fill (y0) circle (3pt);
\path (4,1) coordinate (y4); \fill (y4) circle (3pt);
\path (0,2) coordinate (z0); \fill (z0) circle (3pt);
\path (4,2) coordinate (z4); \fill (z4) circle (3pt);
\draw[line width = 1pt] (x1) -- (x4) (y0) -- (y4) (z0) -- (z4) (v1) -- (v4) (y0) -- (z0) -- (y1) -- (v3) -- (z4) -- (v4) -- (x1) (y0) -- (v2) -- (z3) (x4) -- (y4) -- (x3) -- (z1) (x4) -- (v1) (y1) -- (x2) -- (y3);
\end{tikzpicture} 
\qquad
\begin{tikzpicture}
\foreach \i in {1,...,4} {
	\path (\i,0) coordinate (x\i); \fill (x\i) circle (3pt);
	\path (\i,1) coordinate (y\i); \fill (y\i) circle (3pt);
	\path (\i,2) coordinate (z\i); \fill (z\i) circle (3pt);
	\path (\i,3) coordinate (v\i); \fill (v\i) circle (3pt);
	\draw[line width = 1pt] (x\i) -- (v\i);}
\foreach \i in {x,y,z,v} {\draw[line width = 1pt] (\i1) -- (\i4);}
\draw[line width = 1pt] (x1) -- (y2) (v1) -- (z2) (z3) -- (v4) (y3) -- (x4) (y1) -- (v3) -- (z4) -- (x2) -- (y1) 
(z1) -- (v2) -- (y4) -- (x3) -- (z1);
\end{tikzpicture} \\
(a) \hspace{2 in} (b) \hspace{1.5 in} (c) 
\end{center}
\caption{Examples of graphs in the Grid Family}
\label{planar}
\end{figure}
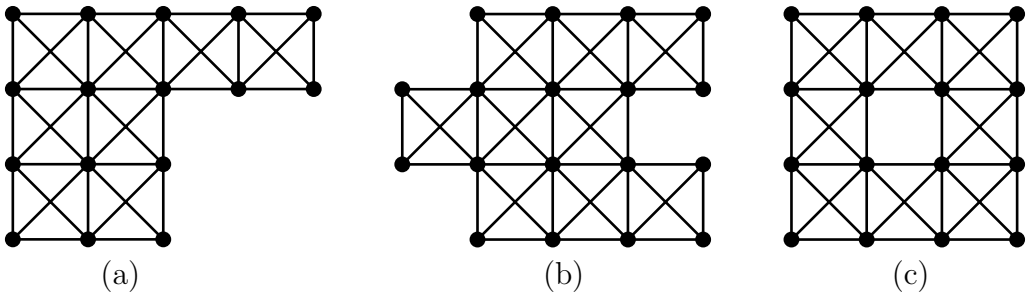

\begin{theorem}
Let $G$ be a \RAAG\ with an associated graph belonging to the \Grid family. Then $G$ is \cm.
\label{theorem grids}
\end{theorem}

\begin{proof}
Let $\Gamma$ be the graph associated to $G$ and let $k$ be the number of 4-cliques in $\Gamma$. Recall that Proposition \ref{hnewlowerbound} gives $2b_2(G) - m_2(G)\leq h(G).$ We can view this lower bound as $b_2(G)$ plus the minimum dimension of the radical. Each 4-clique has 6 edges, so clearly $b_2(G)$ is equal to 6k minus the total number of shared edges in $\Gamma$. We will show that the realizing manifold $M$ has $b_2(M)$ equal to $6k$ minus twice the number of possible dual sphere surgeries. Thus, to prove the theorem, we can show that the
\begin{equation*}
\textrm{$\#$ of shared edges} - \dim(\textrm{minimum radical}) = 2(\textrm{$\#$ possible dual 2-sphere surgeries}).
\end{equation*} 

Fortunately, it suffices to show the above equation holds separately for each linear string of 4-cliques in the graph. That is, we can consider each row and each column of $\Gamma$ separately. This is because the number of shared edges, number of basis elements of the minimum radical, and number of dual sphere surgeries in a single string of 4-cliques are additive and do not conflict with the count for other strings of 4-cliques in other rows and columns of $\Gamma$: the items counted in a horizontal string correspond only to vertical edges and vertical pairings of vertices in the string, and the items we are counting in a vertical string correspond only to horizontal edges and horizontal pairings of vertices in the string. Therefore, the separate counts will not conflict with each other. Further, Lemma \ref{cross} asserts that each of the dual sphere surgeries are possible when we consider all of $\Gamma$.

However, in splitting up the proof we must take care to use the same choice of $\alpha \in H_4(G,\Z_2)$. Fortunately, we may assume that $\alpha=\alpha_1 + \dotsc + \alpha_k$ ($c_i=0$ for all $i$) minimizes the nullity. If instead $c_i=0$ for some $i$ so that $\alpha = \alpha_1 + \dotsc + \alpha_{i-1} + \alpha_{i+1} + \dotsc +\alpha_k$, the dimension of the radical can only increase. The generator $\alpha_i$ represents a choice of the $i$th 4-clique in the graph. Every 4-clique lies in exactly one row and one column. If the $i$th 4-clique is part of a string of an even number of 4-cliques (in either direction), the elements that would be part of the radical had $c_i$ been 1 would no longer cause the form to be nondegenerate, so the count of the dimension of the radical will decrease by \emph{at most} two. However, by construction of the graph, at least two edges in every 4-clique are not shared by any other 4-clique (the two diagonal edges). The unshared edges of the $i^th$ 4-clique are now basis elements of the radical. This causes the count of the dimension to increase by \emph{at least} two. 

Now, consider a string of $\ell$ connected 4-cliques. The number of shared edges is $\ell-1$. As we saw in Example \ref{m2 by hand}, there is an element of the radical if and only if $\ell$ is even. Thus for this string, the left-hand side of the equation is $\ell - 2$ if $\ell$ is even, and $\ell -1$ if $\ell$ is odd. By Lemma \ref{sphere surgery}, we can perform $\lfloor \frac{\ell-2}{2}\rfloor$ dual 2-sphere surgeries without changing $\pi_1$. Thus, the right-hand side of the equation is $\ell-2$ if $\ell$ is even, and $\ell-1$ if $\ell$ is odd.  
\end{proof}

%%%%%%%%%%%%%%%%%%%%%%%%%%%%%%%%%%%%%%%%%%%%%%%%%%%%%%%%%%%%%%%%%%%%%%%%%%%%%%%%%%%%%%%%%%%%%%%%%%%%%%%%%%%%%%%%%%%%
\subsection{4-cliques that share faces}

Next we will consider graphs of $k$ 4-cliques that share \emph{faces}, or triangles. First, consider the family of graphs represented by strings of $k$ 4-cliques as exemplified by the graphs in Figure \ref{face sharing lines}. In (a), $k=2$; in (b), $k=3$; in (c), $k=4$; in (d), $k=5$. We call a graph of this form a member of the \String family. 

\begin{figure}[h]
\begin{center}
\begin{tikzpicture}
\foreach \i in {1,...,3}{\path (\i+1,0) coordinate (x\i);}
\foreach \i in {1,...,2}{\path (\i+1.5,1) coordinate (y\i);}
\draw[line width =1pt] (x1)--(y1)--(x2)--(y2)--(x3) (x1)--(x3) (y1)--(y2) (x1)--(y2) (y1)--(x3);
\foreach \i in {1,...,2}{\fill (x\i) circle (3pt); \fill (y\i) circle (3pt);}
\fill (x3) circle (3pt);
\end{tikzpicture}
\qquad
\begin{tikzpicture}
\foreach \i in {1,...,3}{
	\path (\i+1,0) coordinate (x\i); 
	\path (\i+1.5,1) coordinate (y\i);}
\draw[line width =1pt] (x1)--(y1)--(x2)--(y2)--(x3)--(y3) (x1)--(x3) (y1)--(y3) (x1)--(y2) (x2)--(y3) (y1)--(x3);
\foreach \i in {1,...,3}{\fill (x\i) circle (3pt); \fill (y\i) circle (3pt);}
\end{tikzpicture}
\qquad
\begin{tikzpicture}
\foreach \i in {1,...,4}{\path (\i+1,0) coordinate (x\i);}
\foreach \i in {1,...,3}{\path (\i+1.5,1) coordinate (y\i);}
\draw[line width =1pt] (x1)--(y1)--(x2)--(y2)--(x3)--(y3)--(x4) (x1)--(x4) (y1)--(y3) (x1)--(y2) (x2)--(y3) (y1)--(x3) (y2) -- (x4);
\foreach \i in {1,...,3}{\fill (x\i) circle (3pt); \fill (y\i) circle (3pt);}
\fill (x4) circle (3pt);
\end{tikzpicture}
\qquad
\begin{tikzpicture}
\foreach \i in {1,...,4}{
	\path (\i+1,0) coordinate (x\i); 
	\path (\i+1.5,1) coordinate (y\i);}
\draw[line width =1pt] (x1)--(y1)--(x2)--(y2)--(x3)--(y3)--(x4)--(y4) (x1)--(x4) (y1)--(y4) (x1)--(y2) (x2)--(y3) (x3) -- (y4) (y1)--(x3) (y2) -- (x4);
\foreach \i in {1,...,4}{\fill (x\i) circle (3pt); \fill (y\i) circle (3pt);}
\end{tikzpicture}\\
\centering{\hspace{.1 in} (a) \hspace{.9 in} (b) \hspace{1.35 in} (c) \hspace{1.35 in} (d) \hspace{.40 in}}
\caption{Graphs in the \String family}
\label{face sharing lines}
\end{center}
\end{figure}
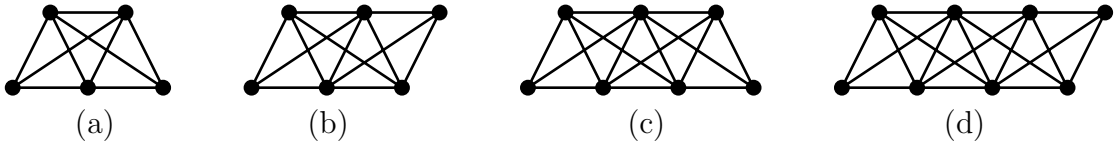

\begin{theorem}
Let G be a \RAAG\ whose associated graph is in the \String family. For $k=b_4(G)$, then
$h(G) = 3k+6$ if $k$ is even and $3k+5$ if $k$ is odd.
In particular, $G$ is \cm.
\end{theorem}

\begin{proof}
We will bound $h(G)$ below by calculating the dimension of the maximum isotropic subspace in terms of $k$. We will denote this dimension by $d$. Figure \ref{face sharing lines iso} \textbf{\textcolor{blue}{highlights}} the edges of the graphs in Figure \ref{face sharing lines} which make up a maximum isotropic subspace in each case. When $k=2$ (Figure \ref{face sharing lines iso} (a)), $d=6$: two edges line the bottom of the graph, two edges are on either end of the string, and two are long diagonal edges. When $k=3$ (Figure \ref{face sharing lines iso} (b)), $d=7$: two edges line the bottom of the graph, two are end edges, and three are long diagonal edges. When $k=4$ (Figure \ref{face sharing lines iso} (c)), $d=9$: three edges line the bottom of the graph, two are end edges, and four are long diagonal edges. Following the pattern, we see that for general $k$, $d=(\lfloor\frac k 2\rfloor+1) + 2 + k$: $\lfloor\frac k2\rfloor+1$ edges line the bottom of the graph, two are end edges, and $k$ are long diagonal edges. Thus when $k$ is even, $d=\frac12(3k+6)$ and when $k$ is odd, $d=\frac12(3k+5)$. Twice this dimension $d$ yields the necessary lower bound.

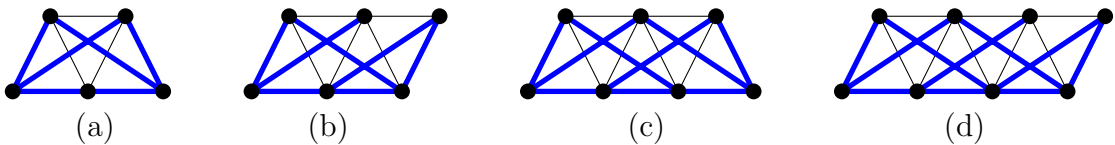
\begin{figure}[h]
\begin{center}
\begin{tikzpicture}
\foreach \i in {1,...,3}{\path (\i+1,0) coordinate (x\i);}
\foreach \i in {1,...,2}{\path (\i+1.5,1) coordinate (y\i);}
\draw(y1)--(x2)--(y2) (y1)--(y2);
\draw[blue, line width = 2pt] (x1)--(x3) (x1)--(y1) (y2)--(x3) (x1)--(y2) (y1)--(x3);
\foreach \i in {1,...,2}{\fill (x\i) circle (3pt); \fill (y\i) circle (3pt);}
\fill (x3) circle (3pt);
\end{tikzpicture}
\qquad
\begin{tikzpicture}
\foreach \i in {1,...,3}{
	\path (\i+1,0) coordinate (x\i); 
	\path (\i+1.5,1) coordinate (y\i);}
\draw (y1)--(x2)--(y2)--(x3) (y1)--(y3);
\draw[blue, line width=2pt] (x1)--(x3) (x1)--(y1) (x3)--(y3) (x1)--(y2) (x2)--(y3) (y1)--(x3);
\foreach \i in {1,...,3}{\fill (x\i) circle (3pt); \fill (y\i) circle (3pt);}
\end{tikzpicture}
\qquad
\begin{tikzpicture}
\foreach \i in {1,...,4}{\path (\i+1,0) coordinate (x\i);}
\foreach \i in {1,...,3}{\path (\i+1.5,1) coordinate (y\i);}
\draw (y1)--(x2)--(y2)--(x3)--(y3) (y1)--(y3);
\draw[blue, line width = 2pt](x1)--(x4) (x1)--(y1) (y3)--(x4) (x1)--(y2) (x2)--(y3) (y1)--(x3) (y2) -- (x4);
\foreach \i in {1,...,3}{\fill (x\i) circle (3pt); \fill (y\i) circle (3pt);}
\fill (x4) circle (3pt);
\end{tikzpicture}
\qquad
\begin{tikzpicture}
\foreach \i in {1,...,4}{
	\path (\i+1,0) coordinate (x\i); 
	\path (\i+1.5,1) coordinate (y\i);}
\draw (y1)--(x2)--(y2)--(x3)--(y3)--(x4) (y1)--(y4); 
\draw[blue, line width = 2pt] (x1)--(x4) (x1)--(y1) (x4)--(y4) (x1)--(y2) (x2)--(y3) (x3) -- (y4) (y1)--(x3) (y2) -- (x4);
\foreach \i in {1,...,4}{\fill (x\i) circle (3pt); \fill (y\i) circle (3pt);}
\end{tikzpicture}\\
\centering{\hspace{.1 in} (a) \hspace{.9 in} (b) \hspace{1.35 in} (c) \hspace{1.35 in} (d) \hspace{.40 in}}
\caption{The \textbf{\textcolor{blue}{bold}} edges of each graph form a maximal isotropic subspace.}
\label{face sharing lines iso}
\end{center}
\end{figure}

To construct realizing 4-manifolds, we will use 4-reductions applied to connected sums of $S^1 \times S^3$. If $\Gamma$ has $k$ 4-cliques, then it is not difficult to see that $b_1(G) = k+3$. We have two constructions, depending on the parity of $k$. First consider the case when $k$ is even.  Begin with the connected sum of $k+3$ copies of $S^1\times S^3$, in which $b_2 = 0$. Let $\{x_1,\dotsc, x_{k+3}\}$ be the $\pi_1$ generators of each copy of $S^1$. Perform the following $(\frac k 2 +1)$ 4-reductions:
$[x_1,x_2,x_3,x_4]$, $[x_2x_6,x_3,x_4,x_5]$, $[x_4x_8,x_5,x_6,x_7], \dotsc, [x_{k-2}x_{k+2},x_{k-1},x_k,x_{k+1}]$, $[x_k,x_{k+1},x_{k+2},x_{k+3}]$, which are shown in the graph of $\pi_1$ below:

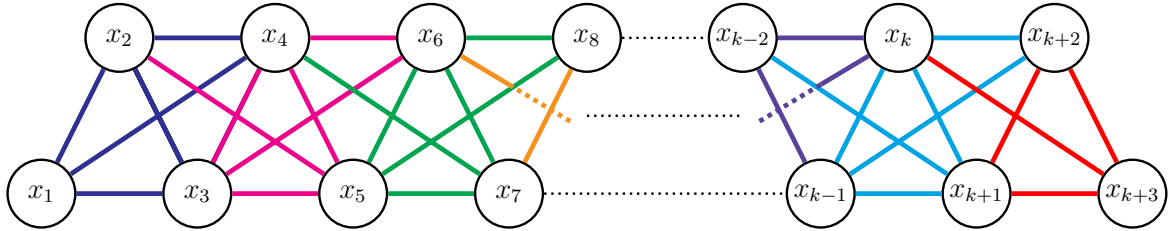
\begin{center}
\resizebox{.95\textwidth}{!}{
\begin{tikzpicture}[scale=2.3]
\tikzstyle{every node}=[draw, shape=circle, inner sep = .5mm, minimum size=1cm, line width=1pt];
\path (0,0) node (x1) {$x_1$};
\path (1,0) node (x3) {$x_3$};
\path (2,0) node (x5) {$x_5$};
\path (3,0) node (x7) {$x_7$};
\path (0.5,1) node (x2) {$x_2$};
\path (1.5,1) node (x4) {$x_4$};
\path (2.5,1) node (x6) {$x_6$};
\path (3.5,1) node (x8) {$x_8$};
\path (5,0) node (x11) {$x_{k-1}$};
\path (6,0) node (x13) {$x_{k+1}$};
\path (7,0) node (x15) {$x_{k+3}$};
\path (4.5,1) node (x10) {$x_{k-2}$};
\path (5.5,1) node (x12) {$x_{k}$};
\path (6.5,1) node (x14) {$x_{k+2}$};
\draw[line width = 2pt, Blue] (x1)--(x2)--(x3)--(x4)--(x1)--(x3)--(x2)--(x4);
\draw[line width = 2pt, magenta] (x5)--(x3)--(x4)--(x5)--(x2) (x4)--(x6)--(x3);
\draw[line width = 2pt, Green] (x7)--(x5)--(x6)--(x7)--(x4) (x6)--(x8)--(x5);
\draw[line width = 2pt, Cerulean] (x13)--(x11)--(x12)--(x13)--(x10) (x12)--(x14)--(x11);
\draw[line width = 2pt, BurntOrange] (x7)--(x8) (x6)--(3,.7);
\draw[line width = 2pt, BurntOrange, dotted] (3,.7)--(3.4,.46);
\draw[line width= 2pt, Violet] (x11)--(x10)--(x12)--(5,.7);
\draw[line width = 2pt, Violet, dotted] (5,.7)--(4.6,.46);
\draw[line width = 2pt, red] (x15)--(x14)--(x13)--(x15)--(x12);
\draw[line width = 1pt, dotted] (x8)--(x10) (x7)--(x11) (3.5,.5)--(4.5,.5);
\end{tikzpicture}
}
\end{center}

It is left to the reader to check that these 4-reductions yield all necessary relations for the correct $\pi_1$. Recall that each 4-reduction increases $b_2$ by 6. The 4-reductions result in a 4-manifold $M$ with $\pi_1(M)=G$ and with $b_2(M)=6(\frac k 2+1) = 3k+6$, equal to the lower bound. 

Now consider the case when $k$ is odd. Again begin with the connected sum of $k+3$ copies of $S^1 \times S^3$, with the same $\pi_1$ generators $\{x_1, \dotsc, x_{k+3}\}$. Perform the following $(\frac{k-1}{2} +1)$ 4-reductions:
$[x_1,x_2,x_3,x_4]$, $[x_2x_6,x_3,x_4,x_5]$, $[x_4x_8,x_5,x_6,x_7],\dotsc, [x_{k-1}x_{k+3},x_k,x_{k+1},x_{k+2}]$, as well as surgery to induce the following commutator relation: $[x_{k+2},x_{k+3}]=1$. The relations created are shown in the graph below:

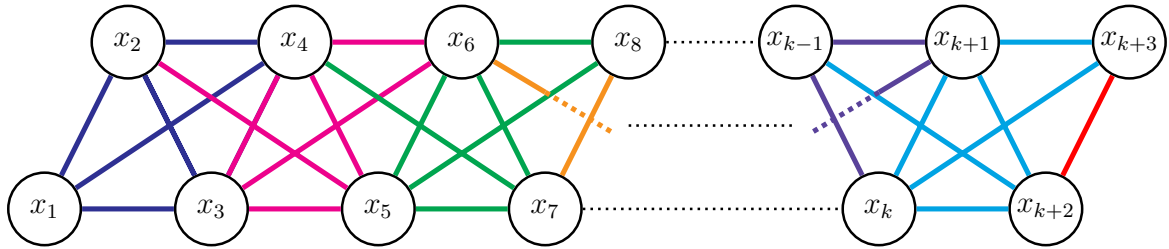
\begin{center}
\resizebox{.95\textwidth}{!}{
\begin{tikzpicture}[scale=2.3]
\tikzstyle{every node}=[draw, shape=circle, inner sep = .5mm, minimum size=1cm, line width=1pt];
\path (0,0) node (x1) {$x_1$};
\path (1,0) node (x3) {$x_3$};
\path (2,0) node (x5) {$x_5$};
\path (3,0) node (x7) {$x_7$};
\path (0.5,1) node (x2) {$x_2$};
\path (1.5,1) node (x4) {$x_4$};
\path (2.5,1) node (x6) {$x_6$};
\path (3.5,1) node (x8) {$x_8$};
\path (5,0) node (x11) {$x_{k}$};
\path (6,0) node (x13) {$x_{k+2}$};
\path (4.5,1) node (x10) {$x_{k-1}$};
\path (5.5,1) node (x12) {$x_{k+1}$};
\path (6.5,1) node (x14) {$x_{k+3}$};
\draw[line width = 2pt, Blue] (x1)--(x2)--(x3)--(x4)--(x1)--(x3)--(x2)--(x4);
\draw[line width = 2pt, magenta] (x5)--(x3)--(x4)--(x5)--(x2) (x4)--(x6)--(x3);
\draw[line width = 2pt, Green] (x7)--(x5)--(x6)--(x7)--(x4) (x6)--(x8)--(x5);
\draw[line width = 2pt, Cerulean] (x13)--(x11)--(x12)--(x13)--(x10) (x12)--(x14)--(x11);
\draw[line width = 2pt, BurntOrange] (x7)--(x8) (x6)--(3,.7);
\draw[line width = 2pt, BurntOrange, dotted] (3,.7)--(3.4,.46);
\draw[line width= 2pt, Violet] (x11)--(x10)--(x12)--(5,.7);
\draw[line width = 2pt, Violet, dotted] (5,.7)--(4.6,.46);
\draw[line width = 2pt, red] (x13)--(x14);
\draw[line width = 1pt, dotted] (x8)--(x10) (x7)--(x11) (3.5,.5)--(4.5,.5);
\end{tikzpicture}
}
\end{center}

Each 4-reduction increases $b_2$ by 6 and the commutator surgery increases $b_2$ by 2. The result is a 4-manifold $M$ with $\pi_1(M)=G$ and with $b_2(M)=6(\frac{k-1}{2}+1)+2 = 3k+5$, equal to the lower bound. 
\end{proof}

%%%%%%%%%%%%%%%%%%%%%%%%%%%%%%%%%%%%%%%%%%%%%%%%%%%%%%%%%%%%%%%%%%%%%%%%%%%%%%%%%%%%%%%%%%%%%%%%%%%%%%%%%%%%%%%%%%%%
\subsection{A hexagonal grid of 4-cliques} \label{hexagonal grids}

\begin{figure}[h]
\centering
\begin{tikzpicture}
\foreach \i in {1,...,5}{\path (\i-1,0) coordinate (a\i);}
\foreach \i in {1,...,4}{\path (\i-.5,1) coordinate (b\i);}
\foreach \i in {1,...,3}{\path (\i,2) coordinate (c\i);}
\foreach \i in {1,2}{\path (\i+.5,3) coordinate (d\i);}
\path (2,4) coordinate (e1);
\draw[line width = 1pt] (a1)--(e1)--(a5)--cycle (a2)--(d2)--(d1)--(a4)--(b4)--(b1)--cycle (a3)--(c3)--(c1)--cycle
(a1)--(c3)--(d1)--(b2)--(a4)--(c3) (e1)--(a3)--(b1)--(c2)--(b4)--(a3) (a5)--(c1)--(d2)--(b3)--(a2)--(c1);
\foreach \i in {a,b,c,d,e}{\fill (\i1) circle (3pt);}
\foreach \i in {a,b,c,d}{\fill (\i2) circle (3pt);}
\foreach \i in {a,b,c}{\fill (\i3) circle (3pt);}
\foreach \i in {a,b}{\fill (\i4) circle (3pt);}
\fill (a5) circle (3pt);		
\end{tikzpicture}
\caption{A graph in the \Hex family}
\label{tree}
\end{figure}
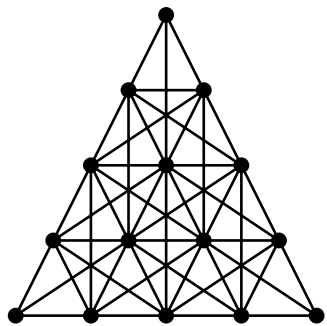

Consider the infinite family of graphs with 4-cliques attached along faces whose vertices lie in a hexagonal grid. In this setup, each triangle in the graph is not shared by more than three 3-cliques, and in each presentation of a 4-clique, the \emph{long} edge is never a shared edge. We call graphs in this family \emph{thick} if all boundary edges of the graph form an isotropic subspace. For example, Figure \ref{thick example} shows two thick 4-cliques and a \emph{thin} (not thick) 4-clique. We will call thick graphs lying in a hexagonal grid members of the \Hex family.

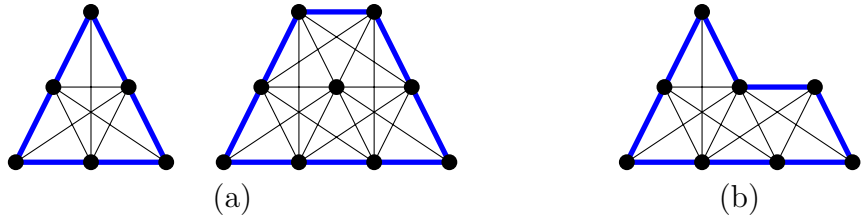
\begin{figure}[h]
\centering
\begin{minipage}{.5\linewidth}
\centering
\begin{tikzpicture}
\foreach \i in {1,...,3}{\path (\i-1,0) coordinate (a\i);}
\foreach \i in {1,2}{\path (\i-.5,1) coordinate (b\i);}
\path (1,2) coordinate (c1);
\draw[line width = 2pt,blue] (c1)--(a1)--(a3)--cycle; 
\draw (a2)--(b1)--(b2)--(a2)--(c1) (a1)--(b2) (a3)--(b1);
\foreach \i in {a,b,c}{\fill (\i1) circle (3pt);}
\foreach \i in {a,b}{\fill (\i2) circle (3pt);}
\foreach \i in {a}{\fill (\i3) circle (3pt);}
\end{tikzpicture}
\quad
\begin{tikzpicture}
\foreach \i in {1,...,4}{\path (\i-1,0) coordinate (a\i);}
\foreach \i in {1,2,3}{\path (\i-.5,1) coordinate (b\i);}
\path (1,2) coordinate (c1); \path (2,2) coordinate (c2);
\draw[line width = 2pt,blue] (c1)--(a1)--(a4)--(b3)--(c2)--cycle; 
\draw (a2)--(b1)--(b2)--(a2)--(c1)--(b2)--(c2)--(a3) (a1)--(b2) (a3)--(b1) (b2)--(a3)--(b3) (a2)--(b3) (b2)--(a4) (b2)--(b3)--(c1) (b1)--(c2);
\foreach \i in {1,...,4}{\fill (a\i) circle (3pt);}
\foreach \i in {1,2,3}{\fill (b\i) circle (3pt);}
\fill (c1) circle (3pt); \fill (c2) circle (3pt);
\end{tikzpicture} \\
(a)
\end{minipage}
\begin{minipage}{.3\linewidth}
\centering
\begin{tikzpicture}
\foreach \i in {1,...,4}{\path (\i-1,0) coordinate (a\i);}
\foreach \i in {1,2,3}{\path (\i-.5,1) coordinate (b\i);}
\path (1,2) coordinate (c1); 
\draw[line width = 2pt,blue] (c1)--(a1)--(a4)--(b3)--(b2)--cycle; 
\draw (a2)--(b1)--(b2)--(a2)--(c1) (a1)--(b2) (a3)--(b1) (b2)--(a3)--(b3) (a2)--(b3) (b2)--(a4);
\foreach \i in {1,...,4}{\fill (a\i) circle (3pt);}
\foreach \i in {1,2,3}{\fill (b\i) circle (3pt);}
\fill (c1) circle (3pt);
\end{tikzpicture} \\
(b)
\end{minipage}
\caption{The graphs in (a) are \emph{thick} and the graph in (b) is \emph{thin}.}
\label{thick example}
\end{figure}

\begin{theorem}
\RAAGs\ with associated graphs belonging to the \Hex family are \cm.
\end{theorem}

\begin{proof}
To prove this theorem, we will first discuss a how to find an isotropic subspace from a graph in the \Hex family in order to bound $h$ from below, and then show this lower bound can be realized by a 4-manifold constructed from 4-reductions and surgeries. 

Consider an arbitrary graph in the \Hex family. Since the graph is thick, all boundary edges form an isotropic subspace. Additionally, since every long edge of each 4-clique is not a shared edge, we can add it to the isotropic subspace. As an example, consider the graph in Figure \ref{large face example} (a), which we will denote by $\Gamma$. 

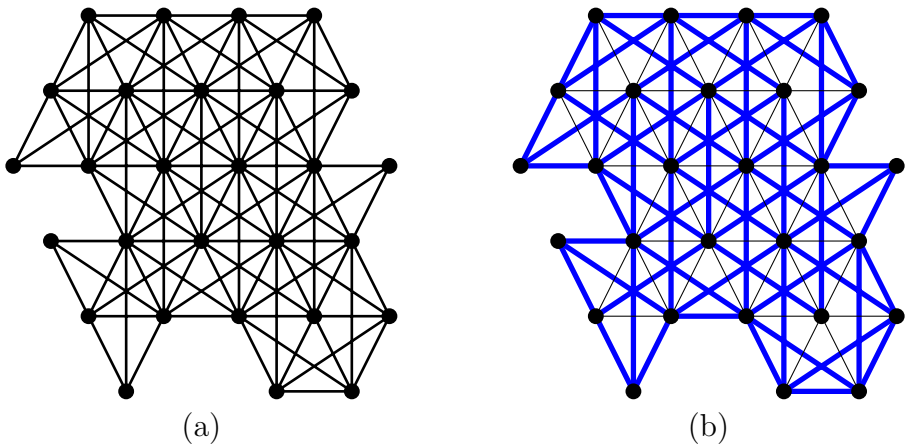
\begin{figure}[h]
\centering
\begin{minipage}{.4\linewidth}
\centering
\begin{tikzpicture}
\foreach \i in {2,4,5}{\path (\i-0.5,0) coordinate (a\i);}
\foreach \i in {1,...,5}{
	\path (\i,1) coordinate (b\i);
	\path (\i-0.5,2) coordinate (c\i);
	\path (\i-0.5,4) coordinate (e\i);
	}
\foreach \i in {0,...,5}{\path (\i,3) coordinate (d\i);}
\foreach \i in {0,...,4}{\path (\i,5) coordinate (f\i);}
\draw[line width = 1pt] (a2)--(c1)--(c2)--(d1)--(d0)--(f1)--(f4)--(e5)--(d4)--(d5)--(c5)--(b5)--(a5)--(a4)--(b3)--(b2)--(a2)
(b1)--(d4)--(f1)--(d1)--(f4)--(d4) (c2)--(a5)--(c5)--(e2)--(f3)--(e5)--(c2)--(a2) (c1)--(b2)--(d5) (d0)--(e2)--(c2) (e1)--(b5)--(a4)--(e4)--(f2)--(e1) (f2)--(b2) (f3)--(b3)--(c5) (d1)--(b4)--(d4) (d2)--(e4) (c3)--(e3) (a4)--(c5)--(f3)--(c2)--(b2)--(b1)--(c2)--(c5) (a5)--(f2)--(d1)--(d4)--(b3)--(f1) (b3)--(b5) (d1)--(e1)--(e5)
(b2)--(f4);
\fill (d0) circle (3pt);	
\foreach \i in {b,c,d,e,f}{\fill (\i1) circle (3pt);}
\foreach \i in {a,b,c,d,e,f}{\fill (\i2) circle (3pt);}
\foreach \i in {b,c,d,e,f}{\fill (\i3) circle (3pt);}
\foreach \i in {a,b,c,d,e,f}{\fill (\i4) circle (3pt);}
\foreach \i in {a,b,c,d,e}{\fill (\i5) circle (3pt);}
\end{tikzpicture}\\
(a)
\end{minipage}
\begin{minipage}{.4\linewidth}
\centering
\begin{tikzpicture}
\foreach \i in {2,4,5}{\path (\i-0.5,0) coordinate (a\i);}
\foreach \i in {1,...,5}{
	\path (\i,1) coordinate (b\i);
	\path (\i-0.5,2) coordinate (c\i);
	\path (\i-0.5,4) coordinate (e\i);
	}
\foreach \i in {0,...,5}{\path (\i,3) coordinate (d\i);}
\foreach \i in {0,...,4}{\path (\i,5) coordinate (f\i);}
\draw[line width = 2pt, blue] (a2)--(c1)--(c2)--(d1)--(d0)--(f1)--(f4)--(e5)--(d4)--(d5)--(c5)--(b5)--(a5)--(a4)--(b3)--(b2)--(a2);
\draw[line width = 2pt, blue] (b1)--(d4)--(f1)--(d1)--(f4)--(d4) (c2)--(a5)--(c5)--(e2)--(f3)--(e5)--(c2)--(a2)
(c1)--(b2)--(d5) (d0)--(e2)--(c2) (e1)--(b5)--(a4)--(e4)--(f2)--(e1) (f2)--(b2) (f3)--(b3)--(c5) (d1)--(b4)--(d4)
(d2)--(e4) (c3)--(e3);
\draw (a4)--(c5)--(f3)--(c2)--(b2)--(b1)--(c2)--(c5) (a5)--(f2)--(d1)--(d4)--(b3)--(f1) (b3)--(b5) (d1)--(e1)--(e5)
(b2)--(f4);
\fill (d0) circle (3pt);	
\foreach \i in {b,c,d,e,f}{\fill (\i1) circle (3pt);}
\foreach \i in {a,b,c,d,e,f}{\fill (\i2) circle (3pt);}
\foreach \i in {b,c,d,e,f}{\fill (\i3) circle (3pt);}
\foreach \i in {a,b,c,d,e,f}{\fill (\i4) circle (3pt);}
\foreach \i in {a,b,c,d,e}{\fill (\i5) circle (3pt);}
\end{tikzpicture}  \\
(b)
\end{minipage}
\caption{(a) A graph in the \Hex family as well as (b) the edges that form an isotropic subspace}
\label{large face example}
\end{figure}

The boundary edges and the long diagonal edges of every 4-clique in $\Gamma$, \textbf{\textcolor{blue}{highlighted}} in Figure \ref{large face example} (b), form an isotropic subspace. Later, we will see that this isotropic subspace is a maximum.

For any arbitrary graph in the \Hex Family, we will construct a realizing 4-manifold as follows. Begin with the connected sum of $b_1$ copies of $S^1\times S^3$. We will need to perform both 4-reductions on these generators as well as commutator surgeries in order for the 4-manifold to have the correct $\pi_1$. Since it is intractible to give an arbitrary graph a set of generators and list the necessary 4-reductions and commutator surgeries, we will instead describe the pattern in which one can determine the surgeries from our example graph $\Gamma$. First note that all necessary 4-reductions will be of the form $[a,b,c,d]$, $[a,b,c,de]$, or $[a,b,c,def]$, where $a,b,c,d,e,$ and $f$ represent $\pi_1$ generators. In Section \ref{4-reductions in action} we discussed the useful technique of shading a triangle in the graph bounded by the edges between $a$, $b$, and $c$. Figure \ref{shading example} shows two possible yet equally sufficient constructions of a realizing 4-manifold $M$ that has an associated $\pi_1$ graph $\Gamma$.

\begin{figure}[h]
\centering
\begin{minipage}{.4\linewidth}
\centering
\begin{tikzpicture}
\foreach \i in {2,4,5}{\path (\i-0.5,0) coordinate (a\i);}
\foreach \i in {1,...,5}{
	\path (\i,1) coordinate (b\i);
	\path (\i-0.5,2) coordinate (c\i);
	\path (\i-0.5,4) coordinate (e\i);
	}
\foreach \i in {0,...,5}{\path (\i,3) coordinate (d\i);}
\foreach \i in {0,...,4}{\path (\i,5) coordinate (f\i);}
\draw[line width = 2pt, Violet] (a2)--(c1)--(c2)--(d1) (f1)--(f4) (e5)--(d4)--(d5)--(c5) (b5)--(a5) (a4)--(b3) (b2)--(a2);
\filldraw[OliveGreen, pattern=dots, pattern color=OliveGreen, line width = 1pt] (d0)--(e1)--(d1)--cycle;
\filldraw[OliveGreen, pattern=dots, pattern color=OliveGreen, line width = 1pt] (e1)--(f1)--(e2)--cycle;
\filldraw[OliveGreen, pattern=dots, pattern color=OliveGreen, line width = 1pt] (e2)--(f2)--(e3)--cycle;
\filldraw[OliveGreen, pattern=dots, pattern color=OliveGreen, line width = 1pt] (e3)--(f3)--(e4)--cycle;
\filldraw[OliveGreen, pattern=dots, pattern color=OliveGreen, line width = 1pt] (e4)--(f4)--(e5)--cycle;
\filldraw[OliveGreen, pattern=dots, pattern color=OliveGreen, line width = 1pt] (d1)--(e2)--(d2)--cycle;
\filldraw[OliveGreen, pattern=dots, pattern color=OliveGreen, line width = 1pt] (d2)--(e3)--(d3)--cycle;
\filldraw[OliveGreen, pattern=dots, pattern color=OliveGreen, line width = 1pt] (d3)--(e4)--(d4)--cycle;
\filldraw[OliveGreen, pattern=dots, pattern color=OliveGreen, line width = 1pt] (c2)--(d2)--(c3)--cycle;
\filldraw[OliveGreen, pattern=dots, pattern color=OliveGreen, line width = 1pt] (c3)--(d3)--(c4)--cycle;
\filldraw[OliveGreen, pattern=dots, pattern color=OliveGreen, line width = 1pt] (c4)--(d4)--(c5)--cycle;
\filldraw[OliveGreen, pattern=dots, pattern color=OliveGreen, line width = 1pt] (b1)--(c2)--(b2)--cycle;
\filldraw[OliveGreen, pattern=dots, pattern color=OliveGreen, line width = 1pt] (b2)--(c3)--(b3)--cycle;
\filldraw[OliveGreen, pattern=dots, pattern color=OliveGreen, line width = 1pt] (b3)--(c4)--(b4)--cycle;
\filldraw[OliveGreen, pattern=dots, pattern color=OliveGreen, line width = 1pt] (b4)--(c5)--(b5)--cycle;
\filldraw[OliveGreen, pattern=dots, pattern color=OliveGreen, line width = 1pt] (a4)--(b4)--(a5)--cycle;
\fill (d0) circle (3pt);	
\foreach \i in {b,c,d,e,f}{\fill (\i1) circle (3pt);}
\foreach \i in {a,b,c,d,e,f}{\fill (\i2) circle (3pt);}
\foreach \i in {b,c,d,e,f}{\fill (\i3) circle (3pt);}
\foreach \i in {a,b,c,d,e,f}{\fill (\i4) circle (3pt);}
\foreach \i in {a,b,c,d,e}{\fill (\i5) circle (3pt);}
\end{tikzpicture}\\
(a)
\end{minipage}
\begin{minipage}{.4\linewidth}
\centering
\begin{tikzpicture}
\foreach \i in {2,4,5}{\path (\i-0.5,0) coordinate (a\i);}
\foreach \i in {1,...,5}{
	\path (\i,1) coordinate (b\i);
	\path (\i-0.5,2) coordinate (c\i);
	\path (\i-0.5,4) coordinate (e\i);
	}
\foreach \i in {0,...,5}{\path (\i,3) coordinate (d\i);}
\foreach \i in {0,...,4}{\path (\i,5) coordinate (f\i);}
\draw[line width = 2pt, Violet] (d1)--(d0)--(f1) (f4)--(e5) (c5)--(b5) (a5)--(a4) (b3)--(b2);	
\filldraw[Magenta, pattern=dots, pattern color=Magenta, line width = 1pt] (f1)--(e2)--(f2)--cycle;
\filldraw[Magenta, pattern=dots, pattern color=Magenta, line width = 1pt] (f2)--(e3)--(f3)--cycle;
\filldraw[Magenta, pattern=dots, pattern color=Magenta, line width = 1pt] (f3)--(e4)--(f4)--cycle;
\filldraw[Magenta, pattern=dots, pattern color=Magenta, line width = 1pt] (e1)--(d1)--(e2)--cycle;
\filldraw[Magenta, pattern=dots, pattern color=Magenta, line width = 1pt] (e2)--(d2)--(e3)--cycle;
\filldraw[Magenta, pattern=dots, pattern color=Magenta, line width = 1pt] (e3)--(d3)--(e4)--cycle;
\filldraw[Magenta, pattern=dots, pattern color=Magenta, line width = 1pt] (e4)--(d4)--(e5)--cycle;
\filldraw[Magenta, pattern=dots, pattern color=Magenta, line width = 1pt] (d1)--(c2)--(d2)--cycle;
\filldraw[Magenta, pattern=dots, pattern color=Magenta, line width = 1pt] (d2)--(c3)--(d3)--cycle;
\filldraw[Magenta, pattern=dots, pattern color=Magenta, line width = 1pt] (d3)--(c4)--(d4)--cycle;
\filldraw[Magenta, pattern=dots, pattern color=Magenta, line width = 1pt] (d4)--(c5)--(d5)--cycle;
\filldraw[Magenta, pattern=dots, pattern color=Magenta, line width = 1pt] (c1)--(b1)--(c2)--cycle;
\filldraw[Magenta, pattern=dots, pattern color=Magenta, line width = 1pt] (c2)--(b2)--(c3)--cycle;
\filldraw[Magenta, pattern=dots, pattern color=Magenta, line width = 1pt] (c3)--(b3)--(c4)--cycle;
\filldraw[Magenta, pattern=dots, pattern color=Magenta, line width = 1pt] (c4)--(b4)--(c5)--cycle;
\filldraw[Magenta, pattern=dots, pattern color=Magenta, line width = 1pt] (b1)--(a2)--(b2)--cycle;
\filldraw[Magenta, pattern=dots, pattern color=Magenta, line width = 1pt] (b3)--(a4)--(b4)--cycle;
\filldraw[Magenta, pattern=dots, pattern color=Magenta, line width = 1pt] (b4)--(a5)--(b5)--cycle;
\fill (d0) circle (3pt);	
\foreach \i in {b,c,d,e,f}{\fill (\i1) circle (3pt);}
\foreach \i in {a,b,c,d,e,f}{\fill (\i2) circle (3pt);}
\foreach \i in {b,c,d,e,f}{\fill (\i3) circle (3pt);}
\foreach \i in {a,b,c,d,e,f}{\fill (\i4) circle (3pt);}
\foreach \i in {a,b,c,d,e}{\fill (\i5) circle (3pt);}
\end{tikzpicture} \\
(b)
\end{minipage}
\caption{Two constructions for a realizing 4-manifold for a graph in the \Hex family}
\label{shading example}
\end{figure}
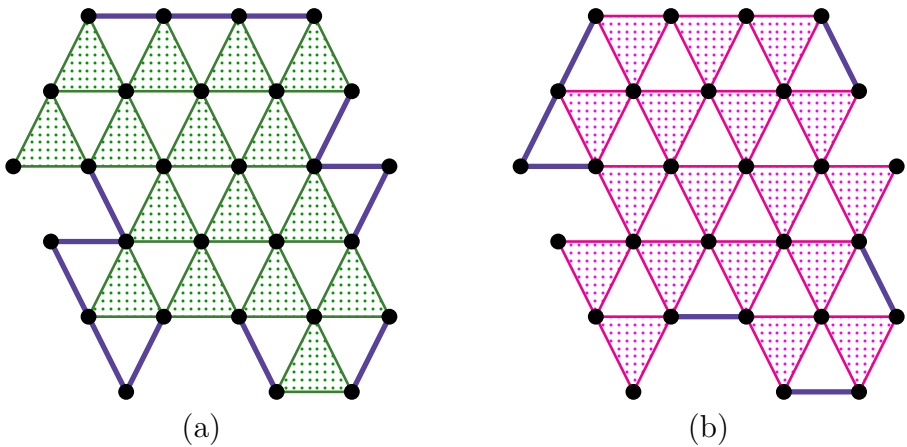

In these constructions, the number of shaded triangles in the graph corresponds the number of necessary 4-reductions. The vertices on a particular shaded triangle correspond to three of the four elements of the 4-reduction. The fourth element of the 4-reduction is either another generator or a product of generators, depending on how many 4-cliques share the face of the shaded triangle. Each \textbf{\textcolor{Violet}{bold}} edge in Figure \ref{shading example} corresponds to a necessary commutator surgery that will ensure the resulting 4-manifold will have the correct $\pi_1$. Note that all these \textbf{\textcolor{Violet}{bold}} edges are boundary edges of the graph which are not covered by any of the shaded triangles.

What remains to be seen is that this construction is ``good enough." That is, given a lower bound calculated from twice the dimension of the isotropic subspace described above, we can always construct a 4-manifold with $b_2$ equal to this lower bound by following this 4-manifold construction pattern. To do this, we will break down this construction pattern and show piece-by-piece that the cost of each 4-reduction and each surgery (in terms of adding $b_2$) can be balanced out by elements in the isotropic subspace. More specifically, we need only see that the cost $x$ of each construction (in terms of adding $b_2$) can be balanced by finding half as many ($\frac 12x$) elements in the isotropic subspace.

To begin, let us discuss the costs of the three necessary types of 4-reductions: $[a,b,c,d]$, $[a,b,c,de]$, and $[a,b,c,def]$. First notice that any shaded triangle in the graph that is not along the boundary is created by a 4-reduction of the form $[a,b,c,def]$. The vertices representing $a$, $b$, and $c$ in each case are vertices of the shaded triangle, and the vertices representing $d$, $e$, and $f$ are the fourth vertices of each respective 4-clique that shares the shaded triangle. Examples of 4-reductions of the type $[a,b,c,def]$ can be viewed in $\Gamma$ below: 
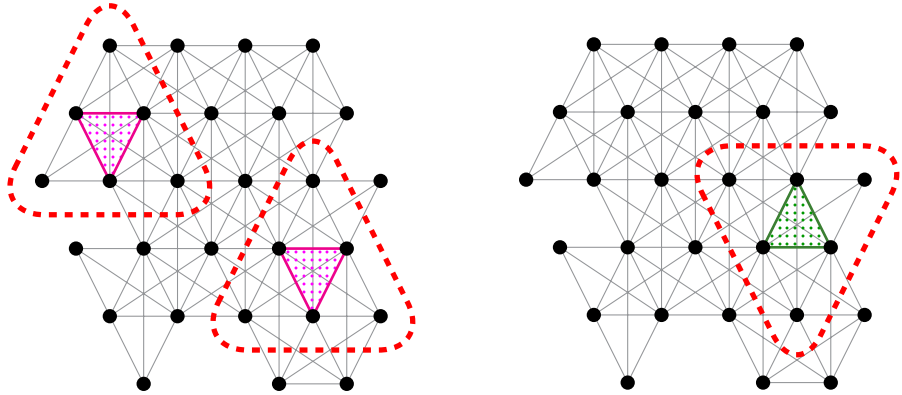
\begin{center}
\begin{minipage}{.4\linewidth}
\centering
\begin{tikzpicture}[scale=.9]
\foreach \i in {2,4,5}{\path (\i-0.5,0) coordinate (a\i);}
\foreach \i in {1,...,5}{
	\path (\i,1) coordinate (b\i);
	\path (\i-0.5,2) coordinate (c\i);
	\path (\i-0.5,4) coordinate (e\i);
	}
\foreach \i in {0,...,5}{\path (\i,3) coordinate (d\i);}
\foreach \i in {0,...,4}{\path (\i,5) coordinate (f\i);}
\draw[Gray] (a2)--(c1)--(c2)--(d1)--(d0)--(f1)--(f4)--(e5)--(d4)--(d5)--(c5)--(b5)--(a5)--(a4)--(b3)--(b2)--(a2) (b1)--(d4)--(f1)--(d1)--(f4)--(d4) (c2)--(a5)--(c5)--(e2)--(f3)--(e5)--(c2)--(a2) (c1)--(b2)--(d5)
(d0)--(e2)--(c2) (e1)--(b5)--(a4)--(e4)--(f2)--(e1) (f2)--(b2) (f3)--(b3)--(c5) (d1)--(b4)--(d4) (d2)--(e4) (c3)--(e3) (a4)--(c5)--(f3)--(c2)--(b2)--(b1)--(c2)--(c5) (a5)--(f2)--(d1)--(d4)--(b3)--(f1) (b3)--(b5) (d1)--(e1)--(e5) (b2)--(f4);
\draw[line width =2pt, dashed, rounded corners=20pt,red] (-0.75,2.5)--(1,6)--(2.75,2.5)--cycle;
\draw[line width =2pt, dashed, rounded corners=20pt,red] (2.25,.5)--(4,4)--(5.75,.5)--cycle;
\filldraw[Magenta, pattern=dots, pattern color= Magenta, line width = 1pt] (d1)--(e1)--(e2)--cycle;
\filldraw[Magenta, pattern=dots, pattern color= Magenta, line width = 1pt] (c4)--(c5)--(b4)--cycle;
\fill (d0) circle (3pt);	
\foreach \i in {b,c,d,e,f}{\fill (\i1) circle (3pt);}
\foreach \i in {a,b,c,d,e,f}{\fill (\i2) circle (3pt);}
\foreach \i in {b,c,d,e,f}{\fill (\i3) circle (3pt);}
\foreach \i in {a,b,c,d,e,f}{\fill (\i4) circle (3pt);}
\foreach \i in {a,b,c,d,e}{\fill (\i5) circle (3pt);}
\end{tikzpicture}
\end{minipage}
\bigskip
\begin{minipage}{.4\linewidth} 
\centering
\begin{tikzpicture}[scale=.9]
\foreach \i in {2,4,5}{\path (\i-0.5,0) coordinate (a\i);}
\foreach \i in {1,...,5}{
	\path (\i,1) coordinate (b\i);
	\path (\i-0.5,2) coordinate (c\i);
	\path (\i-0.5,4) coordinate (e\i);
	}
\foreach \i in {0,...,5}{\path (\i,3) coordinate (d\i);}
\foreach \i in {0,...,4}{	\path (\i,5) coordinate (f\i);}
\path (1,6) coordinate (x);	% used only to shift image down	
\draw[Gray] (a2)--(c1)--(c2)--(d1)--(d0)--(f1)--(f4)--(e5)--(d4)--(d5)--(c5)--(b5)--(a5)--(a4)--(b3)--(b2)--(a2) (b1)--(d4)--(f1)--(d1)--(f4)--(d4) (c2)--(a5)--(c5)--(e2)--(f3)--(e5)--(c2)--(a2) (c1)--(b2)--(d5)
(d0)--(e2)--(c2) (e1)--(b5)--(a4)--(e4)--(f2)--(e1) (f2)--(b2) (f3)--(b3)--(c5) (d1)--(b4)--(d4) (d2)--(e4) (c3)--(e3) (a4)--(c5)--(f3)--(c2)--(b2)--(b1)--(c2)--(c5) (a5)--(f2)--(d1)--(d4)--(b3)--(f1) (b3)--(b5) (d1)--(e1)--(e5) (b2)--(f4);
\draw[line width =2pt, dashed, rounded corners=20pt,red] (2.25,3.5)--(4,0)--(5.75,3.5)--cycle;
\filldraw[OliveGreen, pattern=dots, pattern color= OliveGreen, line width = 1pt] (d4)--(c4)--(c5)--cycle;
\fill (d0) circle (3pt);	
\foreach \i in {b,c,d,e,f}{\fill (\i1) circle (3pt);}
\foreach \i in {a,b,c,d,e,f}{\fill (\i2) circle (3pt);}
\foreach \i in {b,c,d,e,f}{\fill (\i3) circle (3pt);}
\foreach \i in {a,b,c,d,e,f}{\fill (\i4) circle (3pt);}
\foreach \i in {a,b,c,d,e}{\fill (\i5) circle (3pt);}
\end{tikzpicture}
\end{minipage}
\end{center}
\medskip
Individually, each 4-reduction $[a,b,c,def]$ will eventually result in twelve commutator relations, those represented by the edges in Figure \ref{4-reduction type 3} (a). Automatically, the relations $[a,b]=1$, $[a,c]=1$, and $[b,c]=1$ are created, represented by the edges in (b). The other surface-like relations (for example, $[a,def]=1$) will be resolved by other 4-reductions and/or commutator surgeries. However, the three relations represented by the long diagonal edges highlighted in (c), are only induced by this 4-reduction once the outer edges of (a) are created. 

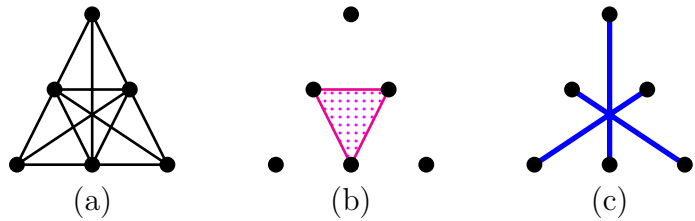
\begin{figure}[h]
\begin{center}
\begin{minipage}{.2\linewidth}
\centering
\begin{tikzpicture}
\foreach \i in {1,...,3}{\path (\i-1,0) coordinate (a\i);}
\foreach \i in {1,2}{\path (\i-.5,1) coordinate (b\i);}
\path (1,2) coordinate (c1);
\draw[line width = 1pt] (c1)--(a1)--(a3)--cycle (a2)--(b1)--(b2)--(a2)--(c1) (a1)--(b2) (a3)--(b1);
\foreach \i in {a,b,c}{\fill (\i1) circle (3pt);}
\foreach \i in {a,b}{\fill (\i2) circle (3pt);}
\foreach \i in {a}{\fill (\i3) circle (3pt);}
\end{tikzpicture} \\
(a)
\end{minipage}
\begin{minipage}{.2\linewidth}
\centering
\begin{tikzpicture}
\foreach \i in {1,...,3}{\path (\i-1,0) coordinate (a\i);}
\foreach \i in {1,2}{\path (\i-.5,1) coordinate (b\i);}
\path (1,2) coordinate (c1);
\filldraw[Magenta, pattern=dots, pattern color= Magenta, line width = 1pt] (b1)--(b2)--(a2)--cycle;
\foreach \i in {a,b,c}{\fill (\i1) circle (3pt);}
\foreach \i in {a,b}{\fill (\i2) circle (3pt);}
\foreach \i in {a}{\fill (\i3) circle (3pt);}
\end{tikzpicture} \\
(b)
\end{minipage}
\begin{minipage}{.2\linewidth}
\centering
\begin{tikzpicture}
\foreach \i in {1,...,3}{\path (\i-1,0) coordinate (a\i);}
\foreach \i in {1,2}{\path (\i-.5,1) coordinate (b\i);}
\path (1,2) coordinate (c1);
\draw[line width = 2pt, blue] (a2)--(c1) (a1)--(b2) (a3)--(b1);
\foreach \i in {a,b,c}{\fill (\i1) circle (3pt);}
\foreach \i in {a,b}{\fill (\i2) circle (3pt);}
\foreach \i in {a}{\fill (\i3) circle (3pt);}
\end{tikzpicture} \\
(c)
\end{minipage}
\end{center}
\caption{The edges created by a 4-reduction of type $[a,b,c,def]$, (b) the triangle created by $a$, $b$, and $c$, and (c) the edges belonging to the maximal isotropic subspace}
\label{4-reduction type 3}
\end{figure}

Each 4-reduction adds 6 to the total $b_2$ of the 4-manifold. This addition is balanced out by three edges that represent basis elements in the isotropic subspace. These three edges are the long diagonal edges of the three 4-cliques created by this 4-reduction, shown in (c). Any remaining relations are induced by other 4-reductions and/or commutator surgeries and their costs are balanced elsewhere.

Next, notice that any shaded triangle in the graph that has one edge along the boundary is created by a 4-reduction of the form $[a,b,c,de]$. As in the previous case, the vertices representing $a$, $b$, and $c$ are vertices of the shaded triangle, and the vertices representing $d$ and $e$ are the fourth vertices of the two respective 4-cliques that share the shaded triangle. Examples of 4-reductions of the type $[a,b,c,de]$ can be viewed in $\Gamma$ below: 
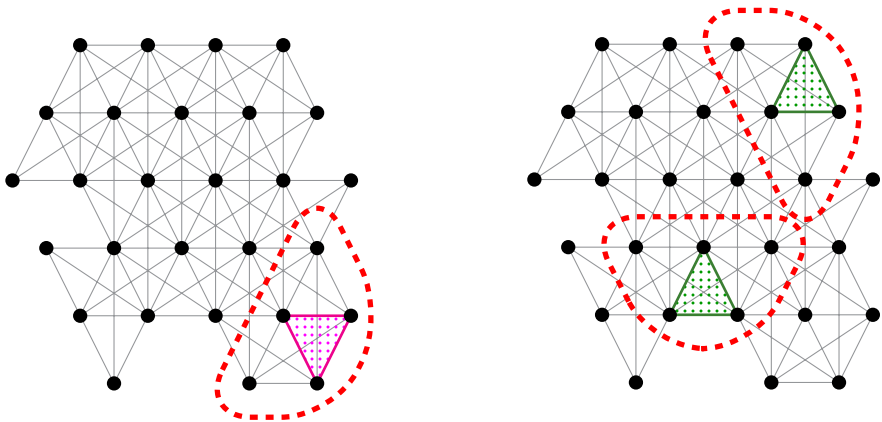
\begin{center}
\begin{minipage}{.4\linewidth}
\centering
\begin{tikzpicture}[scale=.9]
\foreach \i in {2,4,5}{\path (\i-0.5,0) coordinate (a\i);}
\foreach \i in {1,...,5}{
	\path (\i,1) coordinate (b\i);
	\path (\i-0.5,2) coordinate (c\i);
	\path (\i-0.5,4) coordinate (e\i);
	}
\foreach \i in {0,...,5}{\path (\i,3) coordinate (d\i);}
\foreach \i in {0,...,4}{	\path (\i,5) coordinate (f\i);}
\path (1,6) coordinate (x);	% used only to shift image down
\draw[Gray] (a2)--(c1)--(c2)--(d1)--(d0)--(f1)--(f4)--(e5)--(d4)--(d5)--(c5)--(b5)--(a5)--(a4)--(b3)--(b2)--(a2) (b1)--(d4)--(f1)--(d1)--(f4)--(d4) (c2)--(a5)--(c5)--(e2)--(f3)--(e5)--(c2)--(a2) (c1)--(b2)--(d5)
(d0)--(e2)--(c2) (e1)--(b5)--(a4)--(e4)--(f2)--(e1) (f2)--(b2) (f3)--(b3)--(c5) (d1)--(b4)--(d4) (d2)--(e4) (c3)--(e3) (a4)--(c5)--(f3)--(c2)--(b2)--(b1)--(c2)--(c5) (a5)--(f2)--(d1)--(d4)--(b3)--(f1) (b3)--(b5) (d1)--(e1)--(e5) (b2)--(f4);
\draw[line width =2pt, dashed, rounded corners=20pt, red] (2.75,-0.5) -- (4.5,3) -- (5.5,1) -- (4.75,-0.5) -- cycle;
\filldraw[Magenta, pattern=dots, pattern color= Magenta, line width = 1pt] (a5)--(b4)--(b5)--cycle;
\fill (d0) circle (3pt);	
\foreach \i in {b,c,d,e,f}{\fill (\i1) circle (3pt);}
\foreach \i in {a,b,c,d,e,f}{\fill (\i2) circle (3pt);}
\foreach \i in {b,c,d,e,f}{\fill (\i3) circle (3pt);}
\foreach \i in {a,b,c,d,e,f}{\fill (\i4) circle (3pt);}
\foreach \i in {a,b,c,d,e}{\fill (\i5) circle (3pt);}
\end{tikzpicture}
\end{minipage}
\begin{minipage}{.4\linewidth}
\centering
\begin{tikzpicture}[scale=.9]
\foreach \i in {2,4,5}{\path (\i-0.5,0) coordinate (a\i);}
\foreach \i in {1,...,5}{
	\path (\i,1) coordinate (b\i);
	\path (\i-0.5,2) coordinate (c\i);
	\path (\i-0.5,4) coordinate (e\i);
	}
\foreach \i in {0,...,5}{\path (\i,3) coordinate (d\i);}
\foreach \i in {0,...,4}{	
	\path (\i,5) coordinate (f\i);
	}
\path (1,5.5) coordinate (x);	% used only to shift image down
\draw[Gray] (a2)--(c1)--(c2)--(d1)--(d0)--(f1)--(f4)--(e5)--(d4)--(d5)--(c5)--(b5)--(a5)--(a4)--(b3)--(b2)--(a2) (b1)--(d4)--(f1)--(d1)--(f4)--(d4) (c2)--(a5)--(c5)--(e2)--(f3)--(e5)--(c2)--(a2) (c1)--(b2)--(d5)
(d0)--(e2)--(c2) (e1)--(b5)--(a4)--(e4)--(f2)--(e1) (f2)--(b2) (f3)--(b3)--(c5) (d1)--(b4)--(d4) (d2)--(e4) (c3)--(e3) (a4)--(c5)--(f3)--(c2)--(b2)--(b1)--(c2)--(c5) (a5)--(f2)--(d1)--(d4)--(b3)--(f1) (b3)--(b5) (d1)--(e1)--(e5) (b2)--(f4);
\draw[line width =2pt, dashed, rounded corners=20pt, red] (1.75,0.5) -- (0.75,2.45) --(4.25,2.45)--(3.25,0.5)--cycle;	
\filldraw[OliveGreen, pattern=dots, pattern color= OliveGreen, line width = 1pt] (c3)--(b2)--(b3)--cycle;
\draw[line width = 2pt, dashed, rounded corners = 20pt, red] (2.25,5.5) -- (4,2) -- (5,4) -- (4.25,5.5) -- cycle;
\filldraw[OliveGreen, pattern=dots, pattern color= OliveGreen, line width = 1pt] (f4)--(e4)--(e5)--cycle;
\fill (d0) circle (3pt);	
\foreach \i in {b,c,d,e,f}{\fill (\i1) circle (3pt);}
\foreach \i in {a,b,c,d,e,f}{\fill (\i2) circle (3pt);}
\foreach \i in {b,c,d,e,f}{\fill (\i3) circle (3pt);}
\foreach \i in {a,b,c,d,e,f}{\fill (\i4) circle (3pt);}
\foreach \i in {a,b,c,d,e}{\fill (\i5) circle (3pt);}
\end{tikzpicture}
\end{minipage} 
\end{center}
\medskip
Each 4-reduction of the form $[a,b,c,de]$ will eventually result in nine commutator relations, represented by the edges in Figure \ref{4-reduction type 2} (a). Again, we see the triangle in (b) represents the relations $[a,b]=1$, $[a,c]=1$, and $[b,c]=1$. The two long diagonal edges in (c) will be resolved by other 4-reductions and/or commutator surgeries.

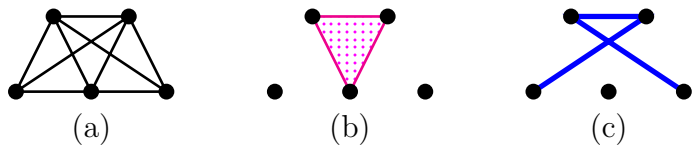
\begin{figure}[h]
\begin{center}
\begin{minipage}{.2\linewidth}
\centering
\begin{tikzpicture}
\foreach \i in {1,2}{
	\path (\i-1,0) coordinate (a\i);
	\path (\i-.5,1) coordinate (b\i);
	}
\path (2,0) coordinate (a3);
\draw[line width=1pt] (a1)--(a3)--(b2)--(b1)--(a2)--(b2)--(a1)--(b1)--(a3);
\foreach \i in {1,...,3}{\fill (a\i) circle (3pt);}
\foreach \i in {1,2}{\fill (b\i) circle (3pt);}
\end{tikzpicture} \\
(a)
\end{minipage}
\begin{minipage}{.2\linewidth}
\centering
\begin{tikzpicture}
\foreach \i in {1,2}{
	\path (\i-1,0) coordinate (a\i);
	\path (\i-.5,1) coordinate (b\i);
	}
\path (2,0) coordinate (a3);
\filldraw[Magenta, pattern=dots, pattern color= Magenta, line width = 1pt] (b1)--(b2)--(a2)--cycle;
\foreach \i in {1,...,3}{\fill (a\i) circle (3pt);}
\foreach \i in {1,2}{\fill (b\i) circle (3pt);}
\end{tikzpicture} \\
(b)
\end{minipage}
\begin{minipage}{.2\linewidth}
\centering
\begin{tikzpicture}
\foreach \i in {1,2}{
	\path (\i-1,0) coordinate (a\i);
	\path (\i-.5,1) coordinate (b\i);
	}
\path (2,0) coordinate (a3);
\draw[line width=2pt,blue] (a1)--(b2)--(b1)--(a3);
\foreach \i in {1,...,3}{\fill (a\i) circle (3pt);}
\foreach \i in {1,2}{\fill (b\i) circle (3pt);}
\end{tikzpicture} \\
(c)
\end{minipage}
\end{center}
\caption{The edges created by a 4-reduction of type $[a,b,c,de]$, (b) the triangle created by $a$, $b$, and $c$, and (c) the edges belonging to the maximal isotropic subspace}
\label{4-reduction type 2}
\end{figure}

Each 4-reduction of this form still adds 6 to the total $b_2$ of the 4-manifold. This addition is balanced out by three edges that represent basis elements in the isotropic subspace, the two long diagonal edges and the boundary edge in (c). Any remaining relations are induced by other 4-reductions and/or commutator surgeries and their costs are balanced elsewhere.

The last type of 4-reduction, $[a,b,c,d]$, occurs when the shaded triangle has two boundary edges. In this case, there is only one 4-clique in the graph containing the shaded triangle, the one formed by vertices representing $a$, $b$, $c$, and $d$. Examples of 4-reductions of this type can be viewed in $\Gamma$ below: 
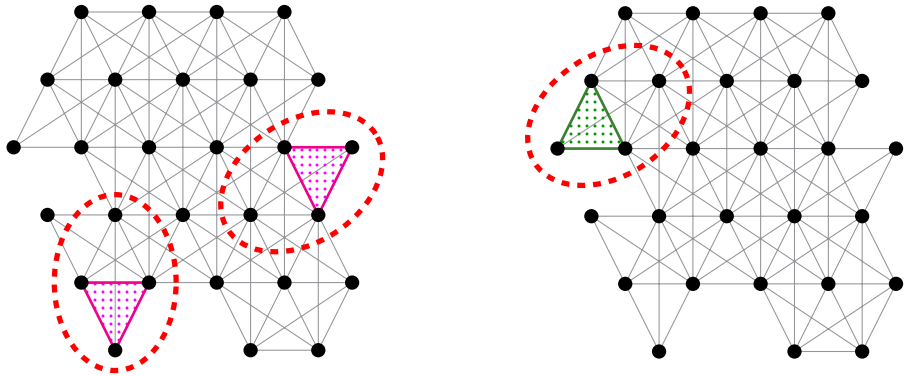
\begin{center}
\begin{minipage}{.4\linewidth}
\centering
\begin{tikzpicture}[scale=.9]
\foreach \i in {2,4,5}{\path (\i-0.5,0) coordinate (a\i);}
\foreach \i in {1,...,5}{
	\path (\i,1) coordinate (b\i);
	\path (\i-0.5,2) coordinate (c\i);
	\path (\i-0.5,4) coordinate (e\i);
	}
\foreach \i in {0,...,5}{\path (\i,3) coordinate (d\i);}
\foreach \i in {0,...,4}{	\path (\i,5) coordinate (f\i);}
\path (1,5.3) coordinate (x);	% used only to shift image down
\draw[Gray] (a2)--(c1)--(c2)--(d1)--(d0)--(f1)--(f4)--(e5)--(d4)--(d5)--(c5)--(b5)--(a5)--(a4)--(b3)--(b2)--(a2) (b1)--(d4)--(f1)--(d1)--(f4)--(d4) (c2)--(a5)--(c5)--(e2)--(f3)--(e5)--(c2)--(a2) (c1)--(b2)--(d5)
(d0)--(e2)--(c2) (e1)--(b5)--(a4)--(e4)--(f2)--(e1) (f2)--(b2) (f3)--(b3)--(c5) (d1)--(b4)--(d4) (d2)--(e4) (c3)--(e3) (a4)--(c5)--(f3)--(c2)--(b2)--(b1)--(c2)--(c5) (a5)--(f2)--(d1)--(d4)--(b3)--(f1) (b3)--(b5) (d1)--(e1)--(e5) (b2)--(f4);
\draw[line width = 2pt, dashed, red, shift={(1.5,0)}] (0,1) ellipse [x radius = .9cm, y radius = 1.3cm];
\filldraw[Magenta, pattern=dots, pattern color= Magenta, line width = 1pt] (a2)--(b1)--(b2)--cycle;
\draw[line width = 2pt, dashed, red, shift={(3.5,2)}] (33: 0.9cm) ellipse [x radius = 1.3cm, y radius = .9cm, rotate=33];
\filldraw[Magenta, pattern=dots, pattern color= Magenta, line width = 1pt] (c5)--(d4)--(d5)--cycle;
\fill (d0) circle (3pt);	
\foreach \i in {b,c,d,e,f}{\fill (\i1) circle (3pt);}
\foreach \i in {a,b,c,d,e,f}{\fill (\i2) circle (3pt);}
\foreach \i in {b,c,d,e,f}{\fill (\i3) circle (3pt);}
\foreach \i in {a,b,c,d,e,f}{\fill (\i4) circle (3pt);}
\foreach \i in {a,b,c,d,e}{\fill (\i5) circle (3pt);}
\end{tikzpicture}
\end{minipage}
\begin{minipage}{.4\linewidth}
\centering
\begin{tikzpicture}[scale=.9]
\foreach \i in {2,4,5}{\path (\i-0.5,0) coordinate (a\i);}
\foreach \i in {1,...,5}{
	\path (\i,1) coordinate (b\i);
	\path (\i-0.5,2) coordinate (c\i);
	\path (\i-0.5,4) coordinate (e\i);
	}
\foreach \i in {0,...,5}{\path (\i,3) coordinate (d\i);}
\foreach \i in {0,...,4}{\path (\i,5) coordinate (f\i);}
%\path (1,6) coordinate (x);	% used only to shift image down
\draw[Gray] (a2)--(c1)--(c2)--(d1)--(d0)--(f1)--(f4)--(e5)--(d4)--(d5)--(c5)--(b5)--(a5)--(a4)--(b3)--(b2)--(a2) (b1)--(d4)--(f1)--(d1)--(f4)--(d4) (c2)--(a5)--(c5)--(e2)--(f3)--(e5)--(c2)--(a2) (c1)--(b2)--(d5)
(d0)--(e2)--(c2) (e1)--(b5)--(a4)--(e4)--(f2)--(e1) (f2)--(b2) (f3)--(b3)--(c5) (d1)--(b4)--(d4) (d2)--(e4) (c3)--(e3) (a4)--(c5)--(f3)--(c2)--(b2)--(b1)--(c2)--(c5) (a5)--(f2)--(d1)--(d4)--(b3)--(f1) (b3)--(b5) (d1)--(e1)--(e5) (b2)--(f4);
\draw[line width = 2pt, dashed, red, shift={(0,3)}] (33: 0.9cm) ellipse [x radius = 1.3cm, y radius = .9cm, rotate=33];
\filldraw[OliveGreen, pattern=dots, pattern color= OliveGreen, line width = 1pt] (d0)--(d1)--(e1)--cycle;
\fill (d0) circle (3pt);	
\foreach \i in {b,c,d,e,f}{\fill (\i1) circle (3pt);}
\foreach \i in {a,b,c,d,e,f}{\fill (\i2) circle (3pt);}
\foreach \i in {b,c,d,e,f}{\fill (\i3) circle (3pt);}
\foreach \i in {a,b,c,d,e,f}{\fill (\i4) circle (3pt);}
\foreach \i in {a,b,c,d,e}{\fill (\i5) circle (3pt);}
\end{tikzpicture}
\end{minipage} 
\end{center}
\medskip
Each 4-reduction of the form $[a,b,c,d]$ induces 6 commutator relations, shown in Figure \ref{4-reduction type 1} (a) and adds 6 to the total $b_2$ of the 4-manifold. This addition is balanced out by three edges that represent basis elements in the isotropic subspace, the single long diagonal edge of the 4-clique and the two boundary edges shown in (c). 

\begin{figure}[h]
\begin{center}
\begin{minipage}{.2\linewidth}
\centering
\begin{tikzpicture}
\foreach \i in {1,2}{
	\path (\i-1,0) coordinate (a\i);
	\path (\i-.5,1) coordinate (b\i);
	}
\draw[line width=1pt] (a1)--(a2)--(b2)--(b1)--(a2) (b1)--(a1)--(b2);
\foreach \i in {1,2}{
	\fill (a\i) circle (3pt);
	\fill (b\i) circle (3pt);
	}
\end{tikzpicture} \\
(a)
\end{minipage}
\begin{minipage}{.2\linewidth}
\centering
\begin{tikzpicture}
\foreach \i in {1,2}{
	\path (\i-1,0) coordinate (a\i);
	\path (\i-.5,1) coordinate (b\i);
	}
\filldraw[OliveGreen, pattern=dots, pattern color= OliveGreen, line width = 1pt] (a1)--(a2)--(b1)--cycle;
\foreach \i in {1,2}{
	\fill (a\i) circle (3pt);
	\fill (b\i) circle (3pt);
	}
\end{tikzpicture} \\
(b) 
\end{minipage}
\begin{minipage}{.2\linewidth}
\centering
\begin{tikzpicture}
\foreach \i in {1,2}{
	\path (\i-1,0) coordinate (a\i);
	\path (\i-.5,1) coordinate (b\i);
	}
\draw[line width=2pt,blue] (a1)--(b2) (b1)--(a1)--(a2);
\foreach \i in {1,2}{
	\fill (a\i) circle (3pt);
	\fill (b\i) circle (3pt);
	}
\end{tikzpicture} \\
(c)
\end{minipage} 
\end{center}
\caption{The edges created by a 4-reduction of type $[a,b,c,d]$, (b) the triangle created by $a$, $b$, and $c$, and (c) the edges belonging to the maximal isotropic subspace}
\label{4-reduction type 1}
\end{figure}
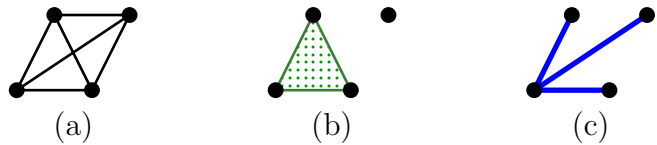

Lastly, we will consider the cost of the commutator surgeries. Each commutator surgery induces a relation that represents a boundary edge of the graph, and the cost of the surgery (an addition of 2 to $b_2$) is balanced out by the fact that the corresponding boundary edge in the graph is in the isotropic subspace. 

Since the cost of each 4-reduction and each surgery are balanced by elements in the isotropic subspace, it is clear that the pattern exemplified by Figure \ref{large face example} (b) yields a maximum dimensional isotropic subspace and the construction pattern in Figure \ref{shading example} yields a realizing 4-manifold. This also shows that, interestingly, either pattern in Figure \ref{shading example} is sufficient to construct a realizing manifold.
\end{proof}

%%%%%%%%%%%%%%%%%%%%%%%%%%%%%%%%%%%%%%%%%%%%%%%%%%%%%%%%%%%%%%%%%%%%%%%%%%%%%%%%%%%%%%%%%%%%%%%%%%%%%%%%%%%%%%%%%%%%
\subsection{\RAAGs\ with nontrivial higher cohomology}\label{higher cohomology}

In graph theory, the \emph{dimension} of a graph refers to the dimension of the largest clique in the graph. In terms of the cohomology of \RAAGs, it is the largest nonzero cohomological dimension. Until now, we have only considered \RAAGs\ of dimension 4. There are many reasons for this. 

Right-angled Artin groups of dimension 4 are special, as 4 is the first dimension in which the cohomology ring really has an interesting influence on the possible values of $b_2(M)$ for arbitrary $M\in \mathcal M(G)$. Determining $h$ is a delicate problem in groups of dimension 4 because calculations of $m_2$ as well as realizing manifold constructions are completely dependent on the ways in which 4-cliques interact in the graph. This provides evidence that the difficulty in determining the minimum $b_2$ problem of \RAAGs\ lies in this dimension. 

We now restrict the discussion to graphs of dimension $k$ in which all $(k-1)$-cliques in the graphs are subgraphs of a $k$-clique. Let us say graphs under this restriction have \emph{pure} dimension $k$. The next theorem gives a result for a family of \cm\ \RAAGs\ of pure dimension 5.

\begin{theorem}
\label{theorem 5-cliques string}
Let $G$ be a \RAAG\ with an associated graph containing $k$ 5-cliques attached edge-to-edge as in Figure \ref{string of 5-cliques}. Then $h(G) = 12k+2$. In particular, $G$ is \cm.
\end{theorem}

\begin{figure}[H]
\centering
\begin{tikzpicture}
\path (0.5,0) coordinate (x1);
\path (1.5,0) coordinate (x5);
\path (3.5,0) coordinate (x7);
\path (4.5,0) coordinate (x11);
\path (6.5,0) coordinate (x13);
\path (7.5,0) coordinate (x17);
\path (9.5,0) coordinate (x19);
\path (0,1) coordinate (x2);
\path (2,1) coordinate (x4);
\path (3,1) coordinate (x8);
\path (5,1) coordinate (x10);
\path (6,1) coordinate (x14);
\path (8,1) coordinate (x16);
\path (9,1) coordinate (x20);
\path (1,1.7) coordinate (x3);
\path (4,1.7) coordinate (x9);
\path (7,1.7) coordinate (x15);
\path (2.5,-0.7) coordinate (x6);
\path (5.5,-0.7) coordinate (x12);
\path (8.5,-0.7) coordinate (x18);
\foreach \i in {1,...,5}{\foreach \j in {\i,...,5}{\draw[line width = 1pt] (x\i)--(x\j);}}
\foreach \i in {4,...,8}{\foreach \j in {\i,...,8}{\draw[line width = 1pt] (x\i)--(x\j);}}
\foreach \i in {7,...,11}{\foreach \j in {\i,...,11}{\draw[line width = 1pt] (x\i)--(x\j);}}
\foreach \i in {10,...,14}{\foreach \j in {\i,...,14}{\draw[line width = 1pt] (x\i)--(x\j);}}
\foreach \i in {13,...,17}{\foreach \j in {\i,...,17}{\draw[line width = 1pt] (x\i)--(x\j);}}
\foreach \i in {16,...,20}{\foreach \j in {\i,...,20}{\draw[line width = 1pt] (x\i)--(x\j);}}
\foreach \i in {1,...,20}{\fill (x\i) circle (3pt);}
\end{tikzpicture}
\caption{A graph of 5-cliques attached edge-to-edge}
\label{string of 5-cliques}
\end{figure}
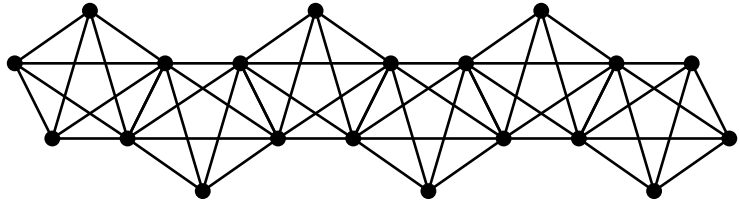

\begin{proof} 
For large $k$, computing $m_2(G)$ is impractical; since $b_4(G) = 5k$, computing $2^{5k}$ ranks using a computer program is too time consuming. However, if we compute $m_2(G)$ for $k=1,\dotsc,4$ we discover a pattern. The table below shows the calculations of the lower bound coming from the cohomology ring of $G$:
\begin{table}[H]
\centering
\begin{tabular}{|c|c|c|c|}
\hline $k$ & $b_2$ & $m_2$ & $2b_2-m_2$ \\
\hline 1 & 10 & 6 & 14 \\
\hline 2 & 19 & 12 & 26 \\
\hline 3 & 28 & 18 & 38 \\
\hline 4 & 37 & 24 & 50 \\
\hline $\vdots$ & $\vdots$ & $\vdots$ & $\vdots$ \\
\hline $k$ & $9k+1$ & $6k$ & $12k+2$ \\
\hline
\end{tabular}
\end{table}
The calculation of $b_2$ is easy to see: each 5-clique has 10 edges,  and $k-1$ edges of the graph are shared; therefore, $b_2=10k-(k-1)=9k+1$. Fortunately, there is a way to prove that the pattern for $m_2(G)$ continues for $k$ larger than 4. To see that $m_2(G)=6k$, one can find a sufficient lower bound for the dimension of the radical of (\ref{formmod2}). This will yield an upper bound for $m_2(G)$ and thus a lower bound for $h(G)$. In fact, we need only find a choice of $\alpha \in H_4(G;\Z_2)$ such that the dimension of the radical is $3k+1$. If the dimension of the radical is bounded below by $3k+1$, then the rank of the form is bounded above by $6k$. Thus, $2(9k+1)-6k = 12k+2\leq h(G)$. We will see that for each $k$, a 4-manifold $M$ can be constructed with $b_2(M)=12k+2$, which will guarantee that $m_2(G)=6k$ and that $M$ is a realizing manifold. 

A graph $G$ with $k$ 5-cliques attached edge-to-edge will have $3k+2$ vertices, $\{s_1,\dotsc, s_{3k+2} \}$.
We can label the vertices in a graph as shown in Figure \ref{labeling vertices 5-cliques}.
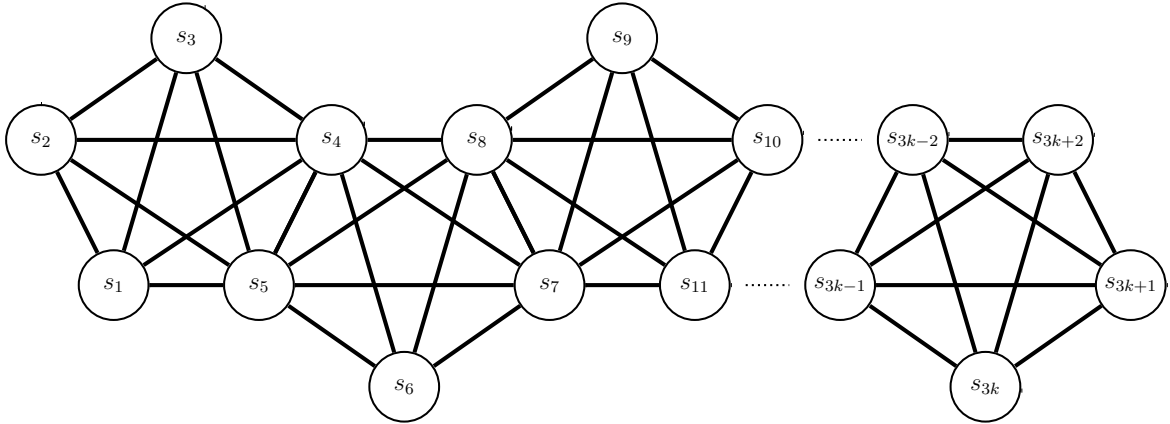
\begin{figure}[h]
\begin{center}
\resizebox{.95\textwidth}{!}{
\begin{tikzpicture}[scale=2.5]
\tikzstyle{every node}=[draw, shape=circle, inner sep=.5mm, minimum size=1.2cm, line width=1pt];
\path (0.5,0) node (x1) {$s_{1}$};
\path (1.5,0) node (x5) {$s_{5}$};
\path (3.5,0) node (x7){$s_{7}$};
\path (4.5,0) node (x11){$s_{11}$};
\path (0,1) node (x2){$s_{2}$};
\path (2,1) node (x4){$s_{4}$};
\path (3,1) node (x8){$s_{8}$};
\path (5,1) node (x10){$s_{10}$};
\path (1,1.7) node (x3){$s_{3}$};
\path (4,1.7) node (x9){$s_{9}$};
\path (2.5,-0.7) node (x6){$s_{6}$};
\path (6.5,-0.7)node (x12){$s_{3k}$}; 
\path (7.5,0) node (x13) {$s_{3k+1}$}; 
\path (7,1) node (x14) {$s_{3k+2}$}; 
\path (5.5,0) node (x15) {$s_{3k-1}$}; 
\path (6,1) node (x16) {$s_{3k-2}$};
\foreach \i in {1,...,5}{\foreach \j in {\i,...,5}{\draw[line width = 2pt] (x\i)--(x\j);}}
\foreach \i in {4,...,8}{\foreach \j in {\i,...,8}{\draw[line width = 2pt] (x\i)--(x\j);}}
\foreach \i in {7,...,11}{\foreach \j in {\i,...,11}{\draw[line width = 2pt] (x\i)--(x\j);}}
\foreach \i in {12,...,16}{\foreach \j in {\i,...,16}{\draw[line width = 2pt] (x\i)--(x\j);}}
\draw[dotted, line width = 1pt] (5.35,1)--(5.65,1) (4.85,0)--(5.15,0);
\end{tikzpicture}
}
\end{center}
\caption{A graph of $k$ 5-cliques attached edge-to-edge}
\label{labeling vertices 5-cliques}
\end{figure}
Consider the following ordering of the basis elements for $H^2(G;\Z_2)$: 
\begin{eqnarray*}
\{z_{12},\ z_{13},\ z_{14},\ z_{15},\ z_{23},\ z_{24},\ z_{25},\ z_{34},\ z_{35},\ z_{45},\ z_{46},\ z_{47},\dotsc,\ z_{(3k+1)(3k+2)} \} 
\end{eqnarray*}
For a choice of $\alpha \in H_4(G;\Z_2)$, we can find a basis for the radical of (\ref{formmod2}), as we did in Example \ref{radical example with matrix}:

If $k=1$ and $\alpha =s_{1234}+ s_{2345}$ then the following 4 elements form a basis for the radical: 
$\{z_{12}+z_{25},\ z_{13} + z_{35},\ z_{14}+z_{45},\ z_{15} \}$. One can see this by verifying that each element cupped with an arbitrary generator $z\in H^2(G; \Z_2)$ and evaluated on the class $\alpha$ is $0\bmod 2$. For example, $\langle (z_{12}+z_{25})\cup z, s_{1234}+ s_{2345} \rangle = \langle z_{12}\cup z, s_{1234} \rangle + \langle z_{25}\cup z,s_{2345} \rangle$ is equal to 0 for all $z\neq z_{34}$ and equal to $0\bmod 2$ for $z=z_{34}$. The edges represented by the radical's basis\footnote{In the case that a basis element is a sum of generators of $H^2(G;\Z_2)$, the edges of each generator are \textbf{\textcolor{blue}{highlighted}} in the graph} are \textbf{\textcolor{blue}}{highlighted} in Figure \ref{5-cliques radical k=1}.
\begin{figure}[h]
\begin{center}
\begin{minipage}{.2\linewidth} \centering
\begin{tikzpicture}
\path (0.5,0) coordinate (x1); \path (0,1) coordinate (x2); \path (1,1.7) coordinate (x3); \path (2,1) coordinate (x4); \path (1.5,0) coordinate (x5);
\foreach \i in {1,...,5}{\foreach \j in {\i,...,5}{\draw (x\i)--(x\j);}}
\draw[line width = 2pt, blue] (x1)--(x2)--(x5);
\foreach \i in {1,...,5}{\fill (x\i) circle (3pt);}
\end{tikzpicture} \\
$z_{12} + z_{25}$
\end{minipage}
\begin{minipage}{.2\linewidth} \centering
\begin{tikzpicture}
\path (0.5,0) coordinate (x1); \path (0,1) coordinate (x2); \path (1,1.7) coordinate (x3); \path (2,1) coordinate (x4); \path (1.5,0) coordinate (x5);
\foreach \i in {1,...,5}{\foreach \j in {\i,...,5}{\draw (x\i)--(x\j);}}
\draw[line width = 2pt, blue] (x1)--(x3)--(x5);
\foreach \i in {1,...,5}{\fill (x\i) circle (3pt);}
\end{tikzpicture} \\
$z_{13}+z_{35}$
\end{minipage}
\begin{minipage}{.2\linewidth} \centering
\begin{tikzpicture}
\path (0.5,0) coordinate (x1); \path (0,1) coordinate (x2); \path (1,1.7) coordinate (x3); \path (2,1) coordinate (x4); \path (1.5,0) coordinate (x5);
\foreach \i in {1,...,5}{\foreach \j in {\i,...,5}{\draw (x\i)--(x\j);}}
\draw[line width = 2pt, blue] (x1)--(x4)--(x5);
\foreach \i in {1,...,5}{\fill (x\i) circle (3pt);}
\end{tikzpicture} \\
$z_{14}+z_{45}$
\end{minipage}
\begin{minipage}{.2\linewidth} \centering
\begin{tikzpicture}
\path (0.5,0) coordinate (x1); \path (0,1) coordinate (x2); \path (1,1.7) coordinate (x3); \path (2,1) coordinate (x4); \path (1.5,0) coordinate (x5);
\foreach \i in {1,...,5}{\foreach \j in {\i,...,5}{\draw (x\i)--(x\j);}}
\draw[line width = 2pt, blue] (x1)--(x5);
\foreach \i in {1,...,5}{\fill (x\i) circle (3pt);}
\end{tikzpicture} \\
$z_{15}$
\end{minipage}
\end{center}
\caption{4 basis elements in the radical for $k=1$ and $\alpha = s_{1234}+ s_{2345}$}
\label{5-cliques radical k=1}
\end{figure}
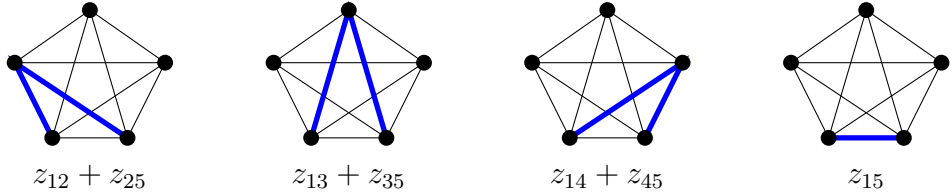

If $k=2$ and $\alpha = s_{1234}+ s_{2345}+s_{4567}+s_{5678}$, then the following 7 elements form a basis for the radical: $\{z_{12}+z_{25},\ z_{13} + z_{35},\ z_{14}+z_{45}+z_{58},\ z_{15},\ z_{46} + z_{68},\ z_{47}+ z_{78},\ z_{48} \}$.

The corresponding edges are \textbf{\textcolor{blue}{highlighted}} in Figure \ref{5-cliques radical k=2}.

\begin{figure}[h]
\begin{center}
\begin{minipage}{.4\linewidth} \centering
\begin{tikzpicture}[scale=.9]
\path (0.5,0) coordinate (x1); \path (0,1) coordinate (x2); \path (1,1.7) coordinate (x3); \path (2,1) coordinate (x4); \path (1.5,0) coordinate (x5); \path (2.5, -0.7) coordinate (x6); \path (3.5,0) coordinate (x7); \path (3,1) coordinate (x8);
\foreach \i in {1,...,5}{\foreach \j in {\i,...,5}{\draw (x\i)--(x\j);}}
\foreach \i in {4,...,8}{\foreach \j in {\i,...,8}{\draw (x\i)--(x\j);}}
\draw[line width = 2pt, blue] (x1)--(x2)--(x5) (x4)--(x7)--(x8);
\foreach \i in {1,...,8}{\fill (x\i) circle (3pt);}
\end{tikzpicture} \\
$z_{12} + z_{25}$ and $z_{47}+ z_{78}$
\end{minipage}
\begin{minipage}{.4\linewidth} \centering
\begin{tikzpicture}[scale=.9]
\path (0.5,0) coordinate (x1); \path (0,1) coordinate (x2); \path (1,1.7) coordinate (x3); \path (2,1) coordinate (x4); \path (1.5,0) coordinate (x5); \path (2.5, -0.7) coordinate (x6); \path (3.5,0) coordinate (x7); \path (3,1) coordinate (x8);
\foreach \i in {1,...,5}{\foreach \j in {\i,...,5}{\draw (x\i)--(x\j);}}
\foreach \i in {4,...,8}{\foreach \j in {\i,...,8}{\draw (x\i)--(x\j);}}
\draw[line width = 2pt, blue] (x1)--(x3)--(x5) (x4)--(x6)--(x8);
\foreach \i in {1,...,8}{\fill (x\i) circle (3pt);}
\end{tikzpicture} \\
$z_{13} + z_{35}$ and $z_{46} + z_{68}$
\end{minipage}
\vskip .3 in
\begin{minipage}{.4\linewidth} \centering
\begin{tikzpicture}[scale=.9]
\path (0.5,0) coordinate (x1); \path (0,1) coordinate (x2); \path (1,1.7) coordinate (x3); \path (2,1) coordinate (x4); \path (1.5,0) coordinate (x5); \path (2.5, -0.7) coordinate (x6); \path (3.5,0) coordinate (x7); \path (3,1) coordinate (x8);
\foreach \i in {1,...,5}{\foreach \j in {\i,...,5}{\draw (x\i)--(x\j);}}
\foreach \i in {4,...,8}{\foreach \j in {\i,...,8}{\draw (x\i)--(x\j);}}
\draw[line width = 2pt, blue] (x1)--(x5) (x4)--(x8);
\foreach \i in {1,...,8}{\fill (x\i) circle (3pt);}
\end{tikzpicture} \\
$z_{15}$ and $z_{48}$
\end{minipage}
\begin{minipage}{.4\linewidth} \centering
\begin{tikzpicture}[scale=.9]
\path (0.5,0) coordinate (x1); \path (0,1) coordinate (x2); \path (1,1.7) coordinate (x3); \path (2,1) coordinate (x4); \path (1.5,0) coordinate (x5); \path (2.5, -0.7) coordinate (x6); \path (3.5,0) coordinate (x7); \path (3,1) coordinate (x8);
\foreach \i in {1,...,5}{\foreach \j in {\i,...,5}{\draw (x\i)--(x\j);}}
\foreach \i in {4,...,8}{\foreach \j in {\i,...,8}{\draw (x\i)--(x\j);}}
\draw[line width = 2pt, blue] (x1)--(x4)--(x5)--(x8);
\foreach \i in {1,...,8}{\fill (x\i) circle (3pt);}
\end{tikzpicture} \\
$z_{14}+z_{45}+z_{58}$
\end{minipage}
\end{center}
\caption{7 basis elements in the radical for $k=2$ and $\alpha = s_{1234}+ s_{2345}+s_{4567}+s_{5678}$}
\label{5-cliques radical k=2}
\end{figure}
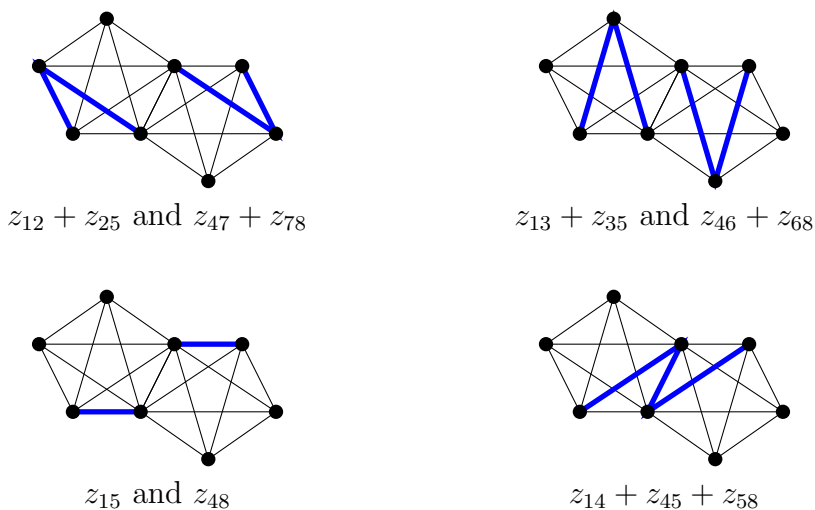

If $k=3$ and $\alpha = s_{1234}+s_{2345}+s_{4567} + s_{5678} + s_{789(10)} + s_{89(10)(11)}$, then the following 10 elements give a basis for the radical: $\{z_{12}+z_{25},\ z_{13} + z_{35},\ z_{14}+z_{45}+z_{58},\ z_{15},\ z_{46} + z_{68},\ z_{47}+ z_{78}+z_{8(11)},\ z_{48},\ z_{79}+z_{9(11)},\ z_{7(10)}+z_{(10)(11)},\ z_{7(11)} \}.$
The corresponding edges are \textbf{\textcolor{blue}{highlighted}} in Figure \ref{5-cliques radical k=3}.

\begin{figure}[h!]
\begin{center}
\begin{minipage}{.48\linewidth} \centering
\begin{tikzpicture}[scale=.9]
\path (0.5,0) coordinate (x1); \path (0,1) coordinate (x2); \path (1,1.7) coordinate (x3); \path (2,1) coordinate (x4); \path (1.5,0) coordinate (x5); \path (2.5, -0.7) coordinate (x6); \path (3.5,0) coordinate (x7); \path (3,1) coordinate (x8); \path (4,1.7) coordinate (x9); \path (5,1) coordinate (x10); \path (4.5,0) coordinate (x11); 
\foreach \i in {1,...,5}{\foreach \j in {\i,...,5}{\draw (x\i)--(x\j);}}
\foreach \i in {4,...,8}{\foreach \j in {\i,...,8}{\draw (x\i)--(x\j);}}
\foreach \i in {7,...,11}{\foreach \j in {\i,...,11}{\draw (x\i)--(x\j);}}
\draw[line width = 2pt, blue] (x1)--(x2)--(x5) (x7)--(x10)--(x11);
\foreach \i in {1,...,11}{\fill (x\i) circle (3pt);}
\end{tikzpicture} \\
$z_{12} + z_{25}$ and $z_{7(10)}+ z_{(10)(11)}$
\end{minipage}
\begin{minipage}{.48\linewidth} \centering
\begin{tikzpicture}[scale=.9]
\path (0.5,0) coordinate (x1); \path (0,1) coordinate (x2); \path (1,1.7) coordinate (x3); \path (2,1) coordinate (x4); \path (1.5,0) coordinate (x5); \path (2.5, -0.7) coordinate (x6); \path (3.5,0) coordinate (x7); \path (3,1) coordinate (x8); \path (4,1.7) coordinate (x9); \path (5,1) coordinate (x10); \path (4.5,0) coordinate (x11); 
\foreach \i in {1,...,5}{\foreach \j in {\i,...,5}{\draw (x\i)--(x\j);}}
\foreach \i in {4,...,8}{\foreach \j in {\i,...,8}{\draw (x\i)--(x\j);}}
\foreach \i in {7,...,11}{\foreach \j in {\i,...,11}{\draw (x\i)--(x\j);}}
\draw[line width = 2pt, blue] (x1)--(x3)--(x5) (x4)--(x6)--(x8) (x7)--(x9)--(x11);
\foreach \i in {1,...,11}{\fill (x\i) circle (3pt);}
\end{tikzpicture} \\
$z_{13} + z_{35}$,\ $z_{46} + z_{68}$, and $z_{79} + z_{9(11)}$
\end{minipage}
\vskip .3 in
\begin{minipage}{.48\linewidth} \centering
\begin{tikzpicture}[scale=.9]
\path (0.5,0) coordinate (x1); \path (0,1) coordinate (x2); \path (1,1.7) coordinate (x3); \path (2,1) coordinate (x4); \path (1.5,0) coordinate (x5); \path (2.5, -0.7) coordinate (x6); \path (3.5,0) coordinate (x7); \path (3,1) coordinate (x8); \path (4,1.7) coordinate (x9); \path (5,1) coordinate (x10); \path (4.5,0) coordinate (x11); 
\foreach \i in {1,...,5}{\foreach \j in {\i,...,5}{\draw (x\i)--(x\j);}}
\foreach \i in {4,...,8}{\foreach \j in {\i,...,8}{\draw (x\i)--(x\j);}}
\foreach \i in {7,...,11}{\foreach \j in {\i,...,11}{\draw (x\i)--(x\j);}}
\draw[line width = 2pt, blue] (x1)--(x5) (x4)--(x8) (x7)--(x11);
\foreach \i in {1,...,11}{\fill (x\i) circle (3pt);}
\end{tikzpicture} \\
$z_{15}$,\ $z_{48}$,\ and  $z_{7(11)}$
\end{minipage}
\begin{minipage}{.48\linewidth} \centering
\begin{tikzpicture}[scale=.9]
\path (0.5,0) coordinate (x1); \path (0,1) coordinate (x2); \path (1,1.7) coordinate (x3); \path (2,1) coordinate (x4); \path (1.5,0) coordinate (x5); \path (2.5, -0.7) coordinate (x6); \path (3.5,0) coordinate (x7); \path (3,1) coordinate (x8); \path (4,1.7) coordinate (x9); \path (5,1) coordinate (x10); \path (4.5,0) coordinate (x11); 
\foreach \i in {1,...,5}{\foreach \j in {\i,...,5}{\draw (x\i)--(x\j);}}
\foreach \i in {4,...,8}{\foreach \j in {\i,...,8}{\draw (x\i)--(x\j);}}
\foreach \i in {7,...,11}{\foreach \j in {\i,...,11}{\draw (x\i)--(x\j);}}
\draw[line width = 2pt, blue] (x1)--(x4)--(x5)--(x8);
\draw[line width = 2pt, OliveGreen] (x4)--(x7)--(x8)--(x11);
\foreach \i in {1,...,11}{\fill (x\i) circle (3pt);}
\end{tikzpicture} \\
$z_{14}+z_{45}+z_{58}$ and $z_{47} + z_{78} + z_{8(11)}$
\end{minipage}
\end{center}
\caption{10 basis elements in the radical for $k=3$ and $\alpha=s_{1234}+s_{2345}+s_{4567} + s_{5678} + s_{789(10)} + s_{89(10)(11)}$}
\label{5-cliques radical k=3}
\end{figure}
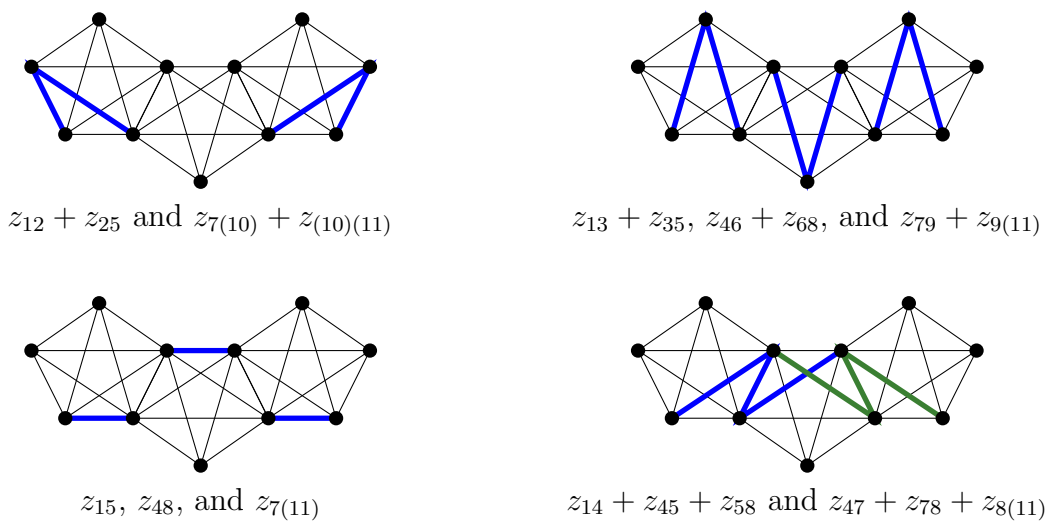

We are developing a pattern to determine a basis for the radical of (\ref{formmod2}) for any $k$. First, we order the basis elements of $H_4(G;\Z_2)$ as follows:
\begin{equation*}
\{s_{1234},\ s_{1235},\ s_{1245},\ s_{1345},\ s_{2345},\ s_{4567},\ s_{4568},\ s_{4578},\ s_{4678},\ s_{5678}, \dotsc,\ s_{(3k-1)(3k)(3k+1)(3k+2)}\} 
\end{equation*}
Note that each 5-clique has five basis elements in $H_4(G;\Z_2)$, ordered consecutively in the set above. Consider the following choice of $\alpha$, in which only the first and last basis elements of each 5-clique are nonzero:
\begin{equation*}
\alpha = s_{1234} + s_{2345} + s_{4567} + s_{5678} + \dotsc + s_{(3k-2)(3k-1)(3k)(3k+1)} + s_{(3k-1)(3k)(3k+1)(3k+2)}.
\end{equation*} 
Notice that this choice of $\alpha$ agrees with the previous choices for small $k$. Based on the developed pattern, we can find a basis for the radical for any $k$: \\

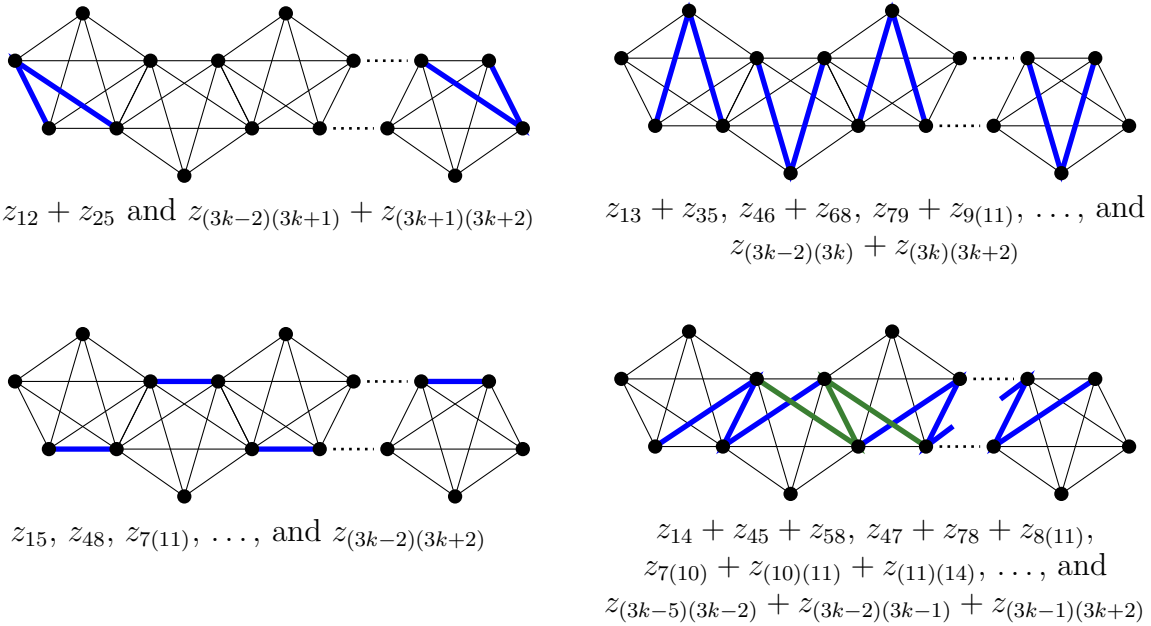
\begin{center}
\begin{minipage}{.48\linewidth} \centering
\begin{tikzpicture}[scale=.9]
\path (0.5,0) coordinate (x1); \path (0,1) coordinate (x2); \path (1,1.7) coordinate (x3); \path (2,1) coordinate (x4); \path (1.5,0) coordinate (x5); \path (2.5, -0.7) coordinate (x6); \path (3.5,0) coordinate (x7); \path (3,1) coordinate (x8); \path (4,1.7) coordinate (x9); \path (5,1) coordinate (x10); \path (4.5,0) coordinate (x11);  \path (6.5,-0.7) coordinate (x12); \path (7.5,0) coordinate (x13); \path (7,1) coordinate (x14); \path (5.5,0) coordinate (x15); \path (6,1) coordinate (x16);
\foreach \i in {1,...,5}{\foreach \j in {\i,...,5}{\draw (x\i)--(x\j);}}
\foreach \i in {4,...,8}{\foreach \j in {\i,...,8}{\draw (x\i)--(x\j);}}
\foreach \i in {7,...,11}{\foreach \j in {\i,...,11}{\draw (x\i)--(x\j);}}
\foreach \i in {12,...,16}{\foreach \j in {\i,...,16}{\draw (x\i)--(x\j);}}
\draw[dotted, line width = 1pt] (5.2,1)--(5.8,1) (4.7,0)--(5.3,0);
\draw[line width = 2pt, blue] (x1)--(x2)--(x5) (x14)--(x13)--(x16);
\foreach \i in {1,...,16}{\fill (x\i) circle (3pt);}
\end{tikzpicture} \\
$z_{12}+z_{25}$ and  $z_{(3k-2)(3k+1)}+z_{(3k+1)(3k+2)}$ \\
\phantom{$s_{more}$}
\end{minipage}
\begin{minipage}{.48\linewidth} \centering
\begin{tikzpicture}[scale=.9]
\path (0.5,0) coordinate (x1); \path (0,1) coordinate (x2); \path (1,1.7) coordinate (x3); \path (2,1) coordinate (x4); \path (1.5,0) coordinate (x5); \path (2.5, -0.7) coordinate (x6); \path (3.5,0) coordinate (x7); \path (3,1) coordinate (x8); \path (4,1.7) coordinate (x9); \path (5,1) coordinate (x10); \path (4.5,0) coordinate (x11);  \path (6.5,-0.7) coordinate (x12); \path (7.5,0) coordinate (x13); \path (7,1) coordinate (x14); \path (5.5,0) coordinate (x15); \path (6,1) coordinate (x16);
\foreach \i in {1,...,5}{\foreach \j in {\i,...,5}{\draw (x\i)--(x\j);}}
\foreach \i in {4,...,8}{\foreach \j in {\i,...,8}{\draw (x\i)--(x\j);}}
\foreach \i in {7,...,11}{\foreach \j in {\i,...,11}{\draw (x\i)--(x\j);}}
\foreach \i in {12,...,16}{\foreach \j in {\i,...,16}{\draw (x\i)--(x\j);}}
\draw[dotted, line width = 1pt] (5.2,1)--(5.8,1) (4.7,0)--(5.3,0);
\draw[line width = 2pt, blue] (x1)--(x3)--(x5) (x4)--(x6)--(x8) (x7)--(x9)--(x11) (x14)--(x12)--(x16);
\foreach \i in {1,...,16}{\fill (x\i) circle (3pt);}
\end{tikzpicture} \\
 $z_{13}+z_{35}$,\ $z_{46}+z_{68}$,\ $z_{79}+z_{9(11)}$, $\dotsc$,\ and $z_{(3k-2)(3k)}+z_{(3k)(3k+2)}$
\end{minipage}
\vskip .3 in
\begin{minipage}{.48\linewidth} \centering
\begin{tikzpicture}[scale=.9]
\path (0.5,0) coordinate (x1); \path (0,1) coordinate (x2); \path (1,1.7) coordinate (x3); \path (2,1) coordinate (x4); \path (1.5,0) coordinate (x5); \path (2.5, -0.7) coordinate (x6); \path (3.5,0) coordinate (x7); \path (3,1) coordinate (x8); \path (4,1.7) coordinate (x9); \path (5,1) coordinate (x10); \path (4.5,0) coordinate (x11);  \path (6.5,-0.7) coordinate (x12); \path (7.5,0) coordinate (x13); \path (7,1) coordinate (x14); \path (5.5,0) coordinate (x15); \path (6,1) coordinate (x16);
\foreach \i in {1,...,5}{\foreach \j in {\i,...,5}{\draw (x\i)--(x\j);}}
\foreach \i in {4,...,8}{\foreach \j in {\i,...,8}{\draw (x\i)--(x\j);}}
\foreach \i in {7,...,11}{\foreach \j in {\i,...,11}{\draw (x\i)--(x\j);}}
\foreach \i in {12,...,16}{\foreach \j in {\i,...,16}{\draw (x\i)--(x\j);}}
\draw[dotted, line width = 1pt] (5.2,1)--(5.8,1) (4.7,0)--(5.3,0);
\draw[line width = 2pt, blue] (x1)--(x5) (x4)--(x8) (x7)--(x11) (x14)--(x16);
\foreach \i in {1,...,16}{\fill (x\i) circle (3pt);}
\end{tikzpicture} \\
 $z_{15}$,\ $z_{48}$,\ $z_{7(11)}$, $\dotsc$,\ and $z_{(3k-2)(3k+2)}$
\newline \phantom{$s_{more}$} 
\newline \phantom{$s_{more}$} 
\end{minipage}
\begin{minipage}{.48\linewidth} \centering
\begin{tikzpicture}[scale=.9]
\path (0.5,0) coordinate (x1); \path (0,1) coordinate (x2); \path (1,1.7) coordinate (x3); \path (2,1) coordinate (x4); \path (1.5,0) coordinate (x5); \path (2.5, -0.7) coordinate (x6); \path (3.5,0) coordinate (x7); \path (3,1) coordinate (x8); \path (4,1.7) coordinate (x9); \path (5,1) coordinate (x10); \path (4.5,0) coordinate (x11);  \path (6.5,-0.7) coordinate (x12); \path (7.5,0) coordinate (x13); \path (7,1) coordinate (x14); \path (5.5,0) coordinate (x15); \path (6,1) coordinate (x16);
\foreach \i in {1,...,5}{\foreach \j in {\i,...,5}{\draw (x\i)--(x\j);}}
\foreach \i in {4,...,8}{\foreach \j in {\i,...,8}{\draw (x\i)--(x\j);}}
\foreach \i in {7,...,11}{\foreach \j in {\i,...,11}{\draw (x\i)--(x\j);}}
\foreach \i in {12,...,16}{\foreach \j in {\i,...,16}{\draw (x\i)--(x\j);}}
\draw[dotted, line width = 1pt] (5.2,1)--(5.8,1) (4.7,0)--(5.3,0);
\draw[line width = 2pt, blue] (x1)--(x4)--(x5)--(x8) (x7)--(x10)--(x11)--(4.9,.3) (x14)--(x15)--(x16)--(5.6,.7);
\draw[line width = 2pt, OliveGreen] (x4)--(x7)--(x8)--(x11) ;
\foreach \i in {1,...,16}{\fill (x\i) circle (3pt);}
\end{tikzpicture} \\
 $z_{14}+z_{45}+z_{58}$,\ $z_{47}+z_{78}+z_{8(11)}$,\ $z_{7(10)}+z_{(10)(11)}+z_{(11)(14)}$, $\dotsc$,\ and $z_{(3k-5)(3k-2)}+z_{(3k-2)(3k-1)}+z_{(3k-1)(3k+2)}$
\end{minipage}
\end{center}

\noindent We conclude there are $2 + k + k + (k-1) = 3k+1$ elements in this basis for the radical. As previously noted, this implies the rank of the form for our choice of $\alpha$ is $6k$. \\

The realizing manifold construction for the upper bound is quite straightforward. We start with two copies of a 4-torus and $k-1$ copies of $T^2 \times \Sigma_2$. The required surgeries are most easily explained with an example. Let $k=4$. Let $\pi_1$ of the two 4-tori be generated by $\{x_1,x_2,x_3,x_4\}$ and $\{y_1,y_2,y_3,y_4\}$, and let $\pi_1$ of the three copies of $T^2 \times \Sigma_2$ be generated by $\{a_1,a_2\}$ and $\{b_1,b_2,b_3,b_4\}$, $\{c_1,c_2\}$ and $\{d_1,d_2,d_3,d_4\}$, and $\{t_1,t_2\}$ and $\{s_1,s_2,s_3,s_4\}$. (See Figure \ref{pieces of 5} for the graphical representations of $\pi_1$.) Before surgeries, $b_2(\#2T^4\#3(T^2 \times \Sigma_2))=42$. 

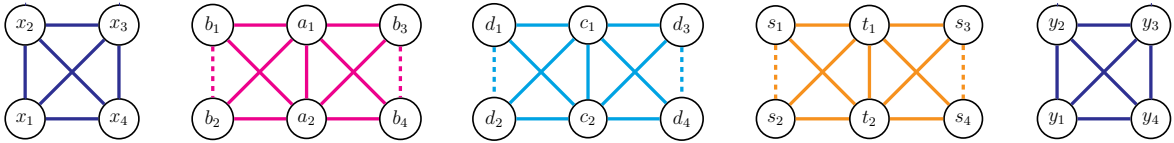
\begin{figure}[h]
\centering
\resizebox{.95\textwidth}{!}{
\begin{tikzpicture}[scale=2]
\tikzstyle{every node}=[draw, shape=circle, line width = 1pt];
\path (0,0) node (x1) {$x_1$};
\path (0,1) node (x2) {$x_2$};
\path (1,1) node (x3){$x_3$};
\path (1,0) node (x4) {$x_4$};
\path (2,0) node (y2) {$b_2$};
\path (2,1) node (y1){$b_1$};
\path (3,1) node (z1) {$a_1$};
\path (3,0) node (z2) {$a_2$};
\path (4,0) node (y4){$b_4$};
\path (4,1) node (y3) {$b_3$};
\path (5,1) node (w1) {$d_1$};
\path (5,0) node (w2){$d_2$};
\path (6,0) node (v2) {$c_2$};
\path (6,1) node (v1) {$c_1$};
\path (7,1) node (w3){$d_3$};
\path (7,0) node (w4) {$d_4$};
\path (8,0) node (s2) {$s_2$};
\path (8,1) node (s1){$s_1$};
\path (9,1) node (t1) {$t_1$};
\path (9,0) node (t2) {$t_2$};
\path (10,0) node (s4){$s_4$};
\path (10,1) node (s3) {$s_3$};
\path (11,1) node (p2) {$y_2$};
\path (11,0) node (p1){$y_1$};
\path (12,0) node (p4) {$y_4$};
\path (12,1) node (p3) {$y_3$};
\foreach \i in {1,...,4}{\foreach \j in {\i,...,4}{\draw[line width = 2pt,Blue] (x\i)--(x\j) (p\i)--(p\j);}}
\foreach \i in {1,2,3,4}{\foreach \j in {1,2}{\draw[line width = 2pt,Magenta] (y\i)--(z\j) ;}}
\foreach \i in {1,2,3,4}{\foreach \j in {1,2}{\draw[line width = 2pt,Cerulean] (w\i)--(v\j) ;}}
\foreach \i in {1,2,3,4}{\foreach \j in {1,2}{\draw[line width = 2pt,BurntOrange] (s\i)--(t\j);}}
\draw[line width = 2pt,Magenta] (z1)--(z2); \draw[line width = 2pt, dashed, Magenta] (y1)--(y2) (y3)--(y4);
\draw[line width = 2pt,Cerulean] (v1)--(v2);  \draw[line width = 2pt, dashed, Cerulean] (w1)--(w2) (w3)--(w4);
\draw[line width = 2pt,BurntOrange] (t1)--(t2); \draw[line width = 2pt, dashed, BurntOrange] (s1)--(s2) (s3)--(s4);
\end{tikzpicture}
}
\vskip .1 in
\caption{A graph representing $\pi_1$ of the products of surfaces necessary for this 4-manifold construction}
\label{pieces of 5}
\end{figure}

Perform surgery on the connected sum to induce the following 12 identifications:
\begin{eqnarray*}
x_1=b_2 \quad & b_3=c_1 & \quad d_4=t_2 \\
x_3=b_1 \quad & b_4=d_2 & \quad y_1 = s_4 \\
x_4=a_2 \quad & c_2=s_2 & \quad y_2=t_1 \\
a_1=d_1\quad & d_3 = s_1& \quad y_3=s_3 
\end{eqnarray*}
\noindent These do not change $b_2$. Lastly, perform surgeries to induce the following four commutator relations: $[x_2,a_1]=1$, $[x_4,c_2]=1$, $[b_3,y_2]=1$, and $[d_4,y_4]=1$. After these surgeries, $b_2 = 42 + 4(2) = 50$. The resulting graph associated to $\pi_1$ of this realizing manifold is shown in Figure \ref{4 5-cliques}. 

\begin{figure}[h]
\centering
\begin{tikzpicture}[scale=2]
\tikzstyle{every node}=[draw, shape=circle, inner sep=.5mm, minimum size=.7cm, line width=1pt];
\path (0.5,0) node (x1) {$x_1$};
\path (0.5,0) node (b2) {};
\path (1.5,0) node (x4) {$x_4$};
\path (1.5,0) node (a2) {};
\path (3.5,0) node (c2) {$c_2$};
\path (3.5,0) node (w2) {};
\path (4.5,0) node (d4) {$d_4$};
\path (4.5,0) node (z2) {};
\path (6.5,0) node (y4) {$y_4$};
\path (0,1) node (x2) {$x_2$};
\path (2,1) node (a1) {$a_1$};
\path (2,1) node (d1) {};
\path (3,1) node (b3) {$b_3$};
\path (3,1) node (c1) {};
\path (5,1) node (y2) {$y_2$};
\path (5,1) node (z1) {};
\path (6,1) node (y3) {$y_3$};
\path (6,1) node (w3) {};
\path (1,1.7) node (x3) {$x_3$};
\path (1,1.7) node (b1) {};
\path (4,1.7) node (d3){$d_3$};
\path (4,1.7) node (w1){};
\path (2.5,-0.7) node (b4){$b_4$};
\path (2.5,-0.7) node (d2){};
\path (5.5,-0.7) node (y1){$y_1$};
\path (5.5,-0.7) node (w4){};
\foreach \i in {1,...,4}{\foreach \j in {\i,...,4}{\draw[line width = 2pt,Blue] (x\i)--(x\j) (y\i)--(y\j);}}
\foreach \i in {1,2,3,4}{\foreach \j in {1,2}{\draw[line width = 2pt,Magenta] (b\i)--(a\j) ;}}
\foreach \i in {1,2,3,4}{\foreach \j in {1,2}{\draw[line width = 2pt,Cerulean] (d\i)--(c\j) ;}}
\foreach \i in {1,2,3,4}{\foreach \j in {1,2}{\draw[line width = 2pt,BurntOrange] (w\i)--(z\j);}}
\draw[line width = 2pt,Magenta] (a1)--(a2); \draw[line width = 2pt, loosely dashed, Magenta] (b1)--(b2) (b3)--(b4);
\draw[line width = 2pt,Cerulean] (c1)--(c2);  \draw[line width = 2pt, loosely dashed, Cerulean] (d1)--(d2) (d3)--(d4);
\draw[line width = 2pt,BurntOrange] (z1)--(z2); \draw[line width = 2pt, loosely dashed, BurntOrange] (w1)--(w2) (w3)--(w4);
\draw[line width = 2pt,OliveGreen] (x2)--(a1) (x4)--(c2) (b3)--(y2) (d4)--(y4);
\end{tikzpicture}
\caption{A graph representing $\pi_1$ of a realizing manifold for four 5-cliques attached edge-to-edge}
\label{4 5-cliques}
\end{figure}
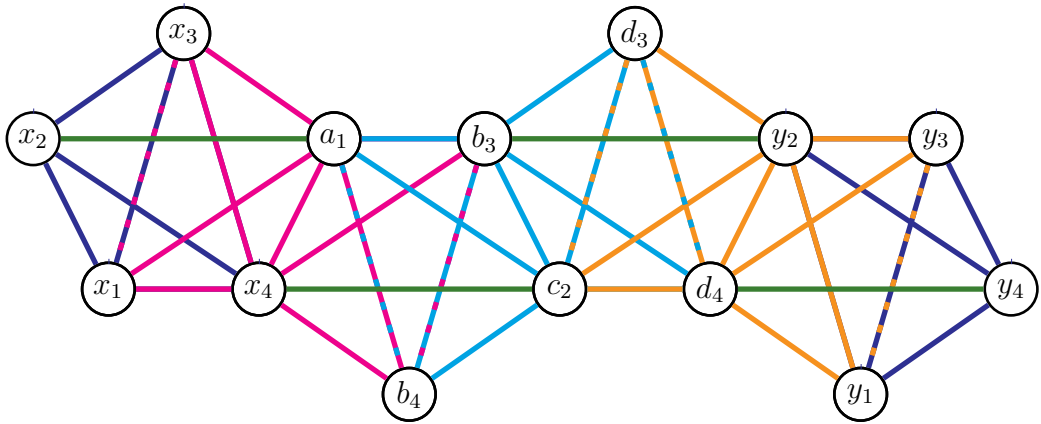

For a graph with $k$ 5-cliques, $k-1$ copies of $T^2 \times \Sigma_2$ are required, and the necessary identification surgeries follow the pattern described by the example. Each copy of $T^4$ adds 6 to the count of $b_2$, and each copy of $T^2 \times \Sigma_2$ adds 10. Lastly, $k$ commutator surgeries are necessary and increase $b_2$ by 2 each. The resulting 4-manifold $M$ has $b_2(M)=6(2)+10(k-1)+2k = 12k+2$.
\end{proof}

Recall Proposition \ref{evenform} which states that $m_2(G)$ is even for a \RAAG\ $G$. The following theorem relies on this proposition and calculates $h$ for a family of graphs of pure dimension $6$.

\begin{theorem}
\label{theorem 6-cliques string}
Let $G$ be a \RAAG\ with an associated graph containing $k$ 6-cliques attached edge-to-edge as in Figure \ref{string of 6-cliques}. Then $h(G) = 14k+2$. 
\end{theorem}

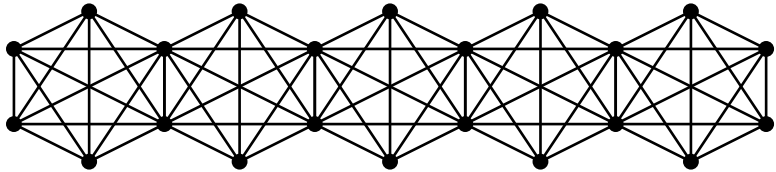
\begin{figure}[H]
\centering
\begin{tikzpicture}
\foreach \i in {1,5,9,13,17,21}{\path (.5*\i-.5,1) coordinate (x\i);}
\foreach \i in {2,6,10,14,18,22}{\path (.5*\i-1,0) coordinate (x\i);}
\foreach \i in {3,7,11,15,19}{\path (.5*\i-.5,1.5) coordinate (x\i);}
\foreach \i in {4,8,12,16,20}{\path (.5*\i-1,-0.5) coordinate (x\i);}
\foreach \i in {1,...,6}{\foreach \j in {\i,...,6}{\draw[line width = 1pt] (x\i)--(x\j);}}
\foreach \i in {5,...,10}{\foreach \j in {\i,...,10}{\draw[line width = 1pt] (x\i)--(x\j);}}
\foreach \i in {9,...,14}{\foreach \j in {\i,...,14}{\draw[line width = 1pt] (x\i)--(x\j);}}
\foreach \i in {13,...,18}{\foreach \j in {\i,...,18}{\draw[line width = 1pt] (x\i)--(x\j);}}
\foreach \i in {17,...,22}{\foreach \j in {\i,...,22}{\draw[line width = 1pt] (x\i)--(x\j);}}
\foreach \i in {1,...,22}{\fill (x\i) circle (3pt);}
\end{tikzpicture}
\caption{A graph of 6-cliques attached edge-to-edge}
\label{string of 6-cliques}
\end{figure}

\begin{proof}
Let $G$ be a \RAAG\ with an associated graph of $k$ 6-cliques attached edge-to-edge as in Figure \ref{string of 6-cliques}. Each 6-clique has 15 edges and $k-1$ edges in the graph are shared, so $b_2(G)=15k-(k-1)=14k+1$. Because $b_2(G)$ is odd, we know that $14k+2 \leq h(G)$ by Proposition \ref{evenform}.

The construction of a realizing manifold for $h(G)$ contains $k-1$ copies of $T^2 \times \Sigma_3$ as well as one copy of $T^2 \times \Sigma_2$ and one 4-torus. As in the proof of Theorem \ref{theorem 5-cliques string}, we will see the pattern of necessary identification surgeries with an example. Let $k=3$. Start with $(T^2 \times \Sigma_2) \# (T^2 \times \Sigma_3) \# (T^2 \times \Sigma_3) \# T^4$, where $\pi_1(T^2\times \Sigma_2)$ is generated by $\{x_1,x_2\}$ and $\{y_1,y_2,y_3,y_4\}$ and $\pi_1(T^4)$ is generated by $\{u_1,u_2,u_3,u_4\}$. Let $\{s_1,s_2\}$ and $\{t_1,\dotsc,t_6\}$ generate $\pi_1$ of the first copy of $T^2 \times \Sigma_3$ and let $\{w_1,w_2\}$ and $\{z_1,\dotsc,z_6\}$ generate $\pi_1$ of the second. Figure \ref{pieces of 6} shows the graphical representation of the fundamental group of each summand of the 4-manifold. 

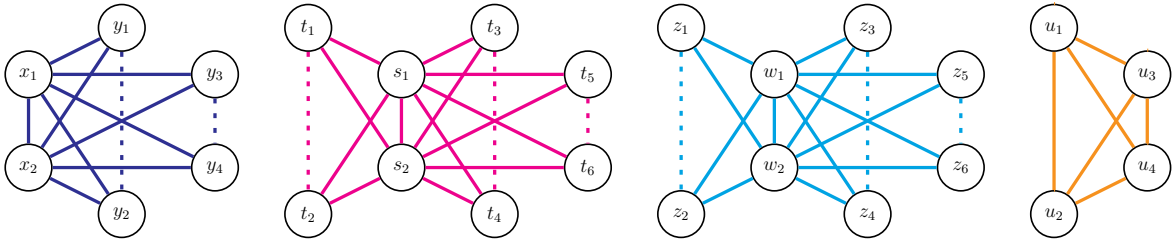
\begin{figure}[h]
\centering
\resizebox{.95\textwidth}{!}{
\begin{tikzpicture}[scale=2]
\tikzstyle{every node}=[draw, shape=circle, minimum size = 1cm, line width = 1pt];
\path (0,1) node (x1) {$x_1$};
\path (0,0) node (x2) {$x_2$};
\path (1,1.5) node (y1) {$y_1$};
\path (1,-0.5) node (y2) {$y_2$};
\path (2,1) node (y3) {$y_3$};
\path (2,0) node (y4) {$y_4$};
\path (3,1.5) node (t1) {$t_1$};
\path (3,-0.5) node (t2) {$t_2$};
\path (4,1) node (s1) {$s_1$};
\path (4,0) node (s2) {$s_2$};
\path (5,1.5) node (t3) {$t_3$};
\path (5,-0.5) node (t4) {$t_4$};
\path (6,1) node (t5) {$t_5$};
\path (6,0) node (t6) {$t_6$};
\path (7,1.5) node (z1) {$z_1$};
\path (7,-0.5) node (z2) {$z_2$};
\path (8,1) node (w1) {$w_1$};
\path (8,0) node (w2) {$w_2$};
\path (9,1.5) node (z3) {$z_3$};
\path (9,-0.5) node (z4) {$z_4$};
\path (10,1) node (z5) {$z_5$};
\path (10,0) node (z6) {$z_6$};
\path (11,1.5) node (u1) {$u_1$};
\path (11,-0.5) node (u2) {$u_2$};
\path (12,1) node (u3) {$u_3$};
\path (12,0) node (u4) {$u_4$};
\foreach \i in {1,...,4}{\foreach \j in {1,2}{\draw[line width = 2pt,Blue] (y\i)--(x\j);}} \draw[line width = 2pt,Blue] (x1)--(x2); \draw[line width = 2pt, loosely dashed,Blue] (y1)--(y2) (y3)--(y4);
\foreach \i in {1,...,6}{\foreach \j in {1,2}{\draw[line width = 2pt,Magenta] (t\i)--(s\j);}} \draw[line width = 2pt,Magenta] (s1)--(s2); \draw[line width = 2pt, loosely dashed,Magenta] (t1)--(t2) (t3)--(t4) (t5)--(t6);
\foreach \i in {1,...,6}{\foreach \j in {1,2}{\draw[line width = 2pt,Cerulean] (z\i)--(w\j);}} \draw[line width = 2pt,Cerulean] (w1)--(w2); \draw[line width = 2pt, loosely dashed,Cerulean] (z1)--(z2) (z3)--(z4) (z5)--(z6);
\foreach \i in {1,...,4}{\foreach \j in {\i,...,4}{\draw[line width = 2pt, BurntOrange] (u\i)--(u\j);}}
\end{tikzpicture}
}
\caption{A graph representing $\pi_1$ of the products of surfaces necessary for this 4-manifold construction}
\label{pieces of 6}
\end{figure}

Perform surgeries to induce the following identifications:
\begin{eqnarray*}
y_1=t_1 \quad & t_3 = z_1 & \quad z_3 = u_1 \\
y_2=t_2 \quad & t_4 = z_2 & \quad z_4 = u_2 \\
y_3 = s_1 \quad & t_5 = w_1 &  \quad z_5 = u_3 \\
y_4 = s_2 \quad & t_6 = w_2 & \quad z_6 = u_4
\end{eqnarray*}
These surgeries yield a 4-manifold with the correct fundamental group. (See Figure \ref{3 6-cliques}.) This example shows the identification surgery pattern one would use to construct a realizing manifold for any $k$. The copy of $T^2 \times \Sigma_2$ adds 10 to the count of $b_2$, each copy of $T^2 \times \Sigma_3$ adds 14, and the 4-torus adds 6. The resulting manifold, $M$, will have $b_2(M) = 10(1) + 14(k-1) + 6 = 14k+2$.
\end{proof}

\begin{figure}[h]
\centering
\begin{tikzpicture}[scale=2]
\tikzstyle{every node}=[draw, shape=circle, inner sep=.5mm, minimum size=.7cm, line width=1pt];
\path (0,1) node (x1) {$x_1$};
\path (0,0) node (x2) {$x_2$};
\path (1,1.5) node (y1) {$y_1$};
\path (1,-0.5) node (y2) {$y_2$};
\path (2,1) node (y3) {$y_3$};
\path (2,0) node (y4) {$y_4$};
\path (1,1.5) node (t1) {};
\path (1,-0.5) node (t2) {};
\path (2,1) node (s1) {};
\path (2,0) node (s2) {};
\path (3,1.5) node (t3) {$t_3$};
\path (3,-0.5) node (t4) {$t_4$};
\path (4,1) node (t5) {$t_5$};
\path (4,0) node (t6) {$t_6$};
\path (3,1.5) node (z1) {};
\path (3,-0.5) node (z2) {};
\path (4,1) node (w1) {};
\path (4,0) node (w2) {};
\path (5,1.5) node (z3) {$z_3$};
\path (5,-0.5) node (z4) {$z_4$};
\path (6,1) node (z5) {$z_5$};
\path (6,0) node (z6) {$z_6$};
\path (5,1.5) node (u1) {};
\path (5,-0.5) node (u2) {};
\path (6,1) node (u3) {};
\path (6,0) node (u4) {};
\foreach \i in {1,...,4}{\foreach \j in {1,2}{\draw[line width = 2pt,Blue] (y\i)--(x\j);}} \draw[line width = 2pt,Blue] (x1)--(x2); 
\foreach \i in {1,...,6}{\foreach \j in {1,2}{\draw[line width = 2pt,Magenta] (t\i)--(s\j);}} \draw[line width = 2pt,Magenta] (s1)--(s2); \draw[line width = 2pt, dashed,Magenta] (t1)--(t2) (t3)--(t4) (t5)--(t6);
\foreach \i in {1,...,4}{\foreach \j in {\i,...,4}{\draw[line width = 2pt, BurntOrange] (u\i)--(u\j);}}
\foreach \i in {1,...,6}{\foreach \j in {1,2}{\draw[line width = 2pt,Cerulean] (z\i)--(w\j);}} \draw[line width = 2pt,Cerulean] (w1)--(w2); 
\draw[line width = 2pt, dashed, dash pattern= on 4pt off 4pt, dash phase = 4pt, Magenta] (t1)--(t2) (t3)--(t4) (t5)--(t6);
\draw[line width = 2pt, dashed, dash pattern= on 4pt off 4pt, Cerulean] (z1)--(z2) (z3)--(z4) (z5)--(z6);
\draw[line width = 2pt, dashed, dash pattern= on 4pt off 4pt,Blue] (y1)--(y2) (y3)--(y4);
\end{tikzpicture}
\caption{A graph representing $\pi_1$ of a realizing manifold for three 6-cliques attached edge-to-edge}
\label{3 6-cliques}
\end{figure}
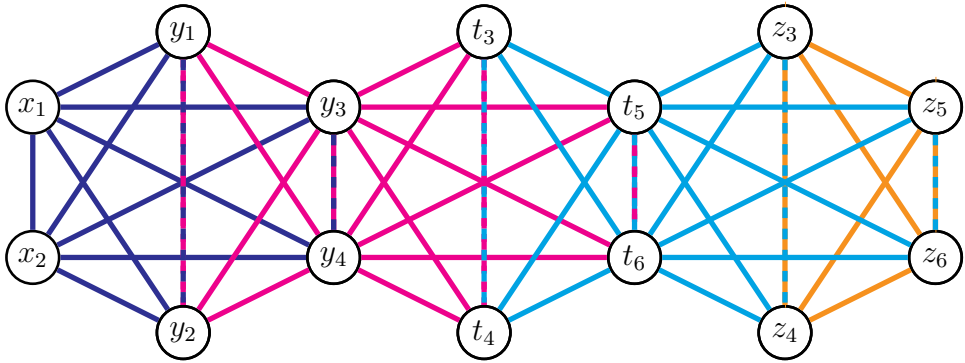

The last family of \RAAGs\ we will explore in this paper is a family of graphs of pure dimension 7.

\begin{theorem}
\label{theorem 7-cliques string}
Let $G$ be a \RAAG\ with an associated graph containing $k$ 7-cliques attached edge-to-edge as in Figure \ref{string of 7-cliques}. Then $h(G) = 20k+2$. 
\end{theorem}

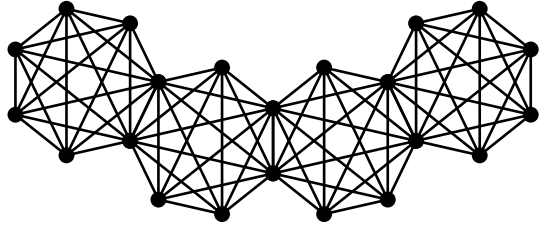
\begin{figure}[H]
\centering
\begin{tikzpicture}
\foreach \i in {1,...,7}{\path (\i*51.429:1cm) coordinate (x\i);}
\begin{scope}[shift={(-25.714:1.8019cm)}]
	\foreach \i in {1,...,7}{\path (\i*51.429+25.714:1cm) coordinate (y\i);}
		\begin{scope}[shift={(0:1.8019cm)}]
			\foreach \i in {1,...,7}{\path (\i*51.429:1cm) coordinate (z\i);}
				\begin{scope}[shift={(25.714:1.8019cm)}]
					\foreach \i in {1,...,7}{\path (\i*51.429+25.714:1cm) coordinate (w\i);}
				\end{scope}		
		\end{scope}			
\end{scope}	
\foreach \i in {1,...,7}{
	\foreach \j in {\i,...,7}{\draw[line width=1pt] (x\i)--(x\j) (y\i)--(y\j) (z\i)--(z\j) (w\i)--(w\j);}
	\fill (x\i) circle (3pt); \fill (y\i) circle (3pt); \fill (z\i) circle (3pt); \fill (w\i) circle (3pt);
	}	
\end{tikzpicture}
\caption{A graph of 7-cliques attached edge-to-edge}
\label{string of 7-cliques}
\end{figure}

\begin{proof}
Let $G$ be a \RAAG\ with an associated graph of $k$ 7-cliques attached edge-to-edge as in Figure \ref{string of 7-cliques}. Each 7-clique has 21 edges and $k-1$ edges in the graph are shared, so $b_2(G)=21k-(k-1)=20k+1$. Because $b_2(G)$ is odd, Proposition \ref{evenform} asserts that $20k+2 \leq h(G)$.

The following construction of a realizing manifold for $h(G)$ contains $k-1$ copies of $T^2 \times \Sigma_3$ as well as one copy of $T^2 \times \Sigma_2$ and one 4-torus. 

As in the proofs of Theorems \ref{theorem 5-cliques string} and \ref{theorem 6-cliques string}, we will see the pattern of necessary identification surgeries with an example. Let $k=3$. Start with $T^4 \# (T^2 \times \Sigma_3) \# (T^2 \times \Sigma_3) \# (T^2 \times \Sigma_2)$. Let $\{x_1,x_2\}$ and $\{y_1,\dotsc,y_6\}$ generate $\pi_1$ of the first copy of $T^2 \times \Sigma_3$, and let $\{s_1,s_2\}$ and $\{t_1,\dotsc,t_6\}$ generate $\pi_1$ of the second. Let $\{z_1,z_2,z_3,z_4\}$ generate $\pi_1(T^4)$ and let $\{u_1,u_2\}$ and $\{v_1,v_2,v_3,v_4\}$ generate $\pi_1(T^2\times \Sigma_2)$. See Figure \ref{pieces of 7} for the graphical representations of each summand of the 4-manifold. 

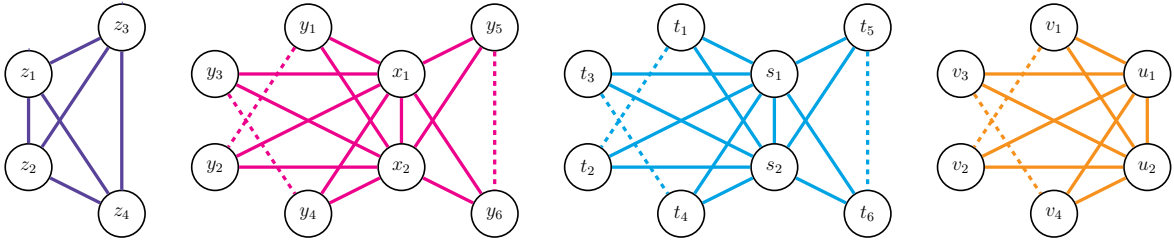
\begin{figure}[h]
\centering
\resizebox{.95\textwidth}{!}{
\begin{tikzpicture}[scale=2]
\tikzstyle{every node}=[draw, shape=circle, minimum size = 1cm, line width = 1pt];
\path (-2,1) node (z1) {$z_1$};
\path (-2,0) node (z2) {$z_2$};
\path (-1,1.5) node (z3) {$z_3$};
\path (-1,-.5) node (z4) {$z_4$};
\path (2,1) node (x1) {$x_1$};
\path (2,0) node (x2) {$x_2$};
\path (1,1.5) node (y1) {$y_1$};
\path (1,-0.5) node (y4) {$y_4$};
\path (0,1) node (y3) {$y_3$};
\path (0,0) node (y2) {$y_2$};
\path (3,1.5) node (y5) {$y_5$};
\path (3,-0.5) node (y6) {$y_6$};
\path (6,1) node (s1) {$s_1$};
\path (6,0) node (s2) {$s_2$};
\path (5,1.5) node (t1) {$t_1$};
\path (5,-0.5) node (t4) {$t_4$};
\path (4,1) node (t3) {$t_3$};
\path (4,0) node (t2) {$t_2$};
\path (7,1.5) node (t5) {$t_5$};
\path (7,-0.5) node (t6) {$t_6$};
\path (8,1) node (v3) {$v_3$};
\path (8,0) node (v2) {$v_2$};
\path (9,1.5) node (v1) {$v_1$};
\path (9,-0.5) node (v4) {$v_4$};
\path (10,1) node (u1) {$u_1$};
\path (10,0) node (u2) {$u_2$};
\foreach \i in {1,2}{\foreach \j in {1,...,6}{\draw[magenta, line width = 2pt] (x\i)--(y\j); \draw[Cerulean, line width = 2pt] (s\i)--(t\j);}} \draw[magenta, line width = 2pt] (x1)--(x2); \draw[Cerulean, line width = 2pt] (s1)--(s2);
\foreach \i in {1,2}{\foreach \j in {1,...,4}{\draw[BurntOrange, line width = 2pt] (u\i)--(v\j); }} \draw[BurntOrange, line width = 2pt] (u1)--(u2);
\draw[magenta, line width = 2 pt, dashed] (y1) -- (y2) (y3) -- (y4) (y5)--(y6);
\draw[Cerulean, line width = 2 pt, dashed] (t1) -- (t2) (t3) -- (t4) (t5)--(t6);
\draw[BurntOrange, line width = 2 pt, dashed] (v1) -- (v2) (v3) -- (v4);
\foreach \i in {1,...,4}{\foreach \j in {\i,...,4}{\draw[line width = 2pt, Violet] (z\i)--(z\j);}}
\end{tikzpicture}
}
\caption{A graph representing $\pi_1$ of the products of surfaces necessary for this 4-manifold construction}
\label{pieces of 7}
\end{figure}

Perform surgeries to induce the following identifications:
\begin{eqnarray*}
y_1 = z_3 & x_1 = t_3 & s_1 = v_3 \\ 
y_2 = z_4 & y_5 = t_1 & t_5 = v_1 \\
y_3 = z_1 & y_6 = t_2 & t_6 = v_2 
\end{eqnarray*}
These surgeries do not change $b_2$. Next perform the following 3 4-reductions:
\begin{equation*}
\textcolor{OliveGreen}{[y_4,z_2,x_1y_1,x_2y_2]},\ \textcolor{Blue}{[x_2,t_4,y_5s_1,y_6,s_2]},\ \textcolor{Red}{[s_2,v_4,t_5u_1,t_6u_2]}
\end{equation*}
These 4-reductions result in a 4-manifold with the correct $\pi_1$. (See Figure \ref{3 7-cliques}.) This example shows the pattern one would use to construct a realizing manifold for any $k$. The copy of $T^2 \times \Sigma_2$ adds 10 to the count of $b_2$, each copy of $T^2 \times \Sigma_3$ adds 14, and the 4-torus adds 6. In addition, $k$ 4-reductions are required and each adds 6 to $b_2$. The resulting manifold $M$ will have 
\begin{equation*}
b_2(M) = 6+14(k-1) + 10+ 6k = 20k+2.
\end{equation*}
\end{proof}

\begin{figure}[h]
\centering
\begin{tikzpicture}[scale = 2] 
\tikzstyle{every node}=[draw, shape=circle, inner sep=.5mm, minimum size=.7cm, line width = 1pt];
\path (51.429:1cm) node (y1) {$y_1$}; \path (4*51.429:1cm) node (y2) {$y_2$}; 
\path (2*51.429:1cm) node (y3) {$y_3$}; \path (5*51.429:1cm) node (y4) {$y_4$};
\path (0:1cm) node (x1) {$x_1$}; \path (-51.429:1cm) node (x2) {$x_2$}; 
\path (3*51.429:1cm) node (z1) {$z_2$};
\begin{scope}[shift={(-25.714:1.8019cm)}]
	\path (231.4305:1cm) node (y6) {$y_6$}; \path (77.1435:1cm) node (y5) {$y_5$};
	\path (51.429+25.714:1cm) node (t1) {}; \path (4*51.429+25.714:1cm) node (t2) {}; 
	\path (2*51.429+25.714:1cm) node (t3) {}; \path (5*51.429+25.714:1cm) node (t4) {$t_4$};
	\path (25.714:1cm) node (s1) {$s_1$}; \path (-25.714:1cm) node (s2) {$s_2$}; 		
		\begin{scope}[shift={(0:1.8019cm)}]
			\path (5*51.429:1cm) node (t6) {$t_6$}; \path (2*51.429:1cm) node (t5) {$t_5$};
			\path (2*51.429:1cm) node (v1) {}; \path (5*51.429:1cm) node (v2) {};
			\path (3*51.429:1cm) node (v3) {}; \path (6*51.429:1cm) node (v4) {$v_4$};
			\path (51.429:1cm) node (u1) {$u_1$}; \path (0:1cm) node (u2) {$u_2$}; 
		\end{scope}			
\end{scope}	
\foreach \i in {1,2}{\foreach \j in {1,...,6}{\draw[magenta, line width = 2pt] (x\i)--(y\j); \draw[Cerulean, line width = 2pt] (s\i)--(t\j);}} \draw[magenta, line width = 2pt] (x1)--(x2); \draw[Cerulean, line width = 2pt] (s1)--(s2);
\foreach \i in {1,2}{\foreach \j in {1,...,4}{\draw[BurntOrange, line width = 2pt] (u\i)--(v\j); }} \draw[BurntOrange, line width = 2pt] (u1)--(u2);
\draw[Violet, line width = 2pt] (z1)--(y2)--(y3)--(z1)--(y1)--(y3) (y2)--(y1);
\draw[OliveGreen, line width = 2pt] (x1)--(z1)--(y4)--(y2) (y1)--(y4) (z1)--(x2) (y1)--(y2);
\draw[Blue, line width = 2pt] (s1)--(x2)--(t4)--(t2) (t1)--(t4) (x2)--(s2) (t1)--(t2);
\draw[Red, line width = 2pt] (u1)--(s2)--(v4)--(v2) (v1)--(v4) (s2)--(u2) (v1)--(v2);
\draw[magenta, line width = 2 pt, dashed] (y1) -- (y2) (y3) -- (y4) (y5)--(y6);
\draw[Cerulean, line width = 2 pt, dashed] (t1) -- (t2) (t3) -- (t4) (t5)--(t6);
\draw[BurntOrange, line width = 2 pt, dashed] (v1) -- (v2) (v3) -- (v4);
\end{tikzpicture}
\caption{A graph representing $\pi_1$ of a realizing manifold for three 7-cliques attached edge-to-edge}
\label{3 7-cliques}
\end{figure}
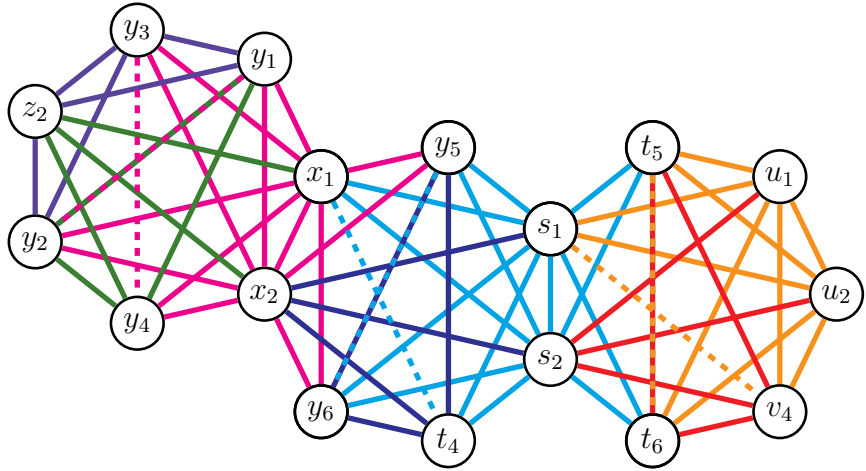

%%%%%%%%%%%%%%%%%%%%%%%%%%%%%%%%%%%%%%%%%%%%%%%%%%%%%%%%%%%%%%%%%%%%%%%%%%%%%%%%%%%%%%%%%%%%%%%%%%%%%%%%%%%%%%%%%%%%
%%%%%%%%%%%%%%%%%%%%%%%%%%%%%%%%%%%%%%%%%%%% CONCLUDING REMARKS %%%%%%%%%%%%%%%%%%%%%%%%%%%%%%%%%%%%%%%%%%%%%%%%%%%%
%%%%%%%%%%%%%%%%%%%%%%%%%%%%%%%%%%%%%%%%%%%%%%%%%%%%%%%%%%%%%%%%%%%%%%%%%%%%%%%%%%%%%%%%%%%%%%%%%%%%%%%%%%%%%%%%%%%%

\section{Concluding Remarks}\label{conclusion}

The author knows no examples of \RAAGs\ that are \emph{not} \cm. We therefore make the following conjecture that is stated previously in the introduction:
\begin{repconj}{Conjecture} 
All \RAAGs\ are \cm. That is, if $G$ is a \RAAG, $h(G) = 2b_2(G)-m_2(G)$.
\end{repconj}

\begin{remark}
This conjecture does not hold for all finitely presented groups. Consider the following counterexample. Let $G= \Z_2 \oplus \Z_2$. A classifying space for $\Z_2 \oplus \Z_2 $ is $\mathbb{RP}^\infty \times \mathbb{RP}^\infty$. Using the Universal Coefficient Theorem, the K\"unneth formula for homology, and the homology of $\mathbb{RP}^\infty$,we see that $b_i(\Z_2\oplus \Z_2)=0$ for $i=1,2$ and 1 for $i=0$. A realizing 4-manifold for $h(\Z_2 \oplus \Z_2)$ is constructed in \cite{KirkLivingston05} from $(L(2,1)\times S^1) \# (S^1 \times S^3)$. Surgery is then performed to identify the generator of $\pi_1(L(2,1))$ and the generator of $\pi_1(S^1\times S^3)$. Let $a$ be the generator of $\pi_1(S^1)$ from $L(2,1)\times S^1$. Surgery is performed on $a^2$, which results in a 4-manifold with the correct $\pi_1$ and $b_2=0$. However, $H^*(\mathbb{RP}^\infty \times \mathbb{RP}^\infty; \Z_2)$ is just the polynomial ring $\Z_2[\alpha,\beta]$. Thus, the form (\ref{formmod2}) must be nondegenerate and so $m_2(\mathbb{RP}^\infty\times \mathbb{RP}^\infty)$ will be positive. Then $2b_2(G) -m_2(G) < h(G) = 0$.

\end{remark}

More generally, the author suspects that the tools described in Section \ref{tools} will be sufficient for all constructions of realizing manifolds for \RAAGs. If true, this would mean that all such realizing manifolds have zero signature, as in the cases of free and free abelian groups \cite{KirkLivingston09}.

\bigskip 

The greatest obstacle in proving this conjecture is in developing a way to generalize current results without using induction. One may expect to find an inductive way to calculate $h$. For example, given any two subgraphs $\Gamma_1$ and $\Gamma_2$ of $\Gamma$, one may expect there is a relationship between $h(G)$ and $h(G_1)$+$h(G_2)$, as is found in the free abelian case \cite[Theorems 8,9]{KirkLivingston05}. 

Kirk and Livingston proved that if 6 divides $mn$, the realizing manifold for $h(\Z^{m+n})$ is constructed from the realizing manifolds for $h(\Z^m)$ and $h(\Z^n)$. Consider a \RAAG\ $G$ that is a quotient of $G_1*G_2$. It is not guaranteed that a realizing manifold for $h(G)$ can be constructed from realizing 4-manifolds for $h(G_i)$, even if the number of added relations necessary to transform $G_1*G_2$ into $G$ is a multiple of 6 (as required in the free abelian case). 

\begin{numbered example}
Suppose we have two disjoint graphs $\Gamma_1$ and $\Gamma_2$, each of which are 4-cliques. Denote the associated \RAAGs\ by $G_1$ and $G_2$. Consider the following graph $\Gamma$ in Figure \ref{boxes} associated to a quotient of $G_1*G_2$, which we will denote by $G$. 
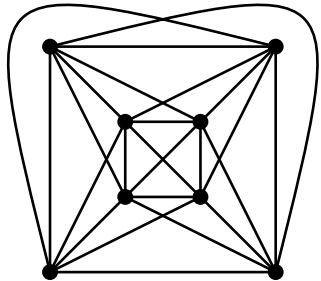
\begin{figure}[h]
\centering
\begin{tikzpicture}
\foreach \i in {1,2,3,4}{
	\path (\i,\i) coordinate (y\i);
	\path (\i,5-\i) coordinate (x\i);
	}
\foreach \i in {1,2,3,4}{\foreach \j in {1,...,4}{\draw[line width = 1pt] (x\i)--(y\j);}}	
\draw[line width =1pt] (x1)--(x4) (y1)--(y4);	
\draw[line width=1pt] (x1) .. controls (5,5) .. (x4) (y1) .. controls (0,5) .. (y4);	
\foreach \i in {1,2,3,4}{\fill (x\i) circle (3pt); \fill (y\i) circle (3pt);}
\end{tikzpicture}
\caption{A dimension 4 graph with 8 vertices and 24 edges, 4 edges short of an 8-clique}
\label{boxes}
\end{figure}

Denote the realizing manifolds for $h(G_i)$ by $M_i$. Each is a 4-torus. Let $\pi_1(M_1)$ be generated by $\{x_1,x_2,x_3,x_4\}$ and $\pi_1(M_2)$ be generated by $\{y_1,y_2,y_3,y_4\}$. Figure \ref{boxes construction} (a) shows $\Gamma_1 * \Gamma_2$ as a subgraph of $\Gamma$, with 12 extra edges. Even though 6 divides 12, it is impossible to construct a realizing manifold for $G$ with 2 4-reductions. This is most easily explained by the fact that $b_2(G) = 24$ and $m_2(G)=22$, so $26 \leq h(G)$. Two 4-reductions performed on $M_1 \# M_2$ would result in a 4-manifold with $b_2(M)=24$ and therefore cannot possibly be in $\mathcal M(G)$ without additional commutator surgeries to yield the correct $\pi_1$. Moreover, it is impossible to construct a realizing 4-manifold from $M_1\#M_2$. Any 4-reductions or commutator surgeries performed on $M_1\#M_2$ will result in a 4-manifold with $b_2 > 26$. 

An actual realizing 4-manifold for $h(G)$ can be constructed by taking the connected sum of 8 copies of $S^1\times S^3$, with $\pi_1$ generated by $\{a,b,c,d,e,f,g,h\}$. Perform the 4-reductions $[a,b,e,f]$, $[a,d,g,h]$, $[b,c,f,g]$, and $[c,d,e,fh]$ and surgery on the commutator $[b,h]$. This yields a 4-manifold $M$ with $\pi_1(M)=G$ and $b_2(M) = 26$. Figure \ref{boxes construction} (b) shows $\pi_1$ of the realizing manifold construction. 

\begin{figure}[h]
\centering
\begin{minipage}{.49\linewidth} \centering
\begin{tikzpicture}[scale=1.5] \centering
\tikzstyle{every node}=[draw, shape=circle, inner sep=.5mm, minimum width = .7cm, line width=1pt];
\path (1,4) node (a) {$x_1$}; \path (2,3) node (b) {$y_1$}; \path (3,2) node (c) {$y_4$}; \path (4,1) node (d) {$x_4$}; \path (1,1) node (e) {$x_2$}; \path (2,2) node (f) {$y_2$}; \path (3,3) node (g) {$y_3$}; \path (4,4) node (h) {$x_3$};
\draw(a)--(b)--(e)--(c)--(d)--(g)--(a)--(f)--(d) (e)--(f) (g)--(h)--(b) (h)--(c);	
\draw[line width=2pt, OliveGreen] (a) .. controls (5,5) .. (d) (e) .. controls (0,5) .. (h);	
\draw[line width=2pt, OliveGreen] (a)--(e)--(d)--(h)--(a);	
\draw[line width = 2pt, Red] (b)--(g)--(c)--(f)--(b)--(c) (f)--(g);
\end{tikzpicture} \\
(a)
\end{minipage}
\begin{minipage}{.49\linewidth} \centering
\begin{tikzpicture}[scale=1.5] \centering
\tikzstyle{every node}=[draw, shape=circle, inner sep=.5mm, minimum width = .7cm, line width=1pt];
\path (1,4) node (a) {$a$}; \path (2,3) node (b) {$b$}; \path (3,2) node (c) {$c$}; \path (4,1) node (d) {$d$};
\path (1,1) node (e) {$e$}; \path (2,2) node (f) {$f$}; \path (3,3) node (g) {$g$}; \path (4,4) node (h) {$h$};
\draw[line width = 2pt, OliveGreen] (a)--(e)--(f)--(b)--(a)--(f) (e)--(b);
\draw[line width = 2pt, Cyan] (b)--(g)--(f)--(c)--(g)--(f)--(b)--(c);
\draw[line width = 2pt, Blue] (g)--(a)--(h)--(g)--(d)--(h) (a) .. controls (5,5) .. (d);
\draw[line width =2pt, Magenta] (c)--(d)--(e)--(c);	
\draw[line width=2pt, Magenta, dashed] (e)--(f)--(d)--(h) (c)--(h) (f)--(c)--(e) (e) .. controls (0,5) .. (h);	
\draw[line width=2pt, BurntOrange] (b)--(h);
\end{tikzpicture} \\
(b)
\end{minipage}
\caption{(a) Subgraphs \textcolor{OliveGreen}{$\Gamma_1$} and \textcolor{Red}{$\Gamma_2$} in $\Gamma$ and (b) $\pi_1$ of a realizing 4-manifold construction of $\Gamma$}
\label{boxes construction}
\end{figure}
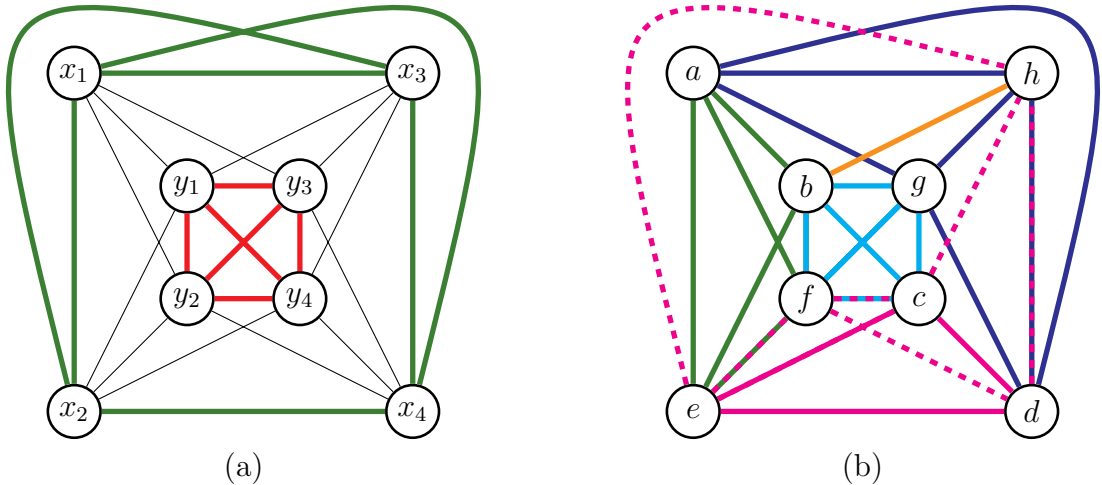
\end{numbered example}

\bigskip

Theorems \ref{freeproducts} and \ref{wedging} provide a beginning towards finding an inductive process of computing $h$ for any \RAAG. Recall that for \cm\ \RAAGs, Theorem \ref{freeproducts} asserts $h(G_1*G_2) =h(G_1)+h(G_2)$. If $G$ is created by the identification of pairwise non-commuting generators from $G_1$ and $G_2$, Theorem \ref{wedging} asserts $h(G) = h(G_1)+h(G_2)$. Logically, the next step is to find a relationship between $h(G)$ and $h(G_1) + h(G_2)$ if $G$ is formed by the identification of generators that \emph{do} commute. This is the situation in which $\Gamma$, a graph associated to $G$, is created by the joining of two graphs $\Gamma_1$ and $\Gamma_2$ along an edge or many edges. Unfortunately, combining graphs along edges yields complications in calculating $h$ because (\ref{formmod2}) does not necessarily split as a direct sum, and so $m_2(G)$ may not equal $m_2(G_1) + m_2(G_2)$.

In general, it is difficult to have inductive results involving graphs. It is not clear whether we should induct on vertices or edges. Because vertices alone correspond only to $b_1$, adding vertices without adding edges changes nothing in terms of computing $h$. However, adding edges can change the value of $h$ drastically.

If the added edge does not form a new 4-clique in the graph, we know from Theorem \ref{breakdown} that this increases $h$ by 2. However, if adding an edge creates additional 4-cliques in the graph, the change in $h$ depends on the structure of the graph. In fact, adding one edge in the graph may result in an entirely different construction of a new realizing manifold.

In light of this difficulty, the only known examples of \cm\ \RAAGs\ belong to infinite families of graphs in which induction on patterns allows us to calculate $h$ for all groups in the family. Beyond finding new families of graphs, however, it is unclear how to proceed in proving this conjecture. 

%%%%%%%%%%%%%%%%%%%%%%%%%%%%%%%%%%%%%%%%%%%% APPENDIX %%%%%%%%%%%%%%%%%%%%%%%%%%%%%%%%%%%%%%%%%%%%%%%%%%%%%%%%%%%%%%
\appendix

%%%%%%%%%%%%%%%%%%%%%%%%%%%%%%%%%%%%%%%%%%%%%%%%%%%%%%%%%%%%%%%%%%%%%%%%%%%%%%%%%%%%%%%%%%%%%%%%%%%%%%%%%%%%%%%%%%%%
%%%%%%%%%%%%%%%%%%%%%%%%%%%%%%%%%%%%%%%%%%%% SAGE CODE %%%%%%%%%%%%%%%%%%%%%%%%%%%%%%%%%%%%%%%%%%%%%%%%%%%%%%%%%%%%%
%%%%%%%%%%%%%%%%%%%%%%%%%%%%%%%%%%%%%%%%%%%%%%%%%%%%%%%%%%%%%%%%%%%%%%%%%%%%%%%%%%%%%%%%%%%%%%%%%%%%%%%%%%%%%%%%%%%%
\section{Sage code to compute $\lowercase{m}_2(G)$} \label{sage code}

The following is the Sage code used to compute $m_2(G)$. All that is necessary as an input is the adjacency matrix of the graph associated to $G$. 

\medskip

\begin{Verbatim}[frame=single]
# This function takes an adjacency matrix A and returns the number of edges 
# of the graph.

def count(A):
    n=A.nrows()
    c = 0 # c will represent the number of 1's in the matrix
    for i in range(n):
        for j in [i+1..n]:
            if A[i,j-1] == 1:
                c = c+1
    return c
\end{Verbatim}

\medskip

\begin{Verbatim}[frame=single]
# This function is called by create_matrix() and returns the correct row 
# and column of A so that the entry into the matrix D is represented by  
# the correct generators of the form (4.2).

def find_row(A,k,l):  
    n = A.nrows() 
    for i in range(n): 
        if A[i,0]==k and A[i,1] == l:
            return i+1
\end{Verbatim}

\medskip

\begin{Verbatim}[frame=single]
# This function generates the matrix representing the form (2) from  
# the adjacency matrix of the graph, B.

def create_matrix(B): # B is the input adjacency matrix of the graph
    n = B.ncols() # better be square!
    m=0 
    edges = matrix(count(B),2) # edges is a matrix that stores all the  
                               # edges of the graph
    for i in [1..n]:
        for j in [i+1..n]:
            if B[i-1,j-1]==1:
                edges[m,0]=i
                edges[m,1]=j
                m=m+1
    D = matrix(count(B)) # D will be the matrix outputted by the function
    p=1
    for i in [1..n]:
       for j in [i+1..n]:
          if B[i-1,j-1]==1:
            for k in [j+1..n]:
              if B[i-1,k-1]==1 and B[j-1,k-1]==1:
                for l in [k+1..n]:
                   if B[i-1,l-1]==1 and B[j-1,l-1]==1 and B[k-1,l-1]==1:
                     D[find_row(edges,i,j)-1,find_row(edges,k,l)-1] = p
                     D[find_row(edges,k,l)-1,find_row(edges,i,j)-1] = p
                     D[find_row(edges,i,k)-1,find_row(edges,j,l)-1] = -p
                     D[find_row(edges,j,l)-1,find_row(edges,i,k)-1] = -p
                     D[find_row(edges,i,l)-1,find_row(edges,j,k)-1] = p
                     D[find_row(edges,j,k)-1,find_row(edges,i,l)-1] = p
                     p=p+1
    return D
\end{Verbatim}    

\medskip

\begin{Verbatim}[frame=single]
# This function takes a matrix with entires in {0,+/- 1,...,+/- n} and a 
# list [x_1,...,x_n] and replaces all +/- i with x_i.

def substitute(M,list):
    n = len(list)
    d = M.ncols() # better be square!
    newM = matrix(d)
    for i in range(d):
        for j in range(d):
            if M[i,j] != 0:
                newM[i,j] = list[abs(M[i,j])-1]
    return newM
\end{Verbatim}

\medskip

\begin{Verbatim}[frame=single]
# This function returns `True' if n is even, `False' if n is odd.

def is_even(n):
    return n%2 == 0
\end{Verbatim}

\medskip

\begin{Verbatim}[frame=single]
# This function returns a list of the 2^n possible base 2 representations
# of the numbers 0,1,...2^{n-1}, to be substituted later into the matrix 
# with numbers 1,...,n.

def possible_combinations(n):
    the_list = []
    for i in range(2^n):
        temp_list = ZZ(i).digits(2) # ZZ(i).digits(2) writes the integer i 
                                    # in binary (mod2)
        m = len(temp_list) # I have to add the trailing zeros...
        for j in [m..n-1]:
            temp_list.append(0)
        the_list.append(temp_list)
    return the_list
\end{Verbatim}

\medskip
    
\begin{Verbatim}[frame=single]
# This function takes a generic matrix whose entries are in 0,1,...,n
# and it will compute the maximum rank of this modulo mod (0 means over ZZ)
# (which must be a prime, or else it will break!)
# with the max being taken over possible assignments of {0,1}
# to the elements 1,...,n in the matrix M-generic.

def max_rank(M_generic, n=0, mod=0):
    if n == 0:
        n = M_generic.height() # OMG there is a function that finds the 
                               # maximimum integer!
    m = M_generic.nrows()       
    if is_even(m) == False:
        m = m-1                            
    rank_list = []
    S = ZZ
    if mod > 0:
        S = GF(mod)
    combos = possible_combinations(n)
    for c in combos:
        M = substitute(M_generic,c)
        M = matrix(S,M)
        rank_list.append(M.rank())
        if max(rank_list) == m:  # Break in the loop if we find highest 
                                 # possible rank
            break          
    return max(rank_list)
\end{Verbatim}

\medskip

\begin{Verbatim}[frame=single]
# Example: A is a matrix representing the adjacency graph of two 4-cliques 
# attached along a 3-clique.

A = matrix([
[0,1,1,1,0],
[1,0,1,1,1],
[1,1,0,1,1],
[1,1,1,0,1],
[0,1,1,1,0]])

import time
t = time.time()
print(max_rank(create_matrix(A),0,mod=2))
elapsed = time.time() - t; elapsed
\end{Verbatim}

\textcolor{Blue}{6} 

\textcolor{Blue}{0.007825136184692383}

\bigskip 

\noindent This next function is similar to \verb|max_rank()| however it also returns the coefficients of the $\alpha_i$ if the for loop breaks; that is, if $b_2(G)$ is odd and there exists $\alpha$ such that $m_2(G) = b_2(G) -1$, it returns the coefficients of $\alpha$ as well as the rank. 

\medskip

\begin{Verbatim}[frame=single]
def max_rank_alpha(M_generic, n=0, mod = 0):
    if n == 0:
        n = M_generic.height() # OMG there is a function that finds the 
                               # maximimum integer!
    m = M_generic.nrows()  
    if is_even(m) == False:
        m = m-1                            
    rank_list = []
    S = ZZ
    if mod > 0:
        S = GF(mod)
    combos = possible_combinations(n)
    for c in combos:
        M = substitute(M_generic,c)
        M = matrix(S,M)
        rank_list.append(M.rank())
        #print (M, M.rank(),c)
        if max(rank_list) == m:
            print c
            break                  
    return max(rank_list)
\end{Verbatim}

\providecommand{\bysame}{\leavevmode\hbox to3em{\hrulefill}\thinspace}
\providecommand{\MR}{\relax\ifhmode\unskip\space\fi MR }
\providecommand{\MRhref}[2]{%
  \href{http://www.ams.org/mathscinet-getitem?mr=#1}{#2}
}
\providecommand{\href}[2]{#2}

\end{document}